\documentclass[a4paper,11pt]{amsart}

\usepackage[utf8]{inputenc}
\usepackage[T1]{fontenc}
\usepackage[margin=1in]{geometry}

\usepackage{amsmath}
\usepackage{amssymb}
\usepackage{amsthm}

\usepackage{eucal}
\usepackage{enumitem}
\usepackage{tikz-cd}
\usepackage{tikz}
\usetikzlibrary{arrows, decorations.pathmorphing,positioning}
\usepackage{float}

\usepackage{xcolor}
\usepackage{hyperref}

\theoremstyle{plain}
\newtheorem{mainthm}{Theorem}

\newtheorem{thm}{Theorem}[section]
\newtheorem{prop}[thm]{Proposition}
\newtheorem{cor}[thm]{Corollary}
\newtheorem{lemma}[thm]{Lemma}

\theoremstyle{definition}
\newtheorem{dfn}[thm]{Definition}
\newtheorem{rmk}[thm]{Remark}
\newtheorem{construction}[thm]{Construction}
\newtheorem{ex}[thm]{Example}
\newcommand{\Acal}{\mathcal{A}} \newcommand{\Bcal}{\mathcal{B}} \newcommand{\Ccal}{\mathcal{C}}
\newcommand{\Dcal}{\mathcal{D}} \newcommand{\Ecal}{\mathcal{E}} \newcommand{\Fcal}{\mathcal{F}}
\newcommand{\Gcal}{\mathcal{G}} \newcommand{\Hcal}{\mathcal{H}} 
 \newcommand{\Kcal}{\mathcal{K}} \newcommand{\Lcal}{\mathcal{L}}
  \newcommand{\Ocal}{\mathcal{O}}
 \newcommand{\Qcal}{\mathcal{Q}} \newcommand{\Rcal}{\mathcal{R}}
\newcommand{\Scal}{\mathcal{S}} \newcommand{\Tcal}{\mathcal{T}} \newcommand{\Ucal}{\mathcal{U}}
\newcommand{\Vcal}{\mathcal{V}}  \newcommand{\Xcal}{\mathcal{X}}
\newcommand{\Ycal}{\mathcal{Y}} \newcommand{\Zcal}{\mathcal{Z}}

\newcommand{\Abb}{\mathbb{A}}  
  \newcommand{\Fbb}{\mathbb{F}}
  
  \newcommand{\Lbb}{\mathbb{L}}
  
 \newcommand{\Qbb}{\mathbb{Q}} \newcommand{\Rbb}{\mathbb{R}}

 \newcommand{\Zbb}{\mathbb{Z}}

\newcommand{\Dsf}{\mathsf{D}}  \newcommand{\Fsf}{\mathsf{F}}
  
  \newcommand{\Lsf}{\mathsf{L}}
  
  \newcommand{\Rsf}{\mathsf{R}}
 \newcommand{\Tsf}{\mathsf{T}}

\renewcommand{\phi}{\varphi}
\renewcommand{\epsilon}{\varepsilon}
\renewcommand{\theta}{\vartheta}
\newcommand{\cpl}{\mathsf{c}}
\DeclareMathOperator{\coker}{coker}
\DeclareMathOperator{\Hom}{\mathsf{Hom}}
\DeclareMathOperator{\RHom}{\Rbb\!\Hom}
\DeclareMathOperator{\End}{\mathsf{End}}
\DeclareMathOperator{\Inj}{\mathsf{Inj}}
\DeclareMathOperator{\Prod}{\mathsf{Prod}}
\DeclareMathOperator{\add}{\mathsf{add}}
\DeclareMathOperator{\Cogen}{\mathsf{Cogen}}
\DeclareMathOperator{\Ind}{\mathsf{Ind}}
\DeclareMathOperator{\GSupp}{\mathsf{Gsupp}}
\DeclareMathOperator{\GSpec}{\mathsf{GSpec}}
\DeclareMathOperator{\Gdim}{\mathsf{Gdim}}
\DeclareMathOperator{\height}{\mathsf{height}}
\DeclareMathOperator{\Loc}{\mathsf{Loc}}
\DeclareMathOperator{\LOC}{\mathcal{L}}
\DeclareMathOperator{\hocolim}{\mathsf{hocolim}}
\DeclareMathOperator{\spcl}{\mathsf{V}}
\DeclareMathOperator{\gncl}{\mathsf{\Lambda}}
\DeclareMathOperator{\Ext}{\mathsf{Ext}}

\DeclareMathOperator{\D}{\mathsf{D}}
\DeclareMathOperator{\Mod}{\mathsf{Mod}}

\DeclareMathOperator{\Spec}{\mathsf{Spec}}

\DeclareMathOperator{\supp}{\mathsf{supp}}
\newcommand{\mf}{\mathfrak{m}}
\newcommand{\pf}{\mathfrak{p}}
\newcommand{\qf}{\mathfrak{q}}
\newcommand{\of}{\mathfrak{o}}
\DeclareMathOperator{\Red}{\mathsf{Red}}
\begin{document}

\pagestyle{headings}

\title{Mutation and the Gabriel Spectrum}
\author{Michal Hrbek}
\address{Michal Hrbek, Institute of Mathematics of
the Czech Academy of Sciences, \v{Z}itn\'a 25, 115 67 Prague, Czech Republic.}
\email{hrbek@math.cas.cz}
\author{Sergio Pavon}
\address{Sergio Pavon, Dipartimento di Matematica ``Tullio Levi-Civita'', Universit\`a degli Studi di Padova, via Trieste 63, 35121 Padova, Italy}
\email{sergio.pavon@math.unipd.it}
\author{Jorge Vit\'oria}
\address{Jorge Vit\'oria, Dipartimento di Matematica ``Tullio Levi-Civita'', Universit\`a degli Studi di Padova, via Trieste 63, 35121 Padova, Italy}
\email{jorge.vitoria@unipd.it}

%\keywords{Gabriel spectrum, t-structure, mutation, commutative noetherian ring}

\subjclass[2020]{18E10, 13D09, 18G80, 16E35}
\thanks{Acknowledgements: The authors would like to thank Jan Šťovíček for reading a preliminary version
of this paper and for his useful comments concerning the proof of
Proposition~\ref{prop:W-giraud-finite}. They would also like to thank
Rosanna Laking for several discussions on the topics of the paper.
M.H.~was supported by the GAČR project 23-05148S and the Czech Academy of Sciences (RVO 67985840).
S.P.~and J.V.~were supported
by NextGenerationEU under NRRP, Call PRIN 2022  No.~104
of February 2, 2022 of Italian Ministry of University and Research; Project
2022S97PMY \textit{Structures for Quivers, Algebras and Representations
(SQUARE)}, and also by the Department of Mathematics `Tullio Levi-Civita' of the University of Padova through its BIRD - Budget
Integrato per la Ricerca dei Dipartimenti 2022, within the project
\textit{Representations of quivers with commutative coefficients}.}
%\maketitle
%\tableofcontents
\begin{abstract}
	Mutations occur in multiple algebraic contexts, often enjoying good
	combinatorial properties. In this paper we study mutations of pure-injective
	cosilting objects in compactly generated triangulated categories from a
	topological point of view. We consider the topologies studied by Gabriel,
	Burke and Prest on the set of indecomposable injective objects in a
	Grothendieck abelian category, transfer them to associated cosilting
	subcategories, and show that, in that context, right mutation induces a
	homeomorphism on two complementary subspaces. We then improve this result in
	the context of the derived category of a commutative noetherian ring, showing
	that right mutation is an open bijection. We end the paper with a detailed
	analysis of a range of cosilting subcategories over commutative noetherian rings
	for which the topology is completely known. As a byproduct of this analysis, we obtain that the category of
modules over a commutative noetherian ring is the unique locally noetherian
Grothendieck category in its derived-equivalence class.
\end{abstract}
\maketitle
%\setcounter{tocdepth}{1}
%\tableofcontents

\section*{Introduction}

A common mantra in mathematics is that a bijection offers two perspectives on
essentially the same set. In representation theory, bijections are bridges
between homological, combinatorial or geometric data regarding the objects of
study. The more structure these bridges have, the better vehicles of information
they are. For example, a good enough bijection between two sets that are
naturally endowed with topologies ought to enjoy, to merit the adjective, good
topological properties. This paper is dedicated to the study of one such bridge:
cosilting mutation, as developed in \cite{ange-laki-stov-vito-22}. Our main aim
is to show that mutation is not only a bijection between sets, but a local
homeomorphism between naturally occurring topologies --- and to explore the
consequences of that fact. 

Mutations occur in multiple algebraic contexts. They are processes that modify
an object of study into another of the same kind by replacing part of it.
Cosilting mutation occurs naturally in compactly generated triangulated
categories and acts on a special class of objects therein: pure-injective
cosilting objects. This class is important because it parametrises algebraically
well-behaved cohomology theories in the triangulated category, namely those
arising from smashing $t$-structures with Grothendieck hearts
\cite{nico-saor-zvon-19}. The mutation process for these objects is a
generalisation of the classic notion in tilting theory (see for example \cite{ried-scho-91}).

The set of indecomposable injective objects in a Grothendieck abelian category
is naturally endowed with a topology that generalises the one on indecomposable
injective modules over a commutative noetherian ring (the Hochster dual to the
Zariski topology). This topology, introduced by Gabriel \cite{gabr-62} and
studied in detail by Burke and Prest (see, for example, \cite{burk-94-thesis, pres-09}), is defined in terms of localising subcategories, thus
encoding some fundamental structure of the Grothendieck category. For any
pure-injective cosilting object $c$ in a compactly generated triangulated
category, there is a bijection between indecomposable summands of products of
$c$, denoted by $\GSpec(c)$, and indecomposable injectives in the
associated Grothendieck heart \cite{ange-mark-vito-17}. This naturally endows
$\GSpec(c)$ with the corresponding Gabriel topology. Now, as shown in
\cite{ange-laki-stov-vito-22}, if $c'$ is a (right) cosilting mutation of $c$,
then there is a bijection $\GSpec(c)\rightarrow \GSpec(c')$, and
we will study the topological properties of this bijection. 

Recall that a Grothendieck category is said to be semi-artinian if it is
generated, as a localising subcategory, by its simple objects. A Grothendieck
category $\Hcal$ is said to be semi-noetherian if it admits a filtration by
localising subcategories with semi-artinian factors.
Our first main theorem reads as follows.

\begin{mainthm}\label{mainthm:A}
	Let $c$ and $c'$ be pure-injective cosilting objects in a compactly generated
	triangulated category $\Dcal$, with associated Grothendieck hearts
	$\Hcal$ and $\Hcal'$, respectively. Suppose that $c'$ is a right
	mutation of $c$ at $\Ecal=\Prod(\Ecal)\subseteq \Prod(c)$. Then, the following assertions hold.
	\begin{enumerate}
		\item \emph{(Corollary~\ref{cor:mutation-semi-noetherianity})}
			$\Hcal$ is semi-noetherian if and only if so is $\Hcal'$.
		\item \emph{(Corollary~\ref{cor:semi-artinian-mutation})}
			If $\Hcal$ is
			semi-artinian, then so is $\Hcal'$.
		\item \emph{(Theorem~\ref{thm:piecewise-homeomorphism})} The bijection $\Theta\colon \GSpec(c)\to
		\GSpec(c')$ induced by right mutation induces homeomorphisms in the
			subspace topologies on $\Ecal\cap \GSpec(c)$ and on
			$\GSpec(c)\setminus \Ecal$.
		\item \emph{(Corollaries~\ref{cor:discrete-mutation}
			and~\ref{cor:strongly-perfect})} Under one of the assumptions below, the topology of $\GSpec(c')$ is the
			finest one for which the bijection $\Theta$ satisfies (3) and
			$\GSpec(c')\cap\Ecal$ is open:
			\begin{enumerate}[label=(\roman*)]
				\item the Serre quotient $\Hcal/(\Hcal\cap {}^\perp\Ecal)$ is
					semi-artinian; or
				\item $c$ is cotilting and $\Hcal\cap {}^\perp\Ecal$ is a perfect
					hereditary torsion class in $\Hcal$.
			\end{enumerate}
\end{enumerate}
\end{mainthm}

Some readers might be familiar with another topology on such sets: the Ziegler
topology. This topology, defined on indecomposable injective objects in
locally coherent categories, is a fundamental tool in the theory of
purity, and it has roots in the model theory of modules. The Gabriel and the Ziegler
topologies coincide for locally noetherian categories, but these rarely occur as
hearts in our main class of study (the derived category of a commutative
noetherian ring), see Theorem~\ref{mainthm:C} below.
One reason for our choice of working with the Gabriel
topology is that it is defined for arbitrary
Grothendieck categories, regardless of local coherence.

\medskip

Our main applications of this topologically enriched mutation lie in the study
of hearts of bounded cosilting complexes in the derived category of modules over
a commutative noetherian ring. These bounded complexes are automatically
pure-injective \cite{mark-vito-18}, and the $t$-structures associated to them
are known to be compactly generated \cite{hrbe-naka-21}. Since compactly
generated $t$-structures have long been classified (see \cite{alon-jere-saor-10})
in terms of descending sequences of open subsets of $\mathsf{Spec}(R)$ (with
respect to the Hochster dual topology to the Zariski topology), we can also use
commutative algebra to study them. This allows us to strengthen
Theorem~\ref{mainthm:A} in this setup, proving that in fact right mutation
yields an open bijection (or, in other words, that left mutation is continuous).
The following theorem summarises our main results in this setting.
We remark that item (1) is a generalisation of the classical Matlis'
correspondence between prime ideals and indecomposable injective $R$-modules
(see \S\ref{subsubsec:GSpec-Spec}).

\begin{mainthm}\label{mainthm:B}
	Let $R$ be a commutative noetherian ring, and consider $\Spec(R)$ to
	be endowed with the Hochster dual topology to the Zariski topology. Let $c$ be
	a bounded cosilting object in the derived category $\D(R)$ of
	$R$-modules. Then the following statements hold.
	\begin{enumerate}
		\item \emph{(Theorem~\ref{thm:c(p)-support})} There is a canonically defined
			open bijection $\Phi_c\colon \Spec(R)\to \GSpec(c)$.
		\item \emph{(Corollaries~\ref{cor:T0-alexandrov}
			and~\ref{cor:topological-GSpec(c)})} The space $\GSpec(c)$ is Kolmogorov,
			Alexandrov, noetherian and sober.
		\item \emph{(Theorem~\ref{thm:theta-open})} If $c'$ is a right mutation of $c$ then we have that
			$\Theta=\Phi_{c'}\circ \Phi_c^{-1}$ and $\Theta$ is an open map.
	\end{enumerate}
\end{mainthm}

Surprisingly, our topological
considerations on the Gabriel spectra of the bounded
cosilting complexes of $\D(R)$ lead us to a result
concerning derived equivalences.
Indeed, we show that
$\Mod(R)$ is the only locally noetherian abelian category in its
derived-equivalence class.
This is reminiscent of the
Bondal--Orlov reconstruction theorem for smooth projective varieties with ample
or anti-ample canonical bundle \cite{bond-orlo-01}.

\begin{mainthm}[Theorem~\ref{thm:locally-noetherian-rare}]\label{mainthm:C}
	If $\Gcal$ is a locally noetherian abelian category such
	that there is a triangle equivalence $\D^{\mathsf{b}}(\Gcal) \simeq
	\D^{\mathsf{b}}(\Mod(R))$ of the bounded derived categories, then $\Gcal \simeq
	\Mod(R)$.
\end{mainthm}

The techniques which we employ are robust enough to compute the Gabriel spectra
of various examples. Our main tool to describe the
Gabriel spectra of two-term cosilting complexes is
Theorem~\ref{thm:onestep-recipe}. We observe that it relies on the ability to
detect the \emph{coherence} of certain subsets of prime ideals, which is
something that depends on the ring $R$, and not just on its prime spectrum
(see Example~\ref{example:H1-twodim}).

\subsection*{Structure of the paper}

The paper is structured in four sections. In \textbf{Section~\ref{sec:gabriel}}
we recall the necessary material on the localisation theory of Grothendieck
categories, the Gabriel filtration, the Gabriel spectrum. Many of the results
from this section are classical, and we collect them for use
in the rest of the paper.
In \textbf{Section~\ref{sec:general}} we study the Gabriel spectra of the hearts of
$t$-structures associated to pure-injective cosilting objects, in the general
context of a compactly generated triangulated category. In particular, we
investigate how the operation of cosilting mutation affects the topology on the
Gabriel spectrum.
In \textbf{Section~\ref{sec:commutative-noetherian}} we specialise to the derived category of a commutative
noetherian ring. In this context, with the help of the well-developed support
theory of commutative algebra, we are able to strengthen the results of
Section~\ref{sec:general} and make them more explicit.
Finally, in \textbf{Section~\ref{sec:concrete-computations}} we apply the findings of the previous sections to
compute the Gabriel spectra of several examples of bounded cosilting complexes.

\subsection*{Notation and conventions}

In a category $\Acal$, we use lowercase letters to denote objects.
We denote the $\Hom$-bifunctor by $\Acal(-,-)$; if
$\Acal$ is abelian, we will moreover write $\Acal^i(-,-)$ for the $i$-th
$\Ext$-functor. An exception to this convention are the category $\Mod(R)$ of right
$R$-modules over a ring $R$ and its derived category $\D(R)$, where we use the
classical notations $\Hom_R(-,-)$, $\Ext^i_R(-,-)$ and $\Hom_{\D(R)}(-,-)$.

The word \textit{subcategory} will always refer to a strict and full
subcategory.
We fix a few notations for subcategories constructed from a
class of objects $\Xcal$ of $\Acal$. Whenever $\Xcal=\{x\}$ is a singleton, we
write $x$ instead of $\Xcal$.
First, we will often consider the smallest subcategory of $\Acal$ containing
$\Xcal$ and closed under certain operations. So we
denote by $\Prod(\Xcal)$ the smallest subcategory of $\Acal$ containing $\Xcal$
and closed under taking direct summands and existing products; it consists of
the direct summands of existing products of objects of $\Xcal$. Similarly,
$\add(\Xcal)$ is the category of direct summands of finite coproducts of objects
of $\Xcal$, which is the smallest closed under taking direct summands and finite
coproducts which contains $\Xcal$. We denote by
$\varinjlim\Xcal$ the smallest subcategory closed under taking existing direct
limits and containing $\Xcal$, which includes the
existing direct limits of direct systems of objects of $\Xcal$.
We denote by $\Cogen(\Xcal)$ the subcategory consisting of subobjects of
existing products of objects of $\Xcal$, which is the smallest subcategory
closed under products and subobjects which contains $\Xcal$.
Second, we will consider subcategories defined by orthogonality conditions
with respect to the class $\Xcal$. We write:
\[\Xcal^\bot:=\{\,a\in\Acal\mid \Acal(x,a)=0\;\;\forall x\in\Xcal\,\},\quad
{}^\bot\Xcal:=\{\,a\in\Acal\mid \Acal(a,x)=0\;\;\forall x\in\Xcal\,\}.\]

If $\Acal$ is abelian, for any set of non-negative integers
$I\subseteq\Zbb_{\geq0}$ we write:
\[\Xcal^{\bot_I}:=\{\,a\in\Acal\mid \Acal^i(x,a)=0\;\;\forall x\in\Xcal,
i\in I\,\},\quad
{}^{\bot_I}\Xcal:=\{\,a\in\Acal\mid \Acal^i(a,x)=0\;\;\forall x\in\Xcal, i\in I\,\},\]
with the convention that $\Acal^0(-,-):=\Acal(-,-)$.
Similarly, if $\Acal$ is triangulated, for any set of integers $I\subseteq \Zbb$
we write:
\[\Xcal^{\bot_I}:=\{\,a\in\Acal\mid \Acal(x,a[i])=0\;\;\forall x\in\Xcal,
i\in I\,\},\quad
{}^{\bot_I}\Xcal:=\{\,a\in\Acal\mid \Acal(a,x[i])=0\;\;\forall x\in\Xcal, i\in I\,\}.\]
We often list the elements of $I$ rather then writing the set, as in
$\Xcal^{\bot_{0,1}}:=\Xcal^{\bot_{\{0,1\}}}$, or we describe them in an obvious
way, as in $\Xcal^{\bot_{>0}}:=\Xcal^{\bot_{\{i>0\}}}$.
Next, if $\Xcal$ and $\Ycal$ are subcategories of an abelian category $\Acal$,
we write:
\[ \Xcal\ast\Ycal:=\{\,a\in\Acal\mid \exists\;0\to x\to a\to y\to 0\text{ short
exact sequence with }x\in\Xcal, y\in\Ycal\,\}.\]
If $\Acal$ is triangulated, we use the same notation, using distinguished
triangles instead of short exact sequences.
Lastly, we denote by $\Ind(\Xcal)$ the subcategory of indecomposable objects of
$\Acal$
belonging to $\Xcal$. If $\Acal$ is abelian, we denote by $\Inj(\Acal)$ the
subcategory of injective objects of $\Acal$.

When working with a subset $S$ of a set $X$, we will denote by
$S^\cpl:=X\setminus S$ its complement.

\section{The Gabriel spectrum of a Grothendieck category}
\label{sec:gabriel}

\subsection{Hereditary torsion theory and localisation}
\label{sub:localising}

We start by recalling some notions about hereditary torsion pairs and the
associated localisation theory of Grothendieck categories.

Let $\Gcal$ be a Grothendieck category. Recall that a subcategory
$\Lcal\subseteq\Gcal$ is \emph{localising} if it is closed under subobjects,
extensions, quotients and coproducts (or equivalently, colimits).

If $\Lcal$ is such a subcategory, there is a category $\Gcal/\Lcal$, called the
\emph{quotient} of $\Gcal$ over $\Lcal$, together with an exact \emph{quotient
functor} $\Lsf\colon \Gcal \to \Gcal/\Lcal$, whose kernel is $\Lcal$. This
functor has a fully faithful right adjoint $\Rsf\colon \Gcal/\Lcal\to \Gcal$,
whose essential image is the \emph{Giraud subcategory}
$\Ccal:=\Lcal^{\bot_{0,1}}$.
Having a right adjoint, $\Lsf$ preserves colimits;
having an exact left adjoint, $\Rsf$ is left exact and it preserves products
and injectives.
These functors induce an equivalence $\Lsf\colon
\Ccal\simeq \Gcal/\Lcal:\!\Rsf$, which restricts to an equivalence between
the subcategories of injectives:
\[\Lsf\colon \Ccal\cap\Inj(\Gcal)=\Lcal^\bot\cap\Inj(\Gcal)\simeq
\Inj(\Gcal/\Lcal):\!\Rsf.\]
Denote by $\eta\colon 1_\Gcal\Rightarrow \Rsf\Lsf$ the unit of
adjunction. The Giraud subcategory $\Ccal$ is reflective, and the reflection of
$x\in\Gcal$ is the natural morphism $\eta_x\colon x\to \Rsf\Lsf(x)$.

It is well known that the assignment $\Lcal\mapsto (\Lcal,\Lcal^\bot)$ gives a
bijection between localising subcategories and \emph{hereditary torsion pairs}
in $\Gcal$. Recall that a \emph{torsion pair} in $\Gcal$ is a pair of
subcategories $(\Tcal,\Fcal)$ such that $\Tcal^\bot=\Fcal$ and
$\Tcal={}^\bot\Fcal$. For such a pair one has that $\Gcal=\Tcal\ast\Fcal$. For
every $x\in\Gcal$, the short exact sequence $0\to t(x)\to x\to f(x)\to 0$ witnessing
this property is called the \emph{torsion sequence} of $x$ with respect to
$(\Tcal,\Fcal)$, and it is functorially determined by the right adjoint $t\colon
\Gcal\to \Tcal$ of the inclusion $\Tcal\subseteq\Gcal$ (the \emph{torsion
radical}) or by the left adjoint $f\colon \Gcal\to \Fcal$ of the inclusion
$\Fcal\subseteq\Gcal$ (the \emph{torsion-free coradical}).
A torsion pair $(\Tcal,\Fcal)$ is moreover \emph{hereditary} if $\Tcal$ is
closed under subobjects, or equivalently if $\Fcal$ is closed under injective
envelopes. For any class $\Ecal\subseteq\Inj(\Gcal)$, the pair
$({}^\bot\Ecal,\Cogen(\Ecal))$ is a hereditary torsion pair, and every hereditary
torsion pair arises in this way (even with $\Ecal=\{e\}$ consisting of a single
injective object, by
\cite[Prop.~VI.3.7]{sten-75}, adapted using quotients of a generator of $\Gcal$
instead of cyclic modules).

Being closed under subobjects and colimits, a localising subcategory
$\Lcal\subseteq\Gcal$ is determined by which quotients of a fixed generator of
$\Gcal$ it contains. Therefore, the localising subcategories of $\Gcal$ form a
\emph{set} $\LOC(\Gcal):=\{\,\Lcal\subseteq\Gcal\text{ localising}\,\}$, as
opposed to a proper class. Since arbitrary intersections of localising
subcategories are still localising, this set, partially ordered by inclusion, is
a (bounded) complete lattice. The meet is given by intersection, as mentioned.
For the join, given any class $\Scal\subseteq\Gcal$ of objects, we can define
the localising subcategory \emph{generated by} $\Scal$ as
$\Loc(\Scal):=\bigcap_{\Scal\subseteq\Lcal\in\LOC(\Gcal)}\Lcal$. Then the join
of a family of localising subcategory is the localising subcategory generated by
their union. This join is easier to characterise from the side of torsion-free
classes. Indeed, it is easy to see that an arbitrary intersection of
hereditary torsion-free classes is again a hereditary torsion-free class, and
therefore it is the torsion-free class of the join. In other words, for any
family $\{\,\Lcal_i\mid i\in I\,\}\subseteq\LOC(\Gcal)$ we have:
\begin{equation}\label{eqn:join}
	\textstyle \bigvee\Lcal_i={}^\bot(\bigcap \Lcal_i^\bot).
\end{equation}

The following technical lemma relates in the expected way the hereditary torsion
pairs of $\Gcal$ with those of its quotients.

\begin{lemma}\label{lemma:torsion-localisation}
	Let $\Gcal$ be a Grothendieck category, and $\Lcal\subseteq\Gcal$ a
	localising subcategory. Denote by
	$\begin{tikzcd}[sep=small,cramped]
		\Lsf\colon \Gcal\arrow[shift left]{r} & \arrow[shift left]{l} \Gcal/\Lcal :\!\Rsf
	\end{tikzcd}$ the quotient functor and its right adjoint. Let
	$\Tcal\subseteq\Gcal$ be a localising subcategory with $\Lcal\subseteq\Tcal$.
	Then:
	\begin{enumerate}
		\item $\Rsf\Lsf(\Tcal^\bot)\subseteq\Tcal^\bot$;
		\item $(\Lsf\Tcal)^\bot=\Lsf(\Tcal^\bot)$ as subcategories of $\Gcal/\Lcal$.
	\end{enumerate}
\end{lemma}

\begin{proof}
	(1) For any injective $e\in\Tcal^\bot\subseteq\Lcal^\bot$ we have $\Rsf\Lsf(e)\simeq
	e\in\Tcal^\bot$, as mentioned above. Since the functor $\Rsf\Lsf$ is left exact and
	$\Tcal^\bot$ is closed under taking injective envelopes and subobjects, we
	obtain the desired inclusion.

	(2) By adjunction, we have that $y\in\Gcal/\Lcal$ belongs to $(\Lsf\Tcal)^\bot$
	if and only if $\Rsf(y)\in\Tcal^\bot$. By item (1), this is the case if
	$y\in \Lsf(\Tcal^\bot)$, showing the inclusion $\Lsf
	(\Tcal^\bot)\subseteq(\Lsf\Tcal)^\bot$. For the
	converse, let $y\in(\Lsf\Tcal)^\bot$, and write it as $y=\Lsf(x)$ for some
	$x\in\Gcal$. Applying the exact functor $\Lsf$ to the torsion sequence of $x$
	with respect to $(\Tcal,\Tcal^\bot)$, we obtain a short exact sequence
	\[0\to \Lsf(t)\to y\to \Lsf(f)\to 0\]
	in $\Gcal/\Lcal$, with $t\in\Tcal$ and $f\in\Tcal^\bot$. By assumption, the
	monomorphism vanishes, and therefore we conclude that $y\simeq \Lsf(f)\in
	\Lsf(\Tcal^\bot)$.
\end{proof}

Recall that a hereditary torsion pair $(\Tcal,\Fcal)$ in $\Gcal$ is \emph{of finite type} if
$\Fcal$ is closed under direct limits, $\varinjlim\Fcal=\Fcal$. We now show that
it is enough to check this closure property for certain direct systems of
injectives in $\Fcal$, inspired by the proof of \cite[Lemma~5.12]{laki-20}.

Let $\Scal=\{\,s_i\mid i\in I\,\}$ be a set of objects of $\Gcal$, and $\Phi$ a
\emph{filter} on $I$, \emph{i.e.}~a non-empty collection of subsets of $I$
closed under finite intersections and supersets. We
regard the poset $(\Phi,\subseteq)$ as a (codirected) category in the usual way.
Then, the \emph{reduced product diagram} of $\Scal$ over $\Phi$ is the functor
$\Phi^{op}\to \Gcal$, which sends every $U\in\Phi$ to the product $\prod_{i\in
U}s_i$, and every inclusion $U\subseteq V$ to the canonical projection
$\prod_{i\in V}s_i\to \prod_{i\in U}s_i$. The direct limit of this direct system
is the \emph{reduced product} $\prod\Scal/\Phi$ of $\Scal$ over $\Phi$ (see
\emph{e.g.}~\cite{krau-98b}).

Let $\Ccal$ be a complete category, and fix $e\in\Ccal$. It is well known that
for every object $c\in\Ccal$ there is a canonical morphism $\eta_c\colon c\to 
e^{\Ccal(c,e)}$. In fact, these morphisms assemble into the unit $\eta\colon
1_\Ccal\Rightarrow e^{\Ccal(-,e)}$ of an adjunction (see \emph{e.g.}~(the dual of)
\cite[\S6.3]{posi-stov-21}):
\[ \begin{tikzcd}\Ccal(-,e)\colon \Ccal \arrow[shift left]{r} &
\arrow[shift left]{l} \mathsf{Sets}^{op}:\! e^{-}.\end{tikzcd}\]
This unit is monic precisely on the objects of $\Cogen(e)$, by definition.

\begin{lemma}
	Let $\Gcal$ be a Grothendieck category. Then the following are equivalent, for
	$e\in\Gcal$:
	\begin{enumerate}
		\item $\Cogen(e)$ is closed under direct limits;
		\item $\Cogen(e)$ is closed under direct limits of direct systems in $\Prod(e)$;
		\item $\Cogen(e)$ is closed under reduced products of sets of objects in $\Prod(e)$.
	\end{enumerate}
\end{lemma}

\begin{proof}
	The implications $(1\Rightarrow 2)$ and $(2 \Rightarrow3)$ are clear. In the following, consider a small directed category $I$.

	$(2\Rightarrow1)$ Let $\Dsf\colon I\to \Cogen(e)\subseteq\Gcal$ be
	a direct system. Then we have a natural transformation:
	\[\eta_\Dsf\colon \Dsf\Rightarrow e^{\Gcal(-,e)}\circ \Dsf,\]
	which is monic and where the latter is a functor $I\to \Prod(e)$. By assumption, its direct limit
	belongs to $\Cogen(e)$. Since direct limits in $\Gcal$ are exact, we have a
	monomorphism:
	\[\varinjlim \eta_\Dsf\colon \varinjlim \Dsf\hookrightarrow \varinjlim
	(e^{\Gcal(-,e)}\circ \Dsf),\]
	from which we deduce that $\varinjlim \Dsf\in\Cogen(e)$ as well.

	$(3\Rightarrow2)$ Let $\Dsf\colon I\to
	\Prod(e)$ be a direct system. Denote by $\Scal=\{\, s_i:=\Dsf(i)\mid i\in I\,\}$
	the image of $\Dsf$. By \cite[Cor.~4.2]{krau-98b}, there is a filter
	$\Phi$ on the objects of $I$ such that $\varinjlim \Dsf$ is a subobject
	of $\prod \Scal/\Phi$. Since the latter belongs to $\Cogen(e)$ by assumption,
	so does $\varinjlim \Dsf$.
\end{proof}

The following corollary can also be found in \cite[Lemma~3.12]{hrbe-hu-zhu-24}.

\begin{cor}\label{cor:finite-type-inj}
	A hereditary torsion pair $(\Tcal,\Fcal)$ in $\Gcal$ is of finite type if and
	only if we have:
	\[\varinjlim(\Fcal\cap\Inj(\Gcal))\subseteq\Fcal.\]
\end{cor}

\begin{proof}
	As explained above, there is an injective object $e\in\Inj(\Gcal)$ such that
	$\Fcal=\Cogen(e)$, for which then $\Fcal\cap\Inj(\Gcal)=\Prod(e)$. Then the
	lemma gives the claim.
\end{proof}

\subsection{Gabriel filtration}

Let $\Gcal$ be a Grothendieck category. We have seen that we can consider the
localising subcategory generated by a class of objects. For example, we denote
by $\Gcal_0$ the localising subcategory generated by the simple objects of
$\Gcal$, which consists of the \emph{semi-artinian} objects, \emph{i.e.}~those which
admit a transfinite filtration with simple factors.
Another way to construct localising subcategories is as follows. If $\Fsf\colon
\Gcal\to \Gcal'$ is an exact, coproduct-preserving functor to another
Grothendieck category $\Gcal'$ and $\Lcal'\in\LOC(\Gcal')$, then
\[\Fsf^{-1}(\Lcal'):=\{\,x\in \Gcal\mid \Fsf(x)\in\Lcal\,\}\]
is a localising subcategory of $\Gcal$.
This leads to the following construction \cite{gabr-62}.

\begin{dfn}\label{dfn:gabriel-filtration}
	The \emph{Gabriel filtration} of $\Gcal$ is an increasing chain of
	localising subcategories:
	\[0=:\Gcal_{-1}\subseteq\Gcal_0\subseteq\cdots\subseteq\Gcal_\alpha\subseteq
	\cdots\subseteq\Gcal,\quad\alpha\text{ ordinal},\]
	defined inductively. Denoting the quotient functors by 
	$\Lsf_\alpha\colon \Gcal \to \Gcal/\Gcal_\alpha$, we define:
	\begin{enumerate}
		\item $\Gcal_{-1}:=0$;
		\item $\Gcal_{\alpha+1}:=\Lsf_\alpha^{-1}((\Gcal/\Gcal_\alpha)_0)$ for every
			ordinal $\alpha$;
		\item $\Gcal_\lambda:=\Loc(\bigcup_{\alpha<\lambda}\Gcal_\alpha)$ for every
			limit ordinal $\lambda=\bigcup_{\alpha<\lambda}\alpha$.
	\end{enumerate}
	Since $\LOC(\Gcal)$ is a set, this chain must stabilise. If it stabilises at
	$\Gcal$, \emph{i.e.}~if there exists an ordinal $\alpha$ such that
	$\Gcal_\alpha=\Gcal$, we say that $\Gcal$ is \emph{semi-noetherian}, and call
	the minimum such $\alpha$ the \emph{Gabriel dimension} of $\Gcal$, denoted by
	$\Gdim(\Gcal)$.
\end{dfn}

The terminology `semi-noetherian' is due to Popescu \cite[\S5.5]{pope-73}. 
Any locally noetherian Grothen\-dieck category is
semi-noetherian \cite[Prop.~IV.7]{gabr-62}.
The property of being semi-noetherian has a `Serre quality' to it, in the sense
of the following result. Recall that a localising subcategory of a Grothendieck
category is itself Grothendieck. Note that since our
ordinals start from $-1$, we need to adjust the sum of ordinals when one of
the two is finite: $\alpha\oplus \beta:=\alpha+\beta+1$ when at least one of
$\alpha,\beta$ is finite, $\alpha\oplus\beta:=\alpha+\beta$ otherwise. See the
beginning of \cite[\S{}IV]{gabr-62}.

\begin{prop}[{\cite[Prop.~IV.1]{gabr-62}}]\label{prop:semi-noetherian-serre}
	Let $\Gcal$ be a Grothendieck category, and $\Lcal\subseteq\Gcal$ a localising
	subcategory. Then $\Gcal$ is semi-noetherian if and only if both $\Lcal$ and
	$\Gcal/\Lcal$ are semi-noetherian. Moreover, in this case we have:
	\[ \max\{\Gdim(\Lcal),\Gdim(\Gcal/\Lcal)\}\leq \Gdim(\Gcal)\leq
	\Gdim(\Lcal)\oplus \Gdim(\Gcal/\Lcal).\]
\end{prop}

As a corollary, we obtain the first useful consequence of semi-noetherianity.

\begin{cor}\label{cor:localising-simple}
	Let $\Gcal$ be semi-noetherian, and $\Lcal\subseteq\Gcal$ a nonzero localising
	subcategory. Then $\Lcal$ contains a simple object of $\Gcal$.
\end{cor}

\begin{proof}
	By the closure properties of $\Lcal\subseteq\Gcal$, it is easy to see that an
	object $x\in\Lcal$ is simple in $\Lcal$ if and only if it is simple in
	$\Gcal$. Now, if $\Lcal$ is nonzero, since it is semi-noetherian by the
	proposition above, it must have a simple object.
\end{proof}

\begin{rmk}
	The conclusion of the previous corollary also holds in \emph{locally finitely generated} Grothendieck
	categories. Indeed, in that setting any nonzero object has a finitely
	generated subobject, which in turn has a simple quotient, using Zorn's Lemma.
	However, the property of being locally finitely generated does not pass to
	Gabriel quotients: see Example~\ref{ex:no-indecomposable} later on, where a
	nonzero Gabriel quotient of a category of modules does not even have any simple
	object.
	On the other hand, being semi-noetherian does pass to quotients, by the
	proposition above; this will be important when we argue by induction
	later on.
\end{rmk}

An important consequence of this corollary is that the localising subcategories generated
by a simple object can be recognised in the lattice $\LOC(\Gcal)$.

\begin{cor}\label{cor:minimal-nonzero}
	Let $\Gcal$ be a semi-noetherian Grothendieck category. Then the following are
	equivalent for a localising subcategory $\Lcal\subseteq\Gcal$:
	\begin{enumerate}
		\item $\Lcal=\Loc(s)$ for a simple object $s\in\Gcal$;
		\item $\Lcal$ is minimal nonzero in $\LOC(\Gcal)$.
	\end{enumerate}
\end{cor}

\begin{proof}
	($1\Rightarrow 2$) If $s$ is a simple object, the subcategory of objects
	admitting a transfinite filtration with factors isomorphic to $s$ is
	localising, and therefore it is $\Loc(s)$. It follows that in such a
	subcategory, every nonzero object has $s$ as a subfactor. In particular, $\Loc(s)$
	is the smallest localising subcategory containing any of its nonzero objects,
	and so it is minimal nonzero.

	($2\Rightarrow1$) If $\Lcal$ is nonzero, it
	contains a simple object $s$ by the Proposition, so
	$\Loc(s)\subseteq\Lcal$. If $\Lcal$ is minimal nonzero, this inclusion is
	an equality.
\end{proof}

\subsection{Gabriel Spectrum: via injectives}

We now introduce the topological space which is central to this
Iork. In the next two subsections we will give a few descriptions of its
points and its topology, and establish some of its properties.

As before, let $\Gcal$ be a Grothendieck category. Any indecomposable injective
object of $\Gcal$ is \emph{coirreducible}, meaning that it is the injective
envelope of all its nonzero subobjects. In particular, it is the injective
envelope of some quotient of a fixed generator of $\Gcal$. This shows that the
indecomposable injectives, up to isomorphism, form a set.

\begin{dfn}
	The \emph{Gabriel spectrum} of $\Gcal$, denoted by $\GSpec(\Gcal)$, is the set
	of isoclasses of indecomposable injective objects of $\Gcal$. It is endowed
	with the \emph{Gabriel topology}, whose closed sets are those of the
	form $\GSpec(\Gcal)\cap\Lcal^\bot$, for a localising subcategory
	$\Lcal\subseteq\Gcal$.
\end{dfn}

In other words, the closed sets are the sets of (isoclasses of) indecomposable
injectives which are torsion-free for a certain hereditary torsion pair. This
topology is also referred to as the `full support topology' in the literature
\cite{burk-94-thesis,pres-09}. The word `support'
refers to the notion with the same name in commutative algebra, which can be used
to classify hereditary torsion pairs in $\Mod(R)$, for $R$ commutative
noetherian, and in other related abelian
categories (see \emph{e.g.}~\cite{gabr-62,pavo-vito-21}, and
\S\ref{sec:commutative-noetherian}, in particular
Proposition~\ref{prop:s(p)-inj-envelope}). The
adjective `full' underlines the fact that we are considering the topology
induced by \emph{all} hereditary torsion pairs, as opposed to only the ones of
finite type, as it is done in the literature to define the
\emph{Ziegler spectrum} \cite{herz-97,krau-97}.

To show that these sets
indeed form a topology, we first give an alternative description.

\begin{lemma}\label{lemma:topology-inj}
	The following are equivalent for a subset $V\subseteq\GSpec(\Gcal)$:
	\begin{enumerate}
		\item $V$ is closed;
		\item $V=\Ind(\Ecal)$ for a class $\Ecal\subseteq\Inj(\Gcal)$ such that
			$\Ecal=\Prod(\Ecal)$;
		\item $V=\Ind(\Prod(V))$.
	\end{enumerate}
\end{lemma}

\begin{proof}
	$(1\Rightarrow2)$ Let $V=\GSpec(\Gcal)\cap\Lcal^\bot$ be a closed subset, for a
	localising subcategory $\Lcal\subseteq\Gcal$. Since
	$\Prod(\Lcal^\bot)=\Lcal^\bot$, taking $\Ecal:=\Inj(\Gcal)\cap\Lcal^\bot$
	gives that $\Ecal=\Prod(\Ecal)$, and $V=\Ind(\Ecal)$.

	$(2\Rightarrow1)$ Let $\Ecal=\Prod(\Ecal)\subseteq\Inj(\Gcal)$. Then we know that
	 $\Lcal:={}^\bot\Ecal$ is a localising subcategory, such that
	 $\Lcal^\bot=\Cogen(\Ecal)$. Therefore, we have
	 $\Inj(\Gcal)\cap\Lcal^\bot=\Inj(\Gcal)\cap\Cogen(\Ecal)=\Prod(\Ecal)=\Ecal$. Hence
	 $V:=\Ind(\Ecal)=\GSpec(\Gcal)\cap\Lcal^\bot$ is closed.

	$(2\Rightarrow 3)$ Let $\Ecal=\Prod(\Ecal)\subseteq\Inj(\Gcal)$, and
	$V=\Ind(\Ecal)$. Then, since $\Prod(V)\subseteq\Prod(\Ecal)$, we have
	$\Ind(\Prod(V))\subseteq\Ind(\Prod(\Ecal))=\Ind(\Ecal)=V$, and the converse
	inclusion is clear.

	$(3\Rightarrow2)$ is obvious.
\end{proof}

From item (2), it is easy to see that arbitrary intersections of closed
sets are closed. Indeed, arbitrary intersections of classes of injectives
closed under $\Prod(-)$ are still closed under $\Prod(-)$. For finite unions of closed
sets, we have the following lemma.

\begin{lemma}
	Let $\Ecal_1,\Ecal_2\subseteq\Inj(\Gcal)$ be such that $\Ecal_i=\Prod(\Ecal_i)$.
	Then we have that 
	\[\Ind(\Prod(\Ecal_1\cup\Ecal_2))=\Ind(\Ecal_1)\cup\Ind(\Ecal_2).\]
	In particular, $\Ind(\Ecal_1)\cup\Ind(\Ecal_2)\subseteq\GSpec(\Gcal)$ is closed.
\end{lemma}

\begin{proof}
	We have that $\Prod(\Ecal_1\cup\Ecal_2)=\add\{\,e_1\oplus e_2\mid e_i\in
	\Ecal_i\,\}$. The indecomposable injective objects in a Grothendieck category
	have local endomorphism ring, so they are direct summands of a finite direct sum if
	and only if they are direct summands of one of the summands.
\end{proof}

\begin{rmk}
	Observe that the characterisations (2--3) of Lemma~\ref{lemma:topology-inj} of the
	Gabriel topology do not depend on $\Gcal$, but rather only on its
	subcategory of injectives $\Inj(\Gcal)$. This fact will be exploited later on.
\end{rmk}

By definition, we have the following well-defined surjective assignment.

\begin{dfn}\label{dfn:gsupp}
	For any localising subcategory $\Lcal\subseteq\Gcal$, its \emph{Gabriel
	support} is the open set $\GSupp(\Lcal):=\GSpec(\Gcal)\setminus\Lcal^\bot$.
\end{dfn}

It is natural to ask whether the Gabriel support $\GSupp(\Lcal)$ uniquely
determines the localising subcategory $\Lcal$; that is, whether the topology on
$\GSpec(\Gcal)$ classifies the localising subcategories. This is not the case in
general, but it is true in the semi-noetherian setting.

\begin{prop}\label{prop:gsupp}
	Assume $\Gcal$ is semi-noetherian. Then:
	\begin{enumerate}
		\item For any class $\Ecal\subseteq\Inj(\Gcal)$ such that $\Ecal=\Prod(\Ecal)$, we
			have $\Ecal=\Prod(\Ind(\Ecal))$.
		\item The assignment of Gabriel support is injective.
	\end{enumerate}
	In particular, $\GSupp$ is a lattice isomorphism between $\LOC(\Gcal)$ and
	the frame of opens of $\GSpec(\Gcal)$.
\end{prop}

\begin{proof}
	(1) To show that $\Ecal=\Prod(\Ind(\Ecal))$, it is enough to prove that the two
	hereditary torsion-free classes $\Cogen(\Ecal)=\Cogen(\Ind(\Ecal))$ coincide, and
	then intersect them with $\Inj(\Gcal)$. For the
	nontrivial inclusion ($\subseteq$), we apply \cite[Thm.~IV.1]{gabr-62}, which
	states that in a semi-noetherian category every injective object is the
	injective envelope of a coproduct of
	indecomposable injectives. As a consequence, every injective can be
	reconstructed from its indecomposable direct summands by means of taking coproducts
	and injective envelopes, operations under which the hereditary torsion-free
	class $\Cogen(\Ind(\Ecal))$ is closed.

	(2) For any localising subcategory $\Lcal$, setting
	$\Ecal:=\Inj(\Gcal)\cap\Lcal^\bot$ and using (1), we have:
	\[\Lcal={}^\bot\Ecal={}^\bot\Ind(\Ecal)={}^\bot(\GSpec(\Gcal)\setminus\GSupp(\Lcal)).\]
	This gives an inverse assignment to $\GSupp$, proving its injectivity.
\end{proof}

Recall that a nonzero object in an additive category is \emph{superdecomposable} if it
does not have nonzero indecomposable summands. Such objects may exist in
various natural additive categories (see for example \cite[Summary~8.17]{jens-lenz-89}), but not in the
category of injectives of a semi-noetherian Grothendieck category.

\begin{lemma}[{\cite[Thm.~5.5.5 and Thm.~5.3.3]{pope-73}}]\label{lemma:no-superdecomposables}
	If $\Gcal$ is semi-noetherian, it has no superdecomposable injectives.
\end{lemma}

We now compare the Gabriel spectra of a Grothendieck
category and of its quotients.

\begin{prop}\label{prop:spectrum-quotients}
	Let $\Gcal$ be a Grothendieck category, let $\Lcal\subseteq\Gcal$ be a localising
	subcategory, and denote by
	$\begin{tikzcd}[sep=small,cramped]
		\Lsf\colon \Gcal \arrow[shift left]{r} & \arrow[shift left]{l} \Gcal/\Lcal :\!\Rsf
	\end{tikzcd}$
	the quotient functor and its right adjoint. Then:
	\begin{enumerate}
		\item The inverse equivalences $\Lsf\colon \Inj(\Gcal)\cap\Lcal^\bot\simeq
			\Inj(\Gcal/\Lcal):\!\Rsf$ induce inverse homeomorphisms between the closed subspace
			$\GSpec(\Gcal)\setminus\GSupp(\Lcal)\subseteq\GSpec(\Gcal)$ and
			$\GSpec(\Gcal/\Lcal)$.
		\item Using the homeomorphism provided by (1), for any localising
			subcategory $\Tcal\subseteq\Gcal$ with $\Lcal\subseteq\Tcal$, we have
			$\GSupp(\Tcal)=\GSupp(\Lcal)\cup \Rsf(\GSupp(\Lsf\Tcal))$.
	\end{enumerate}
\end{prop}

\begin{proof}
	(1) Since $\Lsf$ and $\Rsf$ are equivalences between idempotent-complete additive
	categories, they induce a bijection between
	$\GSpec(\Gcal)\cap\Lcal^\bot=\GSpec(\Gcal)\setminus\GSupp(\Lcal)$ and
	$\GSpec(\Gcal/\Lcal)$. Moreover, they also preserves products; so they induce
	a homeomorphism, by the description of closed sets given in Lemma~\ref{lemma:topology-inj}.

	(2) This equality is easier to see passing to the complements.
	Observe that both $(\GSupp(\Tcal))^\cpl$ and
	$\left(\GSupp(\Lcal)\cup \Rsf(\GSupp(\Lsf\Tcal))\right)^\cpl$ are contained in $(\GSupp(\Lcal))^\cpl$. To
	show that these two complements are equal, take
	$e$ in $(\GSupp(\Lcal))^\cpl=\GSpec(\Gcal)\cap\Lcal^\bot$. Then $e$ lies in
	$(\GSupp(\Tcal))^\cpl$ if and only if $e$ belongs to $\Tcal^\bot$. By adjunction, this
	is the case if and only if $\Lsf(e)$ belongs to $\Lsf(\Tcal^\bot)=(\Lsf\Tcal)^\bot$,
	where this equality comes from Lemma~\ref{lemma:torsion-localisation}(2). This
	is equivalent to the fact that $\Lsf(e)$ lies in $(\GSupp(\Lsf\Tcal))^\cpl$,
	\emph{i.e.}~ that $e\simeq \Rsf\Lsf(e)$ lies in $\left(\Rsf(\GSupp(\Lsf\Tcal))\right)^\cpl$.
\end{proof}

Now, under the assumption that $\Gcal$ is semi-noetherian, we give another
description of the points of $\GSpec(\Gcal)$.
In the following, we fix the notation
$0=\Gcal_{-1}\subseteq\Gcal_0\subseteq\cdots\subseteq\Gcal_\delta=\Gcal$ for the
Gabriel filtration of $\Gcal$, assuming that $\delta$ is the Gabriel dimension;
we also denote by $\Ccal_{\alpha}:=\Gcal_\alpha^{\bot_{0,1}}$ the Giraud
subcategories, and by
$\begin{tikzcd}[cramped,sep=small]
	\Lsf_\alpha\colon \Gcal \arrow[shift left]{r} &
	\arrow[shift left]{l} \Gcal/\Gcal_\alpha :\!\Rsf_\alpha
\end{tikzcd}$
the quotient functors and their right adjoints, for all ordinals $\alpha\leq
\delta$.

We start by introducing the following concept.

\begin{dfn}
	An object of $\Gcal$ of the form $\Rsf_\alpha(s)$, for $s\in\Gcal/\Gcal_\alpha$ a
	simple object in the quotient, is called a \emph{Gabriel-simple} of $\Gcal$ of
	\emph{dimension} $\alpha+1$.
\end{dfn}

For example, the Gabriel-simples of dimension $0$ are precisely the
simple objects of $\Gcal$. By Schanuel's Lemma, the simple objects in the
quotients $\Gcal/\Gcal_\alpha$ are all \emph{bricks}, meaning that their endomorphism ring
is a skew-field. Since the functors $\Rsf_\alpha$ are fully faithful, this is also
true of the Gabriel-simples. In particular, they are indecomposable with local endomorphism
ring.

Observe that the Gabriel-simple objects of $\Gcal$ of
dimension $\alpha+1$ are precisely the simple objects of $\Ccal_\alpha$, which
is a Grothen\-dieck category by the equivalence
$\Rsf_\alpha\colon\Gcal/\Gcal_\alpha\simeq\Ccal_\alpha$.
In other words, these are the objects of $\Ccal_\alpha$ whose proper quotients
(in $\Gcal$) all lie in $\Gcal_\alpha$.

The connection with the Gabriel spectrum is given by the following lemma.

\begin{lemma}[{\cite[Thm.~5.56]{pope-73}}]\label{lemma:gabriel-simple-subobject}
	Let $\Gcal$ be semi-noetherian. Then taking injective envelopes gives a
	bijection between isoclasses of Gabriel-simples and indecomposable
	injectives of $\Gcal$.
\end{lemma}

\subsection{Topological properties of the Gabriel Spectrum}\label{subsec:topology-gabriel-spectrum}

We now proceed to establish some topological properties of the Gabriel spectrum,
under the assumption that $\Gcal$ is semi-noetherian.
Recall that a topological space is \emph{Alexandrov} if 
arbitrary unions of closed sets are closed, or equivalently,
if arbitrary intersections of open sets are open. A
famous example of such a space is the Hochster dual of Zariski topology on the
prime spectrum of a commutative
noetherian ring. An Alexandrov topology on a set determines a (partial) preorder
on the points, called the \emph{closure preorder}, defined by setting $x\preceq
y$ if $x\in\overline{\{y\}}$. Conversely, any partial preorder induces an
Alexandrov topology, where the closed sets are the lower-sets. These two
assignments are inverse bijections.
In general, the Gabriel topology on the Gabriel spectrum is not
expected to be Alexandrov. In fact, it is possible to characterise precisely
when this is the case, thanks to the following notion \cite{pavo-25}.

\begin{dfn}\label{dfn:torsion-simple}
	A nonzero object $x\neq 0$ in a Grothendieck category is \emph{hereditary-torsion-simple}
	if it belongs to $\Tcal\cup\Fcal$ for every hereditary torsion pair
	$(\Tcal,\Fcal)$. In other words, if it is always either torsion or
	torsion-free, for any hereditary torsion pair.
\end{dfn}

\begin{thm}\label{thm:torsion-simple-socle}
	Let $\Gcal$ be a semi-noetherian Grothendieck category. Then $\GSpec(\Gcal)$ is
	Alexandrov if and only if every indecomposable injective has a
	hereditary-torsion-simple subobject.
\end{thm}

\begin{proof}
	$(\Rightarrow)$ Assume $\GSpec(\Gcal)$ is Alexandrov, and let
	$e\in\GSpec(\Gcal)$ be an indecomposable injective. Consider the collection of
	all hereditary torsion pairs $(\Tcal,\Fcal)$ for which $e\in\GSupp(\Tcal)$,
	that is, $e\notin\Fcal$.
	Take $S_e$ to be the intersection of all these supports $\GSupp(\Tcal)$, which
	contains $e$ and is open by assumption (in fact, it is the \emph{smallest open
	set} containing $e$). Let therefore $\Tcal_e$ be the corresponding hereditary
	torsion class, with $\GSupp(\Tcal_e)=S_e$ (see Proposition~\ref{prop:gsupp}). Since $e$ belongs to
	$\GSupp(\Tcal_e)$, it has nonzero torsion part $0\neq t\subseteq e$ with
	respect to $\Tcal_e$; we show that $t$ is hereditary-torsion-simple. Indeed,
	let $(\Tcal,\Fcal)$ be any hereditary torsion pair. Then we have two cases.
	If $e\in\GSupp(\Tcal)$, by construction $S_e\subseteq\GSupp(\Tcal)$; since
	Gabriel support is an isomorphism of lattices, by
	Proposition~\ref{prop:gsupp}, we deduce that $t\in\Tcal_e\subseteq\Tcal$.
	If instead $e$ does not belong to $\GSupp(\Tcal)$, then $e$ lies in $\Fcal$;
	therefore $t$ lies in $\Fcal$ as well.

	$(\Leftarrow)$ Suppose that every indecomposable injective has a
	hereditary-torsion-simple subobject. Let
	$\{\,(\Tcal_i,\Fcal_i)\mid i\in I\,\}$ be a family of hereditary torsion pairs
	in $\Gcal$, and $V_i:=\GSpec(\Gcal)\cap\Fcal_i$ the corresponding closed sets.
	Consider the hereditary torsion pair $(\Tcal:=\bigcap\Tcal_i,\Fcal)$; we show
	that $\bigcup V_i=\GSpec(\Gcal)\cap\Fcal$, so this union is closed. Indeed,
	since $\Fcal_i\subseteq\Fcal$ for every $i\in I$, the inclusion $(\subseteq)$
	is clear. For the converse, let $e\in\GSpec(\Gcal)$ be such that $e\notin \Fcal_i$ for every
	$i$; we prove that then $e\notin \Fcal$. Let $0\neq t\subseteq e$ be a
	hereditary-torsion-simple subobject of $e$. Then $e$ is the injective
	envelope of $t$, and since $\Fcal_i$ is closed under injective envelopes, we must
	have $t\notin\Fcal_i$ for every $i$. Since $t$ is hereditary-torsion-simple,
	this implies that $t\in\Tcal_i$ for every $i$. Then $t$ lies in $\Tcal=\bigcap\Tcal_i$,
	and therefore $e$ is not in $\Fcal$.
\end{proof}

As an immediate application, we obtain the following corollary.

\begin{cor}
	If $\Gcal$ is semi-noetherian and every hereditary torsion pair in $\Gcal$ is
	\emph{stable}, \emph{i.e.}~the torsion class is closed under
	taking injective envelopes, then $\GSpec(\Gcal)$ is Alexandrov.
	In particular, this is the case for $\Gcal=\Mod(R)$, for a commutative
	noetherian ring $R$.
\end{cor}

\begin{proof}
	The assumption is equivalent to saying that the indecomposable injectives are
	themselves hereditary-torsion-simple, because if they are not torsion-free,
	they are the injective envelopes of their torsion part. For a commutative
	noetherian ring $R$, all hereditary torsion pairs in $\Mod(R)$ are stable
	\cite[Prop.~VII.4.5]{sten-75}.
\end{proof}

In our application to cosilting hearts in the commutative noetherian setting in
\S\ref{sec:commutative-noetherian}, the strong assumption of this corollary
will not necessarily hold, but we will use
Theorem~\ref{thm:torsion-simple-socle} to show that the Gabriel spectra will
nonetheless be Alexandrov (Corollary~\ref{cor:T0-alexandrov}).

The existence of hereditary-torsion-simple subobjects of indecomposable
injective objects allows for another description of the Gabriel
topology. Assume that every $e\in\GSpec(\Gcal)$ has a hereditary-torsion-simple
subobject $s(e)\subseteq e$. The Gabriel
topology, being Alexandrov, is then determined by its closure preorder, $e'\preceq e$
if and only if $e'\in\overline{\{e\}}$. 

\begin{lemma}\label{lemma:hts-topology-characterisation}
	In the notation above, the following are equivalent:
	\begin{enumerate}
		\item $e'\preceq e$;
		\item $\Gcal(s(e'),e)\neq 0$;
		\item $\exists\, 0\neq x\subseteq s(e')$ such that $\Gcal(x,s(e))\neq 0$.
	\end{enumerate}
\end{lemma}

\begin{proof}
	($1\Leftrightarrow2$) Let $(\Tcal_e,\Fcal_e)=({}^\bot e,\Cogen(e))$ be the hereditary torsion pair
	cogenerated by $e$. We have that $e'\preceq e$ if and only if
	$e'\in\overline{\{e\}}=\GSpec(\Gcal)\cap \Fcal_e$, by definition. This is
	equivalent to $s(e')\in\Fcal_e$, since this class is closed under subobjects
	and injective envelopes. Now, since $s(e')$ is hereditary-torsion-simple, this
	is in turn equivalent to $s(e')\notin\Tcal_e={}^\bot e$.

	($2\Leftrightarrow3$) This just follows from the fact that
	$e$ is the injective envelope of $s(e)$.
\end{proof}

\begin{rmk}
	Observe that, even if the space is not Alexandrov, the closure of a point
	$e\in\GSpec(\Gcal)$ can be characterised as the set of indecomposable
	injectives $e'$ without nonzero subobjects in $\Tcal_e={}^\bot e$ (and hence
	torsion-free, $e'\in\Fcal_e=\Cogen(e)$). This means:
	\[e'\in\overline{\{e\}} \iff \forall\, 0\neq x\subseteq e'\;\; \Gcal(x,e)\neq 0.\]
	The advantage we gain when we can work with the
	hereditary-torsion-simple subobjects is that by (2) above we
	only need to check this condition for $x:=s(e')$.
\end{rmk}

Recall
that a topological space is $(\Tsf_0)$ (or \emph{Kolmogorov}) if for any two
distinct points, there exists an open set which contains one but not the other.
This is equivalent to the property that if two points belong to each other's
closure, they are equal.
We will obtain that $\GSpec(\Gcal)$ is $(\Tsf_0)$ as a consequence of the
following notion.

\begin{dfn}\label{dfn:cantor-bendixson}
	Let $X$ be a topological space. The \emph{Cantor--Bendixson filtration} of $X$
	is an increasing chain of open subsets:
	\[ \emptyset=:X_{-1}\subseteq X_0\subseteq X_1\subseteq\cdots\cdots\subseteq
	X_\alpha\subseteq\cdots\subseteq X, \quad \alpha\text{ ordinal},\]
	defined inductively by setting:
	\begin{enumerate}
		\item $X_{-1}:=\emptyset$;
		\item $X_{\alpha+1}$ to be the union of $X_\alpha$ with the set of isolated
			points of the topological subspace $X\setminus X_\alpha$, for every
			ordinal $\alpha$;
		\item $X_{\lambda}:=\bigcup_{\alpha<\lambda} X_\alpha$ for every limit
			ordinal $\lambda=\bigcup_{\alpha<\lambda}X_\alpha$.
	\end{enumerate}
	This chain must stabilise for cardinality reasons. The \emph{Cantor--Bendixson
	rank} is the minimum ordinal $\alpha$ for which $X_\alpha=X$, if it exists.
\end{dfn}

\begin{lemma}
	If a topological space $X$ has Cantor--Bendixson rank, then it is $(\Tsf_0)$.
\end{lemma}

\begin{proof}
	Let $x_1,x_2\in X$. In the notation above, let $\alpha_i\geq 0$ be the minimum
	ordinal for which $x_i\in X_{\alpha_i}$, for $i=1,2$. If $\alpha_1\neq
	\alpha_2$, say $\alpha_1<\alpha_2$, then $X_{\alpha_1}$ is an open
	set containing $x_1$ but not $x_2$. If instead $\alpha_1=\alpha_2=:\alpha$, then
	$x_1,x_2$ are both isolated points of the subspace $X_{\alpha}$; in
	particular, $X_{\alpha-1}\cup\{x_1\}\subseteq X$ is an open set which contains $x_1$ but
	not $x_2$.
\end{proof}

The striking similarity between Definition~\ref{dfn:cantor-bendixson} and the
definition of the Gabriel filtration of $\Gcal$
(Definition~\ref{dfn:gabriel-filtration}) is explained by the following result.

\begin{prop}\label{prop:cantor-bendixson-gabriel}
	Let $\Gcal$ be a semi-noetherian Grothendieck category. Then:
	\begin{enumerate}
		\item Gabriel support gives a bijective correspondence between localising
			subcategories of $\Gcal$ generated by a simple object and isolated points of
			$\GSpec(\Gcal)$.
		\item Gabriel support gives a bijective correspondence between the Gabriel
			filtration of $\Gcal$ and the Cantor--Bendixson filtration of
			$\GSpec(\Gcal)$.
	\end{enumerate}
	\end{prop}

\begin{proof}
	(1) Recall that since $\Gcal$ is semi-noetherian, by
	Proposition~\ref{prop:gsupp}, Gabriel support gives a bijection between
	localising subcategories of $\Gcal$ and open subsets of $\GSpec(\Gcal)$. Let
	$e\in\GSpec(\Gcal)$ be an isolated (\textit{i.e.} open) point. In
	particular, $\{e\}$ is a minimal non-empty open set, and therefore the corresponding
	localising subcategory is minimal nonzero, having thus the form $\Loc(s)$
	for a simple object $s\in\Gcal$, by Corollary~\ref{cor:minimal-nonzero}.
	Conversely, consider $\Lcal=\Loc(s)$ for a simple object $s$, and let $e$ be the
	(indecomposable) injective envelope of $s$. For any other indecomposable
	injective $e'$, a nonzero morphism $s\to e'$ would necessarily be monic, and $e$ would be a summand of $e'$, showing that indeed 
	$e'\in\Lcal^\bot$. This shows that $\GSupp(\Lcal)=\{e\}$, so $e$ is an
	isolated point.

	(2) If we denote by $(\Gcal_\alpha\mid \alpha)$ the Gabriel filtration of
	$\Gcal$ and by $(X_\alpha\mid \alpha)$ the Cantor--Bendixson filtration of
	$\GSpec(\Gcal)$, we show that $\GSupp(\Gcal_\alpha)=X_\alpha$ by induction on
	$\alpha$. For the base step, observe that by definition $X_0$ is the union of all
	isolated points of $\GSpec(\Gcal)$, while $\Gcal_0$ is the join in $\LOC(\Gcal)$ of the
	localising subcategories generated by the simple objects; by (1) and the fact
	that $\GSupp$ is an isomorphism of lattices (Proposition~\ref{prop:gsupp}), we
	conclude. The limit ordinal case follows immediately from the fact that
	$\GSupp$, as an isomorphism of (complete) lattices, preserves arbitrary joins.

	For the successor ordinal case, assume $\GSupp(\Gcal_\alpha)=X_\alpha$.
	Observe that $\Gcal/\Gcal_\alpha$ is itself semi-noetherian, its Gabriel
	filtration being $0=\Lsf_\alpha\Gcal_\alpha\subseteq
	\Lsf_\alpha\Gcal_{\alpha+1}\subseteq\cdots$. By
	Proposition~\ref{prop:spectrum-quotients}, we can identify the Gabriel spectrum of
	$\Gcal/\Gcal_\alpha$ with the subspace
	$\GSpec(\Gcal)\setminus\GSupp(\Gcal_\alpha)=\GSpec(\Gcal)\setminus X_\alpha$.
	Moreover, using this identification, we have:
	\[\GSupp(\Gcal_{\alpha+1})=\GSupp(\Gcal_\alpha)\cup\GSupp(\Lsf_\alpha\Gcal_{\alpha+1})=
	X_\alpha\cup\GSupp((\Gcal/\Gcal_\alpha)_0).\]
	By (1), this is then the union of $X_\alpha$ and the set of isolated points of
	\[\GSpec(\Gcal/\Gcal_\alpha)\simeq \GSpec(\Gcal)\setminus
	\GSupp(\Gcal_\alpha)=\GSpec(\Gcal)\setminus X_\alpha.\]
	This union is by definition $X_{\alpha+1}$.
\end{proof}

\begin{cor}
	The Gabriel spectrum of a semi-noetherian Grothendieck category is $(\Tsf_0)$.
\end{cor}

\begin{rmk}
	Observe that in the Proposition we assume from the beginning that $\Gcal$ is
	semi-noetherian, as this is necessary to prove item (1). In other words, it is
	not sufficient to know that $\GSpec(\Gcal)$ admits Cantor--Bendixson rank to
	deduce that $\Gcal$ is semi-noetherian; this is illustrated by the following
	counterexample.
\end{rmk}

\begin{ex}[Of a category with Cantor--Bendixson rank but no Gabriel dimension]
	\label{ex:no-indecomposable}
	This example is based on an answer by Eric
	Wofsey on \texttt{math.stackexchange.com}\footnote{\texttt{https://math.stackexchange.com/questions/1212747/on-semisimple-rings}}.
	Recall that a \emph{boolean ring} is a (necessarily commutative) ring $R$ in
	which every element is idempotent.
	If $M$ is an indecomposable $R$-module and
	$r\in R$ is any element,
	since we have $M=rM\oplus (1-r)M$, either $r$ or $1-r$
	belongs to the annihilator $I:=\mathsf{Ann}(M)$. This shows that $M$ is a
	module over $R/I\simeq \Fbb_2$, the field with two elements. Therefore it is
	actually simple, and it has two elements.

	This shows that all indecomposable modules are simple and finite; moreover,
	they are also injective, as they must coincide with their (indecomposable)
	injective envelope. It follows that $\GSpec(\Mod(R))$ is a discrete space, so
	it has Cantor--Bendixson rank $0$.

	On the other hand, if $R$ is infinite (for example, if $R:=\Fbb_2^I$, where
	$\Fbb_2$ is the field with two elements and $I$ is an infinite set), it cannot
	belong to $(\Mod(R))_0$, as otherwise, being finitely generated, it would have
	finite length, and therefore finitely many elements. This shows that in this
	case the quotient $(\Mod(R))/(\Mod(R))_0$ is nonzero, but it does not contain
	any simple objects, or indeed indecomposable,
	objects. In particular, $\Mod(R)$ is not semi-noetherian.
\end{ex}

\section{Cosilting mutation and the Gabriel Spectrum}
\label{sec:general}

In this section we focus on the Gabriel spectra of certain Grothendieck
categories, arising as hearts of $t$-structures in the context of cosilting
theory. We will be particularly interested in the effects of the operation of
\emph{cosilting mutation} on the topologies of these Gabriel spectra.

\subsection{Preliminaries}

We start by recalling the necessary notions about $t$-structures and
pure-injective cosilting objects.

\subsubsection{Compactly generated triangulated categories}
Let $\Dcal$ be a triangulated category admitting all set-indexed coproducts. Recall
that $\Dcal$ is said to be \emph{compactly generated} if the subcategory
$\Dcal^\cpl$ formed by the objects $x\in\Dcal$ such that $\Dcal(x,-)$ commutes with
coproducts (called \emph{compact objects}) is skeletally small and
$(\Dcal^\cpl)^\bot=0$. It is well-known that compactly generated triangulated
categories admit all set-indexed products \cite[Cor.~1.18]{neem-01}. From now on, $\Dcal$ will denote
a compactly generated triangulated category.

\subsubsection{Pure-injective objects}
An object $c$ in a compactly generated triangulated category $\Dcal$ is said to
be \emph{pure-injective} if, for any set $I$, the canonical map $c^{(I)}\to c$
(often referred to as the \emph{summation map}) factors through the canonical
map $c^{(I)}\to c^I$. It can be shown that an object $c$ is pure-injective if
and only if the functor $\Dcal(-,c)_{|\Dcal^\cpl}$ is injective in the category of
(contravariant) additive functors $\Dcal^\cpl\to \Mod(\Zbb)$, which we denote by
$\Mod(\Dcal^\cpl)$. For more on pure-injective objects, we refer to
\cite{krau-00}.

\subsubsection{$t$-structures and HRS-tilting}

A \emph{$t$-structure} in a triangulated category $\Dcal$ is a pair
$\tau:=(\Ucal,\Vcal)$ of subcategories such that $\Dcal(\Ucal,\Vcal)=0$,
$\Ucal[1]\subseteq\Ucal$ and $\Ucal\ast\Vcal=\Dcal$. The subcategories $\Ucal$
and $\Vcal$ are called the \emph{aisle} and the \emph{coaisle} of $\tau$,
respectively. The inclusion of the aisle $\Ucal\subseteq\Dcal$ admits a
right-adjoint $u\colon \Dcal\to \Ucal$, known as the \emph{left-truncation}
functor of $\tau$; similarly, the inclusion $\Vcal\subseteq\Dcal$ has a
left-adjoint $v\colon \Dcal\to\Vcal$, the \emph{right-truncation} functor of
$\tau$. For any object $x\in\D(R)$ there is a distinguished triangle, called the
\emph{truncation triangle}, of the form
\[u(x)\to x\to v(x)\to u(x)[1],\]
witnessing the fact that
$x\in\Dcal=\Ucal\ast\Vcal$.

The intersection $\Hcal_\tau:=\Ucal[-1]\cap\Vcal$, called the \emph{heart} of
$\tau$, is an abelian category, whose short exact sequences are precisely
the distinguished triangles of $\Dcal$ with terms in $\Hcal_\tau$ \cite{bbd-81}.
The truncation functors induce a \emph{cohomology} functor with respect to
$\tau$, which is $H_\tau^0:=u(v(-)[1])[-1]\simeq v(u(-[1])[-1])\colon
\Dcal\to\Hcal_\tau$. This functor is cohomological, \emph{i.e.}~it sends
distinguished triangles of $\Dcal$ to long exact sequences in $\Hcal_\tau$.
For every $n\geq 0$ we have a binatural transformation $\theta_n\colon
\Ext^n_\Hcal(-,-)\Rightarrow\Dcal(-,-[n])$. These are isomorphisms for $n=0,1$,
and a monomorphism for $n=2$. We say that $\tau$ has \emph{derived type} if $\theta_n$ is
an isomorphism for every $n\geq 0$.

Following from the axioms, the aisle of a $t$-structure is closed under
extensions, positive shifts and coproducts; dually, the coaisle is closed under
extensions, negative shifts and products.
A $t$-structure $\tau=(\Ucal,\Vcal)$ is called: \emph{stable} if $\Ucal$
(equivalently $\Vcal$) is triangulated, that is, closed under all shifts;
\emph{smashing} if $\Vcal$ is closed under coproducts; \emph{nondegenerate} if
$\bigcap_{n\in\Zbb}\Ucal[n]=0=\bigcap_{n\in\Zbb}\Vcal[n]$.

If $\tau=(\Ucal,\Vcal)$ is a $t$-structure with heart $\Hcal$ and
$t=(\Tcal,\Fcal)$ is a torsion pair in $\Hcal$, the (right) \emph{HRS-tilt} of $\tau$ at
$t$ is the $t$-structure \cite{happ-reit-smal-96}:
\[\tau_t:=(\Ucal\ast\Tcal,\Fcal\ast\Vcal[-1])\qquad\text{with heart
}\Hcal_t=\Fcal\ast\Tcal[-1].\]
The pair $(\Fcal,\Tcal[-1])$ is a torsion pair in $\Hcal_t$.

\subsubsection{Cosilting objects}
An object $c\in\Dcal$ is said to be \emph{cosilting} if the pair $\tau_c:=({}^{\bot_{\leq
0}}c,{}^{\bot_{>0}}c)$ is a $t$-structure, which is then said to be a
\emph{cosilting} $t$-structure. Observe that in this case $c$ belongs to $({}^{\bot_{\leq
0}}c)^\bot={}^{\bot_{>0}}c$, and therefore we have that $\Dcal(c,c[i])=0$ for every $i>0$. Two cosilting
objects $c$ and $c'$
are said to be \emph{equivalent} if $\Prod(c)=\Prod(c')$, or equivalently if
$\tau_{c}=\tau_{c'}$ \cite{nico-saor-zvon-19,psar-vito-18}. We will denote the heart $\Hcal_{\tau_c}$ and the
cohomology functor $H^0_{\tau_c}$ by $\Hcal_c$ and $H^0_c$, respectively, for
short.
It follows from the definition of $\tau_c$ and the adjunction
formulas for the truncation functors that there is a
natural isomorphism $\Dcal(-,c)\simeq \Hcal_c(H_c^0(-),H_c^0(c))$ as functors on
$\Dcal$ \cite{nico-saor-zvon-19,ange-mark-vito-17}.

By construction, a cosilting $t$-structure $\tau_c=(\Ucal,\Vcal)$ 
is smashing; moreover, $\tau_c$ is also nondegenerate.
Lastly, the heart $\Hcal_c$ has enough injectives, and the cohomology functor
induces an equivalence $H^0_c\colon \Prod(c)\simeq \Inj(\Hcal_c)$
\cite[Lemma~2.8]{ange-mark-vito-17}. The following
theorem shows that among nondegenerate smashing $t$-structures, having a
Grothendieck heart characterises those
associated to \emph{pure-injective} cosilting objects.

\begin{thm}[{\cite[Thm.~3.6]{ange-mark-vito-17}}]\label{thm:pure-inj-tstr}
	Let $\Dcal$ be a compactly generated triangulated category.
	Then the following are equivalent for a $t$-structure $\tau$ in $\Dcal$:
	\begin{enumerate}
		\item $\tau$ is nondegenerate and smashing, and $\Hcal_\tau$ is a Grothendieck category;
		\item $\tau=\tau_c$ for a pure-injective cosilting object $c\in\Dcal$.
	\end{enumerate}
\end{thm}

A cosilting object $c\in\Dcal$ is \emph{cotilting} if
$\Prod(c)\subseteq \Hcal_c$ or, equivalently, if $\Prod(c)\subseteq{}^{\bot_{<0}}c$.
This is the case if and only if $\tau_c$ has derived type \cite[Cor.~5.2]{psar-vito-18}.

\subsubsection{Cosilting mutation}

The main topic of this section is the effect on the Gabriel spectrum of
\emph{cosilting mutation}, whose definition we are now going to recall. For more
background, we refer to \cite{ange-laki-stov-vito-22}.
Let $\Dcal$ be a compactly generated triangulated
category.

\begin{thm}[{\cite[Thm.~4.3]{ange-laki-stov-vito-22}}]
	Let $\tau=(\Ucal,\Vcal)$ be a nondegenerate, smashing $t$-structure with
	Grothendieck heart $\Hcal_\tau$ in $\Dcal$. For a torsion pair
	$t=(\Tcal,\Fcal)$ in $\Hcal_\tau$, the HRS-tilt $\tau_t$ is nondegenerate and
	smashing. Moreover, the following are equivalent:
	\begin{enumerate}
		\item $t$ is of finite type;
		\item $\tau_t$ has Grothendieck heart.
	\end{enumerate}
\end{thm}

In case $t$ is a \emph{hereditary} torsion pair of finite type, we give this
situation a name.

\begin{dfn}
	Let $\tau=(\Ucal,\Vcal)$ be a nondegenerate, smashing $t$-structure with
	Grothen\-dieck heart $\Hcal$. If $t=(\Tcal,\Fcal)$ is a hereditary torsion pair
	of finite type, the HRS-tilt $\tau_t$ is called the \emph{right mutation of
	$\tau$ at $t$}.
\end{dfn}

Observe that, using Theorem~\ref{thm:pure-inj-tstr}, right mutation of
$t$-structures can be
considered from the point of view of the pure-injective cosilting objects
associated to $\tau$ and $\tau_t$. If $c\in\Dcal$ is pure-injective cosilting and
$\tau=\tau_c$, as recalled we have an equivalence
$H_c^0\colon\Prod(c)\simeq\Inj(\Hcal_c)$. Therefore, the assignment $\Ecal\mapsto
t_\Ecal:=\left({}^\bot H_c^0(\Ecal),\Cogen(H_c^0(\Ecal))\right)$ gives a bijective correspondence between
subcategories $\Ecal\subseteq\Prod(c)$ closed under products and direct summands
(in $\Dcal$) and hereditary torsion pairs in $\Hcal_c$. We have the following
theorem.

\begin{thm}[\cite{ange-laki-stov-vito-22}]\label{thm:cosilting-mutation}
	The following are equivalent, for $\Ecal\subseteq\Prod(c)$ such that
	$\Prod(\Ecal)=\Ecal$:
	\begin{enumerate}[label=(\alph*)]
		\item $t_\Ecal$ is of finite type;
		\item $\Ecal$ is covering in $\Prod(c)$.
	\end{enumerate}
	In case (a) and (b) hold, we have the following.
	\begin{enumerate}
		\item The object $c':=e_0\oplus e_1$ defined by the distinguished triangle:
			\[\begin{tikzcd}[sep=1.5pc] e_1 \arrow{r} & e_0 \arrow{r}{\phi} & c
				\arrow{r} & e_1[1], \end{tikzcd}\]
				where $\phi$ is a $\Ecal$-cover, is pure-injective cosilting.
		\item The $t$-structure associated to $c'$ is the right mutation
			$\tau_{t_\Ecal}$ of $\tau$ at $t_\Ecal$.
		\item We have $\Prod(c)\cap\Prod(c')=\Ecal$.
		\item There is a bijection $\Theta_\Ecal\colon \Ind(\Prod(c))\to
			\Ind(\Prod(c'))$, given by:
			\[u \mapsto \begin{cases}
				u & \text{if }u\in\Ecal\\ u' &\text{otherwise, for }
				\begin{tikzcd}[sep=1.5pc] u'\arrow{r} & e \arrow{r}{\phi_u} & u
				\arrow{r} & u'[1]\end{tikzcd} \text{ with }\phi_u\text{ an
				$\Ecal$-cover}.
			\end{cases}\]
		These triangles are often referred to as \emph{exchange triangles}.
	\end{enumerate}
\end{thm}

\begin{proof}
	The equivalence between (a) and (b) is a rephrasing of
	\cite[Prop.~4.5(1--3)]{ange-laki-stov-vito-22}; item (1) is
	\cite[Thm.~4.9]{ange-laki-stov-vito-22}; items (2--3) are
	\cite[Thm.~3.5]{ange-laki-stov-vito-22}; item (4) is
	\cite[Prop.~4.12]{ange-laki-stov-vito-22}.
\end{proof}

\begin{rmk}\label{rmk:left-mutation}
	Both for nondegenerate, smashing $t$-structures with Grothendieck heart
	$\tau$ and for pure-injective cosilting objects $c$ there is a dual notion of
	\emph{left mutation}. We will only work with right mutation, using that $\tau$
	is a left mutation of $\tau'$ if and only if $\tau'$ is a right mutation of
	$\tau$ (and similarly for $c,c'$) \cite[Cor.~3.7]{ange-laki-stov-vito-22}.
\end{rmk}

\subsection{Mutation and semi-noetherianity}\label{subsec:mutation-semi-noetherianity}

The first result we are going to prove is that mutation preserves
semi-noetherianity of the hearts.

Let $\tau=(\Ucal,\Vcal)$ be a nondegenerate, smashing $t$-structure with
Grothendieck heart $\Hcal$, and let $t=(\Tcal,\Fcal)$ be a hereditary torsion
pair of finite type in $\Hcal$. Denote by $\tau_t$ the right mutation of $\tau$
at $t$, and by $\Hcal_t:=\Fcal\ast\Tcal[-1]$ its Grothendieck heart.

We have already recalled that $(\Fcal,\Tcal[-1])$ is a torsion pair in $\Hcal_t$.
Since we are assuming that $t$ is hereditary, though, more is true.
Recall that a \emph{TTF-triple} in an
abelian category is a triple $(\Xcal,\Ycal,\Zcal)$ such that $(\Xcal,\Ycal)$ and
$(\Ycal,\Zcal)$ are torsion pairs.

\begin{lemma}[{\cite[Rmk.~2.15]{pavo-vito-21}}]
	There is a TTF-triple $(\Fcal,\Tcal[-1],\Fcal')$ in $\Hcal_t$.
\end{lemma}

The \emph{core} of a TTF-triple $(\Xcal,\Ycal,\Zcal)$ is the full subcategory
$\Xcal\cap\Zcal$. The core of a TTF-triple in a Grothendieck category is closed under extensions, images of morphisms and coproducts. For the TTF-triple above we have the
following statement.

\begin{lemma}\label{lemma:core=giraud}
	The core $\Fcal\cap\Fcal'$ of $(\Fcal,\Tcal[-1],\Fcal')$ in $\Hcal_t$ coincides
	with the Giraud subcategory $\Ccal:=\Tcal^{\bot_{0,1}}\subseteq\Hcal$ of $t$.
\end{lemma}

\begin{proof}
	We have: $\Fcal\cap\Fcal'=\{\,x\in\Fcal\mid
	0=\Hcal_t(\Tcal[-1],x)\simeq\Dcal(\Tcal,x[1])\simeq
	\Ext^1_\Hcal(\Tcal,x)\,\}$.
\end{proof}

Now, we can apply a well-known fact for TTF-triples.

\begin{prop}[{\cite[Prop.~I.1.3]{beli-reit-07}}]
	Let $(\Xcal,\Ycal,\Zcal)$ be a TTF-triple in an abelian category $\Acal$, and
	let $\Lsf\colon \Acal\to \Acal/\Ycal$ be the localisation functor. Then $\Lsf$
	induces an equivalence $\Lsf\colon \Xcal\cap\Zcal\simeq \Acal/\Ycal$.
\end{prop}

\begin{cor}[{cf.\ \cite[\S6.1]{ange-laki-stov-vito-22}}]\label{cor:mutation-semi-noetherianity}
	There is a diagram, where rows are localisation sequences:
	\[\begin{tikzcd}[row sep=1pc,column sep=1.5pc]
		\Tcal[-1] \arrow[hook]{r} \arrow[-]{d}{\simeq} & \Hcal_t \arrow{r} &
			\Hcal_t/\Tcal[-1] \arrow[-]{d}{\simeq} \\
		\Tcal \arrow[hook]{r} & \Hcal \arrow{r} & \Hcal/\Tcal.
	\end{tikzcd}\]
	In particular, $\Hcal$ is semi-noetherian if and only if $\Hcal_t$ is.
\end{cor}

\begin{proof}
	The shift functor gives an equivalence $\Tcal[-1]\simeq\Tcal$.
	For the right-hand side, we already know that the localisation functor $\Hcal\to\Hcal/\Tcal$ induces an
	equivalence between the Giraud subcategory $\Ccal$ of $t$ and $\Hcal/\Tcal$.
	By the results above, $\Ccal$, being the core of $(\Fcal,\Tcal[-1],\Fcal')$,
	is also equivalent to $\Hcal_t/\Tcal[-1]$. The statement about
	semi-noetherianity then follows from Proposition~\ref{prop:semi-noetherian-serre}.
\end{proof}

From the diagram in this corollary, one could use
Proposition~\ref{prop:semi-noetherian-serre} also to derive a bound for the
Gabriel dimension of $\Hcal_t$: \[\Gdim(\Hcal_t)\leq
\Gdim(\Tcal)\oplus\Gdim(\Hcal/\Tcal)\leq \Gdim(\Hcal)\oplus\Gdim(\Hcal).\] We
can improve it slightly, thanks to the following observation.

\begin{prop}\label{prop:ttf-simple}
	Let $\Acal$ be an abelian category, with a TTF-triple $(\Xcal,\Ycal,\Zcal)$.
	Denote by $\Kcal:=\Xcal\cap\Zcal$ its core. Observe that both $\Ycal$ and
	$\Kcal\simeq\Acal/\Ycal$ are abelian categories. Then:
	\begin{enumerate}
		\item the simple objects of $\Acal$ lie either in $\Ycal$ or in $\Kcal$;
		\item an object in $\Ycal$ is simple in $\Acal$ if and only if it is simple
			in $\Ycal$;
		\item an object in $\Kcal$ is simple in $\Acal$ if and only if it is simple
			in $\Kcal$;
	\end{enumerate}
\end{prop}

\begin{proof}
	(1) A simple object is either torsion or torsion-free, with respect to each
	torsion pair of $\Acal$. Therefore, if it does not belong to $\Ycal$, it
	belongs to $\Xcal\cap\Zcal=\Kcal$.

	(2) Follows easily from the fact that $\Ycal\subseteq\Acal$ is closed under subobjects and
	quotients.

	(3) We first prove that a morphism $f\colon c_1\to c_2$ in $\Kcal$ is monic
	in $\Kcal$ if and only if it is monic in $\Acal$. Let $\ker f$ be the kernel of
	$f$ in $\Acal$; observe that since it is a subobject of
	$c_1\in\Kcal\subseteq\Zcal$, it must also lie in $\Zcal$. Let $\Lsf\colon\Acal\to
	\Acal/\Ycal$ be the localisation functor. We have seen that $\Lsf$ induces an
	equivalence $\Lsf\colon \Kcal\simeq\Acal/\Ycal$; therefore, $f$ is monic in
	$\Kcal$ if and only if $\Lsf(f)$ is monic in $\Acal/\Ycal$. This is in turn the
	case if and only if $\ker f$ lies in $\Ycal(\cap\Zcal)$, that is, if and only
	if $\ker f=0$, which shows the claim. A similar argument applies to epimorphisms.

	We have shown that the inclusion $\Kcal\subseteq\Acal$ is a fully faithful
	functor which preserves monomorphisms (and epimorphisms). It follows that if
	an object $c\in\Kcal$ is not simple in $\Kcal$, that is, if it has nonzero
	proper subobjects in $\Kcal$, it is not simple in $\Acal$ either.

	For the converse, let $c\in\Kcal$ be simple in $\Kcal$. Consider a short exact
	sequence in $\Acal$
	\[0\to x\to c\to y\to 0.\]
	As before, since
	$x$ is a subobject of $c\in\Kcal\subseteq\Zcal$, we must have $x\in\Zcal$;
	similarly, $y\in\Xcal$. Since $\Lsf(c)$ is
	simple in $\Acal/\Ycal$ by the equivalence $\Lsf\colon\Kcal\simeq\Acal/\Ycal$, we
	have that either $\Lsf(x)=0$ or $\Lsf(y)=0$, that is, either $x\in\Ycal(\cap\Zcal)$ or
	$y\in\Ycal(\cap\Xcal)$. This means that either $x=0$ or $y=0$, which shows
	that $c$ is simple in $\Acal$.
\end{proof}

\begin{cor}
	Let $\Acal$ be a semi-noetherian Grothendieck category, with a TTF-triple
	$(\Xcal,\Ycal,\Zcal)$, with core $\Kcal=\Xcal\cap\Zcal$. Then, if $\Ycal\neq
	0$, we have:
	\[\Gdim(\Ycal)\leq \Gdim(\Acal)\leq \Gdim(\Ycal)\oplus\Gdim(\Kcal/\Kcal_0).\]
	In particular, if $\Kcal$ is semi-artinian, we have
	\[\Gdim(\Acal)=\Gdim(\Ycal).\]
\end{cor}

\begin{proof}
	Denote by
	$0=\Acal_{-1}\subseteq\Acal_0\subseteq\cdots\subseteq\Acal_\delta=\Acal$ the
	Gabriel filtration of $\Acal$. Write $\gamma$ for the Gabriel dimension $\Gdim(\Ycal)\geq 0$, and
	similarly denote by
	$0=\Ycal_{-1}\subseteq\Ycal_0\subseteq\cdots\subseteq\Ycal_\gamma=\Ycal$ the
	Gabriel filtration of $\Ycal$. One
	can easily show by induction that $\Ycal_\alpha\subseteq\Acal_\alpha$, for
	every ordinal $\alpha$; in particular, we have
	$\Ycal=\Ycal_\gamma\subseteq\Acal_\gamma$. Denote by $\Lsf\colon\Acal\to
	\Acal/\Ycal$ the localisation functor: we then have that $\Acal/\Acal_\gamma$
	is equivalent to the quotient of $\Acal/\Ycal\simeq\Kcal$ over
	$\Lsf(\Acal_\gamma)$. Moreover, since $\Acal_0\subseteq\Acal_\gamma$ already
	contains the simple objects of $\Kcal$ by the Proposition, then
	$(\Acal/\Ycal)_0=\Lsf(\Kcal_0)\subseteq \Lsf(\Acal_\gamma)$; hence $\Acal/\Acal_\gamma$ is
	actually equivalent to a quotient of $(\Acal/\Ycal)/(\Acal/\Ycal)_0\simeq
	\Kcal/\Kcal_0$. Now, we have:
	\[\Gdim(\Acal)=\gamma\oplus \Gdim(\Acal/\Acal_\gamma)\leq \gamma\oplus
	\Gdim(\Kcal/\Kcal_0)=\Gdim(\Ycal)\oplus \Gdim(\Kcal/\Kcal_0).\qedhere\]
\end{proof}

Applying this result to right mutation, we obtain the following estimate.

\begin{prop}\label{prop:gdim-inequality}
	Let $\tau$ be a nondegenerate, smashing $t$-structure with Grothendieck heart
	$\Hcal$, let $t=(\Tcal,\Fcal)\neq (0,\Hcal)$ be a (nontrivial) hereditary torsion pair of finite type,
	and let $\tau_t$ be the right mutation of $\tau$ at $t$, with heart
	$\Hcal_t=\Fcal\ast\Tcal[-1]$. Then we have:
	\[\Gdim(\Tcal)\leq \Gdim(\Hcal_t)\leq
	\Gdim(\Tcal)\oplus\Gdim((\Hcal/\Tcal)/(\Hcal/\Tcal)_0).\]
	In particular, if $\Hcal/\Tcal$ is semi-artinian, we have:
	\[\Gdim(\Hcal_t)=\Gdim(\Tcal)\leq \Gdim(\Hcal).\]
\end{prop}

\begin{proof}
	If $\Ccal:=\Tcal^{\bot_{0,1}}$ denotes the Giraud subcategory of $t$ in
	$\Hcal$, we have seen in Lemma~\ref{lemma:core=giraud} that $\Ccal$ is the core of the TTF-triple
	$(\Fcal,\Tcal[-1],\Fcal')$ in $\Hcal'$. The claim follows by applying the
	corollary, using that $\Tcal[-1]\simeq\Tcal$ and $\Ccal\simeq\Hcal/\Tcal$.
\end{proof}

\begin{cor}\label{cor:semi-artinian-mutation}
	Right-mutation of semi-artinian hearts is semi-artinian.
\end{cor}

\subsection{Mutation gives piecewise
homeomorphisms}\label{subsec:piecewise-homeomorphism}

By Theorem~\ref{thm:cosilting-mutation}(4), if two hearts are linked by
mutation, their Gabriel spectra are in bijection. In this section we want to
compare their Gabriel topologies.

We fix our setting. Let $\Dcal$ be a compactly generated triangulated category.
Let $c\in\Dcal$ be a pure-injective cosilting object, and $\tau:=\tau_c$ its associated
nondegenerate, smashing $t$-structure with Grothendieck heart $\Hcal:=\Hcal_c$.
Denote by $H^0\colon \Dcal\to \Hcal$ the cohomology functor of $\tau$.
Motivated by the equivalence $H^0\colon
\Prod(c)\simeq \Inj(\Hcal)$, we first give the following definition.

\begin{dfn}
	The \emph{Gabriel spectrum} of a pure-injective cosilting object $c\in\Dcal$
	is the set $\GSpec(c):=\Ind(\Prod(c))$, endowed with the \emph{Gabriel
	topology}, whose closed sets are
	the $V\subseteq\Ind(\Prod(c))$ such that $V=\Ind(\Prod(V))$.
\end{dfn}

By construction, $H^0$ gives a homeomorphism $\GSpec(c)\simeq\GSpec(\Hcal)$ (see
Lemma~\ref{lemma:topology-inj}).

Now, following Theorem~\ref{thm:cosilting-mutation}, let
$\Ecal=\Prod(\Ecal)\subseteq\Prod(c)$ be a covering subcategory, and
$t:=t_\Ecal=:(\Tcal,\Fcal)$ the
associated hereditary torsion pair of finite type in $\Hcal$. Denote by $c'$ the
right mutation of $c$ at $\Ecal$ , and
by $\tau':=\tau_{c'}$ the associated $t$-structure, with heart
$\Hcal_t=\Hcal_{c'}=\Fcal\ast\Tcal[-1]$. From
Theorem~\ref{thm:cosilting-mutation}, we have a bijection:
\[\Theta:=\Theta_\Ecal\colon \GSpec(c) \to \GSpec(c'),\]
of which we are going to study the topological properties. We start with a
general result.

\begin{thm}\label{thm:piecewise-homeomorphism}
	Assume $\Ecal=\Prod(\Ind(\Ecal))$, which is the case if $\Hcal$ is
	semi-noetherian. Then $\Theta$
	induces homeomorphisms: \[ \GSpec(c)\cap\Ecal \simeq
	\GSpec(c')\cap\Ecal\qquad\text{and}\qquad
		\GSpec(c)\setminus \Ecal \simeq \GSpec(c')\setminus \Ecal\]
	where these subsets are endowed with the subspace topologies induced by the
	Gabriel topologies of $\GSpec(c)$ and $\GSpec(c')$.
\end{thm}

To ease the notation, we
write $E:=\GSpec(c)\cap \Ecal=\GSpec(c')\cap\Ecal$. We are assuming, as in the
Theorem, that $\Ecal=\Prod(E)$. The proof of the theorem will depend on the following lemma. 

\begin{lemma}\label{lemma:two-cases}
	Let $V\subseteq\GSpec(c)$ be any set. If either $V\subseteq E$ or $E\subseteq
	V$, then $V\subseteq\GSpec(c)$ is closed if and only if
	$\Theta(V)\subseteq\GSpec(c')$ is closed.
\end{lemma}

\begin{proof}
	\textit{Case $V\subseteq E$.} By definition, $V$ is closed if and only if
	$V=\Ind(\Prod(V))$; observe that this condition is intrinsic to
	$V\subseteq\Dcal$, it does not depend on either $c$ or $c'$.
	If $V\subseteq E$, we have $\Theta(V)=V$, and therefore the claim is trivial.

	\textit{Case $E\subseteq V$.}
	We are going to show that, in this case, for any
	$u\in\GSpec(c)$ we have that $u\in\Prod(V)$ if and only if $\Theta(u)\in\Prod
	\Theta(V)$; the claim will follow.

	Since we are assuming that $E\subseteq V$, and therefore $E=\Theta(E)\subseteq
	\Theta(V)$, the case $u\in E$ is trivial, as we will have both $u\in\Prod(V)$
	and $u=\Theta(u)\in\Prod(\Theta(V))$. We can therefore focus on the case
	$u\notin E$. By definition, $\Theta(u)=:u'$ is determined by the triangle:
	\[\begin{tikzcd}[sep=1.5pc]u[-1]\arrow{r} & u'\arrow{r}{\psi} & e \arrow{r}{\phi} & u, \end{tikzcd}\]
	where $\phi$ is an $\Ecal$-cover. Observe that since
	$\Dcal(u[-1],\Ecal)\subseteq\Dcal(\Prod(c)[-1],\Prod(c))=0$, the map $\psi$ is
	a $\Ecal$-preenvelope. We start by showing the implication $(\Rightarrow)$,
	that is, we assume that $u\in\Prod(V)$. Let $\{\,v_i\mid i\in
	I\,\}\subseteq V$ be such that there is a section $u\to \prod v_i$.
	For each $i\in I$, consider the triangle:
	\[\begin{tikzcd}[sep=1.5pc]
		v_i'\arrow{r}{\psi_i} & e_i \arrow{r}{\phi_i} & v_i \arrow{r} & v_i'[1]
	\end{tikzcd}\]
	where $\phi_i$ is an $\Ecal$-cover. This means that $v_i'=0$ if $v_i\in
	V\cap E$, while $v_i'=\Theta(v_i)$ if $v_i\in V\setminus E$. As
	before, $\psi_i$ is an $\Ecal$-preenvelope. Consider the
	product of these triangles, and the commutative diagram:
	\[\begin{tikzcd}
		u[-1] \arrow[shift left, hook]{d}{\oplus} \arrow{r} &
			u' \arrow[dotted, shift left]{d}{f} \arrow{r}[swap]{\psi} &
			e \arrow[dotted, shift left]{d} \arrow[dotted, bend right]{l}[swap]{h} \arrow{r}{\phi} &
			u \arrow[shift left, hook]{d}{\oplus} \arrow{r} &
			u'[1] \arrow[dotted, shift left]{d} \\
		\prod v_i[-1] \arrow[shift left, two heads]{u} \arrow{r} &
			\prod v_i' \arrow[dotted, shift left]{u}{g} \arrow{r}{\prod \psi_i} &
			\prod e_i \arrow[dotted, shift left]{u} \arrow{r}{\prod \phi_i} &
			\prod v_i \arrow[shift left, two heads]{u} \arrow{r} &
			\prod v_i' \arrow[dotted, shift left]{u}
	\end{tikzcd}\]
	Since a product of precovers is again a precover, $\prod\phi_i$ is an
	$\Ecal$-precover, as $\phi$ is. Hence, the section and the retraction lift to
	dotted morphisms between $e$ and
	$\prod e_i$. The dotted morphisms $f,g$ are then
	obtained by completing the commuting square to a morphism of triangles.
	Now, the composition $gf$ is not necessarily equal to the identity of $u'$;
	but the difference $1_{u'}-gf$ vanishes when precomposed with the morphism
	$u[-1]\to u'$ in the diagram, by an easy chase. Therefore, $1_{u'}-gf$ factors
	as $h\psi$, for some dotted morphism $h$. We then obtain that $[h,g]\colon
	e\oplus\prod v_i'\to u'$ is a retraction for the morphism ${}^t[\psi,f]\colon
	u'\to e\oplus\prod v_i'$, which is then split monic. Since its target belongs
	to
	$\Ecal\oplus\Prod(\Theta(V))=\Prod(E)\oplus\Prod(\Theta(V))=\Prod(\Theta(V))$, we
	conclude that $u'\in\Prod(\Theta(V))$, as wanted.

	For the converse implication $(\Leftarrow)$, the argument is analogous. The
	only notable difference is that instead of the $\Ecal$-precovers $\phi,\phi_i$, one
	needs to use the $\Ecal$-preenvelopes $\psi,\psi_i$; and it is not true that a
	product of preenvelopes is a preenvelope. In this case, though, $\prod\psi_i$
	is indeed an $\Ecal$-preenvelope, because its cocone has no morphisms to
	$\Ecal$, as argued earlier.
\end{proof}

\begin{proof}[Proof of Theorem~\ref{thm:piecewise-homeomorphism}]
	First, if $\Hcal$ is semi-noetherian then every class
	$\Ecal=\Prod(\Ecal)\subseteq\Prod(c)$ is such that $\Ecal=\Prod(\Ind(\Ecal))$ by
	Proposition~\ref{prop:gsupp}(1).

	Now, the case $V\subseteq E$ in the Lemma is showing precisely that the restriction
	of $\Theta$ to the subspace $E=\GSpec(c)\cap\Ecal$ is a homeomorphism. For the restriction
	to $\GSpec(c)\setminus\Ecal$, observe that a subset
	$U\subseteq\GSpec(c)\setminus E$ is closed with respect to the subspace
	topology if and only if $V:=E\cup U\subseteq\GSpec(c)$ is closed. Since this
	$V$ contains $E$, by the Lemma it is closed if and only if
	$\Theta(V)=\Theta(E)\cup\Theta(U)=E\cup\Theta(U)$ is closed in $\GSpec(c')$.
	This is equivalent to $\Theta(U)\subseteq\GSpec(c')\setminus E$ being closed
	for the subspace topology.
\end{proof}

With the following lemma, this result gives a consequence regarding
separation.

\begin{lemma}
	Let $X$ be a set, with two topologies $\Ocal_1$ and $\Ocal_2$. Assume that
	there is a subset $E\subseteq X$ such that:
	\begin{enumerate}[label=(\alph*)]
		\item $E$ is closed for both $\Ocal_1$ and $\Ocal_2$; and
		\item $\Ocal_1$ and $\Ocal_2$ induce the same subspace topologies on $E$ and
			$X\setminus E$.
	\end{enumerate}
	Then $\Ocal_1$ is $(\Tsf0)$ if and only if so is $\Ocal_2$.
\end{lemma}

\begin{proof}
	Without loss of generality, assume $\Ocal_1$ is $(\Tsf0)$. Let $x_1,x_2\in X$
	be points which are contained in each other's closure with respect to
	$\Ocal_2$. Since $E$ is closed for $\Ocal_2$, we must have either $x_1,x_2\in
	E$ or $x_1,x_2\in X\setminus E$. In a subspace topology, the closure is just
	the intersection of the subspace with the closure in the ambient space. Hence
	$x_1,x_2$ are in each other's closure with respect to the subspace topology
	induced by $\Ocal_2$. By (b), then they are in each other's closure with
	respect to $\Ocal_1$ as well. So we must have $x_1=x_2$, showing that
	$\Ocal_2$ is $(\Tsf0)$.
\end{proof}

\begin{cor}\label{cor:piecewise-T0}
	Let $c'$ be obtained from $c$ by right mutation. Then $\GSpec(c)$ is $(\Tsf0)$
	if and only so is $\GSpec(c')$.
\end{cor}

\begin{proof}
	We use the bijection $\Theta$ to identify the two sets $\GSpec(c)$ and
	$\GSpec(c')$, and take this as $X$; the two topologies $\Ocal_1$ and $\Ocal_2$
	are the corresponding Gabriel topologies. As $E$ we take the set at which we have mutated.
	Then (a) holds by Lemma~\ref{lemma:two-cases} and (b) holds
	by Theorem~\ref{thm:piecewise-homeomorphism}, so we conclude by the Lemma.
\end{proof}

Despite not giving a complete description of the topology of $\GSpec(c')$ in
terms of that of $\GSpec(c)$, Theorem~\ref{thm:piecewise-homeomorphism} provides
an \emph{upper bound}, as follows.

\begin{lemma}
	In the notation above, consider a subset $U\subseteq\GSpec(c)$.
	\begin{enumerate}
		\item If $\Theta(U)\subseteq\GSpec(c')$ is closed, then:
			\[U=(V_1\cap E)\sqcup(V_2\cap E^\cpl)\quad\text{for }V_1,V_2\subseteq\GSpec(c)
			\text{ closed.}\]
		\item If $U\subseteq\GSpec(c)$ is closed, then:
			\[U=(V_1\cap E)\sqcup (V_2\cap E^\cpl)\quad
			\text{for }\Theta(V_1),\Theta(V_2)\subseteq\GSpec(c')\text{ closed.}\]
	\end{enumerate}
\end{lemma}

\begin{proof}
	(1) The subspace topologies on $E, E^\cpl\subseteq\GSpec(c)$ have as closed subsets
	those of the form $V_1\cap E$ and $V_2\cap E^\cpl$, for
	$V_1,V_2\subseteq\GSpec(c)$ closed, respectively. 
	If $\Theta(U)\subseteq\GSpec(c')$ is closed, then
	$\Theta(U)\cap \Theta(E)$ and $\Theta(U)\cap\Theta(E^\cpl)$ are closed in the
	subspace topologies, and therefore they are of the form $\Theta(V_1\cap E),
	\Theta(V_2\cap E^\cpl)$, respectively.
	(2) is analogous.
\end{proof}

\begin{rmk}\label{rmk:disconnect}
	By Lemma~\ref{lemma:two-cases}, $\Theta(E)=E$ is closed in $\GSpec(c')$. As
	soon as it is also open (that is, its complement
	$\Theta(E^\cpl)\subseteq\GSpec(c')$ is closed), a subset
	$\Theta(U)\subseteq\GSpec(c')$ is closed \emph{if and only if} both
	intersections $\Theta(U)\cap E$ and $\Theta(U)\cap \Theta(E^\cpl)$ are closed in
	the subspace topologies. In other words, the converse implication of (1) above
	holds as well, yielding a complete description of the topology of
	$\GSpec(c')$ in terms of that of $\GSpec(c)$.
\end{rmk}

\subsection{When mutation disconnects the Gabriel
spectrum}\label{subsec:disconnect}

In this subsection we are going to present two cases in which
Remark~\ref{rmk:disconnect} applies, to give an explicit description of the
topology on $\GSpec(c')$. Throughout, we keep the notation of the previous subsection.

\subsubsection{Discrete mutation}

For a topological space $X$, we say that a $S\subseteq X$ is
a \emph{discrete subspace} if its subspace topology is discrete. Notice that an
open subset is a discrete subspace if and only if it consists of open points;
this is not the case for non-open subsets.
We will now consider the case in which the set $E\subseteq\GSpec(c)$ we tilt at is
a discrete subspace.
We are going to rely on the results of
\S\ref{subsec:mutation-semi-noetherianity}, and the results of this subsection
will be used in \S~\ref{subsec:zerodim-examples}.

\begin{lemma}
	The following are equivalent for $E\subseteq \GSpec(c)$:
	\begin{enumerate}
		\item $E$ is a discrete subspace;
		\item the quotient $\Hcal/{}^\bot E$ is semi-artinian;
	\end{enumerate}
\end{lemma}

\begin{proof}
	It follows from the homeomorphism of Proposition~\ref{prop:spectrum-quotients}(1)
	and the fact that a semi-noether\-ian category is semi-artinian if and only if
	its Gabriel spectrum is discrete.
\end{proof}

\begin{prop}\label{prop:discrete-mutation}
	If $E\subseteq\GSpec(c)$ is a discrete subspace, then
	$E=\Theta(E)\subseteq\GSpec(c')$ is a discrete open subset.
\end{prop}

\begin{proof}
	Denote by $t=(\Tcal,\Fcal)$ the hereditary torsion pair of finite type at which we
	are mutating, and by $\Hcal_t=\Fcal\ast\Tcal[-1]$ the heart of $c'$. Denote also by 
	$H_t^0\colon\Dcal\to \Hcal_t$ the cohomology with respect to $\Hcal_t$.
	If $e\in E=\Theta(E)$, we have that:
	\[\Hcal_t(\Tcal[-1],H_t^0(e))\simeq
	\Dcal(\Tcal[-1],e)\subseteq\Dcal(\Hcal,\Prod(c)[1])=0,\]
	which shows that $H_t^0(e)\in\GSpec(\Hcal_t)\cap \Fcal'$. In
	particular, it does not belong to $\Tcal[-1]$, and therefore it has a
	subobject $x\hookrightarrow H_t^0(e)$ in $\Fcal$. Since $\Fcal'$ is
	closed under subobjects, we have that $x\in\Ccal:=\Fcal\cap\Fcal'$. Recall
	from subsection~\S\ref{subsec:mutation-semi-noetherianity} that $\Ccal\simeq \Hcal/\Tcal$, and therefore
	it is semi-artinian by assumption. Hence $x$ has a simple subobject in $\Ccal$,
	which is also a simple subobject in $\Hcal_t$ by Proposition~\ref{prop:ttf-simple}.
	This shows that $H_t^0(e)$ is the injective envelope of a simple
	object of $\Hcal_t$. As such, it is an isolated point in $\GSpec(\Hcal_t)$.

	Translating back to $\GSpec(c')$, we have shown that $e$ is an isolated point
	of $\GSpec(c')$. It follows that $E=\bigcup_{e\in E}\{e\}\subseteq\GSpec(c')$
	is open, and discrete.
\end{proof}

\begin{cor}\label{cor:discrete-mutation}
	Assume that $E\subseteq\GSpec(c)$ is a discrete subspace. Then the following
	are equivalent, for $U\subseteq\GSpec(c)$:
	\begin{enumerate}
		\item $\Theta(U)$ is closed for $\GSpec(c')$;
		\item $U\cap E^\cpl=V\cap E^\cpl$ for $V\subseteq\GSpec(c)$ closed;
		\item $U\cup E$ is closed for $\GSpec(c)$.
	\end{enumerate}
\end{cor}

\begin{proof}
	Since $E=\Theta(E)\subseteq\GSpec(c')$ is clopen, by
	Remark~\ref{rmk:disconnect} a subset $\Theta(U)\subseteq\GSpec(c')$ is closed
	if and only if the intersections $\Theta(U)\cap E$ and $\Theta(U)\cap \Theta(E^\cpl)$
	are closed in the subspace topologies. Since $E$ is discrete, the first
	condition is always satisfied. Therefore $\Theta(U)$ is closed if and only if
	$U\cap E^\cpl=V\cap E^\cpl$ for $V\subseteq\GSpec(c)$ closed, which is (2).
	The equivalence $(2\Leftrightarrow3)$ follows from the fact that
	$E\subseteq\GSpec(c)$ is closed.
\end{proof}

\subsubsection{Strongly perfect mutation}

The second case is a generalisation of \cite[\S4.2]{pavo-vito-21}. Let
$t=(\Tcal,\Fcal)$ be a hereditary torsion pair in a Grothendieck category
$\Acal$, and let $\Ccal:=\Tcal^{\bot_{0,1}}$ be the Giraud subcategory. Recall
that $t$ is said to be \emph{perfect} if $\Ext^2(\Tcal,\Ccal)=0$ as well. By
dimension shifting, it is easy to see that this equivalent to the fact that
$\Ccal$ is closed under cokernels of monomorphisms. We need the following slightly stronger
notion.

\begin{dfn}\label{dfn:strongly-perfect}
	Let $\Dcal$ be a triangulated category, and $\Hcal$ a Grothendieck heart of
	$t$-structure. Let $t=(\Tcal,\Fcal)$ be a hereditary torsion pair in $\Hcal$,
	with Giraud subcategory $\Ccal\subseteq\Hcal$. Then $t$ is
	\emph{strongly perfect} if $\Dcal(\Tcal,\Ccal[2])=0$.
\end{dfn}

Observe that this notion is not intrinsic to $t$ in $\Hcal$, but it depends on
the embedding $\Hcal\subseteq\Dcal$. In general, we always have a monomorphism $\Ext_\Hcal^2(-,-)\hookrightarrow
\Dcal(-,-[2])$, making this condition stronger then
perfectness. It coincides with perfectness when $\Hcal$ is the heart of a
$t$-structure of derived type, such as the standard $t$-structure in a
derived category.

From now on, we resume the notation of \S\ref{subsec:piecewise-homeomorphism}
we have been using so far.

\begin{prop}
	Assume $t=(\Tcal,\Fcal)$ is strongly perfect in $\Hcal\subseteq\Dcal$. Then:
	\begin{enumerate}
		\item there is a \emph{TTF-quadruple} $(\Fcal,\Tcal[-1],\Ccal,\Rcal)$ in
			$\Hcal_t$;
		\item $\Ccal$ is a hereditary torsion class in
			$\Hcal_t$ with Gabriel support $\GSupp(\Ccal)=\GSupp(\Tcal[-1])^\cpl$, which
			is therefore open in $\GSpec(c')$.
	\end{enumerate}
\end{prop}

\begin{proof}
	(1) We already know that there is a TTF-triple $(\Fcal,\Tcal[-1],\Fcal')$ in
	$\Hcal_t$, and that $\Ccal=\Fcal\cap\Fcal'$. Now, for any object $x\in\Fcal'$, consider the torsion sequence with respect to 
	$(\Fcal,\Tcal[-1])$:
	\[ f\to x\to t[-1]\to f[1]\]
	Since $\Fcal'$ is closed under subobjects, we have $f\in\Ccal$. Hence,
	the connecting morphism $t[-1]\to f[1]$ lies in
	$\Dcal(t[-1],f[1])\simeq\Dcal(t,f[2])\subseteq\Dcal(\Tcal,\Ccal[2])=0$. The
	triangle then splits, forcing $t[-1]=0$ and showing that $x\simeq f$ lies in
	$\Ccal$.

	Now we have to show that the torsion-free class $\Ccal=\Fcal'$ is also a
	torsion class; it is enough to show that it is closed under quotients, and in
	fact, under cokernels of monomorphisms. Consider a short exact sequence in
	$\Hcal_t$:
	\[c_1\to c_2\to x\to c_1[1]\]
	with $c_1,c_2\in\Ccal\subseteq \Fcal$. Since $\Fcal$ is closed under
	quotients, we have $x\in\Fcal$. Hence this triangle is also a short exact
	sequence in $\Hcal$. Since $\Ccal\subseteq\Hcal$ is closed under cokernels of
	monomorphisms, as mentioned above, we conclude that $x\in\Ccal$.

	(2) Since $\Ccal$ is simultaneously a hereditary torsion class and the
	torsion-free class of a hereditary torsion pair, it is \emph{stable}, that is,
	closed under injective envelopes. Therefore, its Gabriel support consists of
	the indecomposable injectives it contains. This set is precisely the
	complement of the Gabriel support of $\Tcal[-1]$.
\end{proof}

Let $c\in\Dcal$ be a pure-injective cosilting object, with heart $\Hcal$, and
let $t=(\Tcal,\Fcal)$ be a hereditary torsion pair of finite type in $\Hcal$.
Let $c'$ be the right mutation of $c$ at $t$.

\begin{cor}\label{cor:strongly-perfect}
	Assume that $t$ is strongly perfect in $\Hcal\subseteq\Dcal$ (which is the
	case, for example, when
	$c$ is cotilting and $t$ is perfect). Then the following are
	equivalent, for $U\subseteq\GSpec(c)$:
	\begin{enumerate}
		\item $\Theta(U)$ is closed for $\GSpec(c')$;
		\item $U$ has the form $U=(V_1\cap E)\sqcup(V_2\cap E^\cpl)$ for
			$V_1, V_2\subseteq\GSpec(c)$ closed.
	\end{enumerate}
\end{cor}

\begin{proof}
	Apply Remark~\ref{rmk:disconnect}.
\end{proof}

Under certain extra assumptions, we can prove that a discrete mutation
is actually perfect. This result, which links the two cases we have considered in
this subsection, will be applied in \S\ref{sec:concrete-computations}. Let
$t=(\Tcal,\Fcal)$ be a hereditary torsion pair of
finite type in $\Hcal$ and denote again by $\Rsf: \Gcal/\Tcal \to \Gcal$ the right
adjoint to the quotient functor. Then $t$ is perfect precisely when $\Rsf$ is
(right-)exact. Following \cite{stov-24}, we say that $t$ is \emph{Giraud-finite}
if the Giraud subcategory $\Ccal = \Tcal^{\bot_{0,1}}$ is closed under direct
limits in $\Hcal$, or equivalently, if $\Rsf$ preserves direct limits. This implies
that $t$ is of finite type by \cite[Lemma~2.4]{krau-97}.

\begin{prop}\label{prop:semi-artinian-perfect}
	Let $t=(\Tcal,\Fcal)$ be a hereditary torsion pair of finite type in $\Hcal$ such that:
	\begin{enumerate}[label=(\alph*)]
		\item $\Gcal/\Tcal$ is semi-artinian;
		\item for every simple object $s\in\Gcal/\Tcal$ we have
			$\Ext_\Hcal^2(\Tcal,\Rsf(s)) = 0$;
		\item $t$ is Giraud-finite.
	\end{enumerate}
	Then $t$ is perfect.
\end{prop}

\begin{proof}
	Unravelling the last two definitions, in (c) we assume $\Rsf$ to commute with direct limits, and
	we want to prove that it is right exact. That is,
	for every epimorphism $f\colon y\to z$ in $\Gcal/\Tcal$
	want to show that $\Rsf(f)$ is epic, or equivalently that $t:=\coker(\Rsf(f))\in\Tcal$ is zero.

	By (a), $x:=\ker f$ admits
	a transfinite filtration with
	simple factors; that is, a filtration $0=:x_{-1}\subset x_0\subset x_1\subset\cdots
	\subset x_\lambda=x$, indexed over an ordinal $\lambda$, such that the factors
	$x_{\alpha+1}/x_\alpha$ are simple for every $\alpha<\lambda$ and we have
	$\varinjlim_{\beta<\alpha}x_\beta=x_\alpha$ for every limit ordinal
	$\alpha=\bigcup_{\beta<\alpha}\beta\leq \lambda$. We fix such a filtration so
	that $\lambda$ is minimal; this ordinal is the \emph{length} of $x$, and we
	argue by induction on it. Observe that, by minimality, every subobject
	$x_\alpha$ has length $\alpha$.
	
	If $\lambda=0$, then $x$ is simple, and by (b) we have $\Rsf(x)\in\Tcal^{\bot_2}$.
	From the long sequence of $\Ext$, we deduce that
	the image of $\Rsf(f)$ belongs to $\Tcal^{\bot_{0,1}}$, and therefore $t$ must split as a summand
	of $\Rsf(z)$, which forces $t=0$.

	If $\lambda=\alpha+1$ is a successor ordinal, we have a short exact
	sequence $0\to x_\alpha\to x\to s\to 0$, with $x_\alpha$ of length $\alpha$ and $s$ simple.
	Consider the pushout diagram in $\Gcal/\Tcal$, on the left:
	\[\begin{tikzcd}[sep=1.5pc]
		& 0 \arrow{d} & 0 \arrow{d} \\
		& x_\alpha \arrow[equal]{r} \arrow{d} & x_\alpha \arrow{d} && \text{in }\Gcal/\Tcal,\\
		0 \arrow{r} & x \arrow[phantom,description,very near end, "\ulcorner"]{dr}
			\arrow{r} \arrow{d} & y \arrow{r}{f} \arrow{d} & z \arrow[equal]{d} \arrow{r} & 0 \\
		0 \arrow{r} & s \arrow{r} \arrow{d} & y' \arrow{r} \arrow{d} & z \arrow{r} & 0\\
		& 0 & 0
	\end{tikzcd}\qquad
	\begin{tikzcd}[sep=1.5pc]
		& 0 \arrow{d} & 0 \arrow{d} \\
		& \Rsf(x_\alpha) \arrow[equal]{r} \arrow{d} & \Rsf(x_\alpha) \arrow{d} && \text{in }\Gcal.\\
		0 \arrow{r} & \Rsf(x) \arrow{r} \arrow{d} & \Rsf(y) \arrow{r}{\Rsf(f)}
		\arrow{d} & \Rsf(z) \arrow[equal]{d} \\
		0 \arrow{r} & \Rsf(s) \arrow{r} \arrow{d} & \Rsf(y') \arrow{r} \arrow{d} & \Rsf(z) \arrow{r} & 0\\
		& 0 & 0
	\end{tikzcd}\]
	Applying $\Rsf$ we obtain the diagram in $\Gcal$ on the right.
	Since $s$ and $x_\alpha$ have length less than $\alpha+1$, by induction the two
	columns and the bottom row are exact.
	It follows that $\Rsf(f)$ is epic in $\Gcal$.

	For the last case, assume that $\lambda=\bigcup_{\alpha<\lambda}$ is a limit
	ordinal. For each
	$\alpha<\lambda$, consider the short exact sequence $0\to x_\alpha\to x\to
	x_\alpha'\to0$ in $\Gcal/\Tcal$, and the pushout diagram on the left:
	\[
		\begin{tikzcd}[sep=1.5pc]
			& 0 \arrow{d} & 0 \arrow{d} \\
			& x_\alpha \arrow[equal]{r} \arrow{d} & x_\alpha \arrow{d} && \text{in
			}\Gcal/\Tcal,\\
			0 \arrow{r} & x \arrow[phantom,description,very near end, "\ulcorner"]{dr}
				\arrow{r} \arrow{d} & y \arrow{r}{f} \arrow{d} & z \arrow[equal]{d} \arrow{r} & 0 \\
			0 \arrow{r} & x_\alpha' \arrow{r} \arrow{d} & y_\alpha'
			\arrow{r}{f_\alpha} \arrow{d} & z \arrow{r} & 0\\
			& 0 & 0
		\end{tikzcd}
		\qquad
		\begin{tikzcd}[sep=1.5pc]
			& 0 \arrow{d} & 0 \arrow{d} \\
			& \Rsf(x_\alpha) \arrow[equal]{r} \arrow{d} & \Rsf(x_\alpha) \arrow{d} && \text{in
			}\Gcal.\\
			0 \arrow{r} & \Rsf(x) \arrow{r} \arrow{d} & \Rsf(y) \arrow{r}{\Rsf(f)} \arrow{d} & \Rsf(z)
				\arrow[equal]{d} \\
			0 \arrow{r} & \Rsf(x_\alpha') \arrow{r} \arrow{d} & \Rsf(y_\alpha')
			\arrow{r}{\Rsf(f_\alpha)} \arrow{d} & \Rsf(z) \\
			& 0 & 0
	\end{tikzcd}
	\]
	Applying $\Rsf$, since $x_\alpha$ has length less than $\lambda$, by induction we
	obtain the exact diagram in $\Gcal$ on the right.
	Taking the direct limit of the
	diagram on the left we obtain by construction that $\varinjlim_\alpha
	x_\alpha'=0$, so that $\varinjlim_\alpha f_\alpha$ is an isomorphism by
	exactness of direct limits in $\Gcal/\Tcal$.
	Since we
	are assuming that $\Rsf$ preserves direct limits, we
	deduce that $\varinjlim_\alpha \Rsf(f_\alpha)=\Rsf(\varinjlim_\alpha f_\alpha)$
	is an isomorphism as well. Taking the direct limit in $\Gcal$ of the
	diagram on the right, we conclude that $\Rsf(f)$, being the composition of an
	isomorphism with an epimorphism, is an epimorphism.
\end{proof}

\subsection{Mutation of injectives}

Now we present another description of the bijection $\Theta$ induced by
right mutation, this time in terms of the Gabriel spectra of the hearts, rather
than those of the cosilting objects.

Let $c\in\Dcal$ be a pure-injective cosilting object, and let
$\tau=(\Ucal,\Vcal)$ be its $t$-structure, with heart $\Hcal$. Let
$\Ecal\subseteq\Prod(c)$ be a class suitable for right mutation, and let
$t=(\Xcal,\Ycal)$ be the corresponding hereditary torsion pair of finite type in
$\Hcal$. Let $c'$ be the right mutation, with heart
$\Hcal_t:=\Ycal\ast\Xcal[-1]$. The equivalences $\Prod(c)\simeq\Inj(\Hcal)$ and
$\Prod(c')\simeq\Inj(\Hcal_t)$ allow us to define (by abuse of notation) a bijection
$\Theta\colon\GSpec(\Hcal)\to \GSpec(\Hcal_t)$. Namely, for $e\in\GSpec(c)$,
there is a unique $u\in\GSpec(c)$ such that $e=H_c^0(u)$; we set
$\Theta(e):=H^0_t(\Theta(u))$. We now give a more explicit description of this
bijection.

\begin{prop}\label{prop:theta-GSupp(H)}
	For any $e\in\GSpec(\Hcal)$, we have that:
	\begin{enumerate}
		\item if $e\in\Ycal$, then $e$ is the torsion part of $\Theta(e)$ for the pair $(\Ycal,\Xcal[-1])$ in $\Hcal_t$;
		\item if $e\notin\Ycal$, then $t(e)[-1]$ is the torsion part of $\Theta(e)$
			for to the pair $(\Xcal[-1],\Ycal')$ in $\Hcal_t$.
	\end{enumerate}
	In particular, $\Theta(e)$ is the injective envelope in $\Hcal_t$ of this
	subobject, in both cases.
\end{prop}

\begin{proof}
	By construction, $\Theta(e)$ is an indecomposable injective object of
	$\Hcal_t$; therefore it is the injective envelope of its nonzero subobjects,
	which shows the last claim.

	(1) If $e$ belongs to $\Ycal$, then it is of the form $e=H_c^0(u)$ for some
	$u\in\Ecal$, and so we have
	$\Theta(u)=u$. It is easy to check that $e=H_c^0(u)$ is the
	torsion part of $H^0_t(u)=:\Theta(e)$ with respect to the pair
	$(\Ycal,\Xcal[-1])$ in $\Hcal_t$.

	(2) If $e\notin\Ycal$, we have $\Theta(e):=H^0_t(u')$, for the exchange triangle
	\[u'\to f\to u\to u'[1]\]
	extending an $\Ecal$-cover $f\to u$. Observe that $u$ may have an
	$\Hcal_t$-cohomology in degree $-1$: this is precisely $H^{-1}_t(u)\simeq
	t(H_c^0(u))[-1]=t(e)[-1]$. From the long exact sequence of
	$\Hcal_t$-cohomology of the triangle above we then deduce that $t(e)[-1]$ is
	a subobject of $H^0_t(u')=:\Theta(e)$ in $\Hcal_t$, as wanted. Now, to show
	that it is the torsion part, consider any morphism $x[-1]\to H^0_t(u')$, for
	$x\in\Xcal$: this gives a morphism $x[-1]\to u'$, which then corresponds to a
	morphism $x\to u'[1]$. Since the composition $x\to u'[1]\to f[1]$ vanishes, as
	$\Hcal(-,f[1])\subseteq\Hcal(-,\Prod(c)[1])=0$, we deduce that the map $x\to
	u'[1]$ factors through $u\to u'[1]$. This means that the morphism $x[-1]\to
	H^0_t(u')$ factors through $H^0_t(u[-1])=t(e)[-1]$.
\end{proof}

\section{The case of commutative noetherian rings}
\label{sec:commutative-noetherian}

In this section, we specialise to $\Dcal=\D(R):=\D(\Mod(R))$, the derived category
of $\Mod(R)$, for a commutative noetherian ring $R$. The standard $t$-structure,
whose heart is $\Mod(R)$, is the one associated to the injective cogenerator
$\omega$ of $\Mod(R)$, which is a pure-injective cosilting (in fact,
cotilting) object. We will only
consider pure-injective cosilting objects which (up to shift) can be obtained from $\omega$ by a
finite chain of mutation operations (in either direction, see
Remark~\ref{rmk:left-mutation}). It is true in general that such an object will
be quasi-isomorphic to a bounded complex of injectives (this can be easily
derived from the construction of the right mutation in
Theorem~\ref{thm:cosilting-mutation}(1)). In the commutative noetherian setting,
moreover, the converse is also true by \cite[Prop.~6.2]{pavo-vito-21}: a
cosilting object can be obtained from $\omega$ by a finite chain of mutation
operations (and in fact, up to shift, by a finite chain of right mutations) if
and only if it is quasi-isomorphic to a bounded complex of injectives. These
objects are known in the literature as \emph{bounded cosilting complexes}
\cite[Def.~4.15 and Prop.~4.17]{psar-vito-18}, and they are always
pure-injective \cite[Prop.~3.10]{mark-vito-18}. We point out that in $\D(R)$, the
$t$-structures associated to a bounded cosilting complex are precisely the
\emph{intermediate compactly generated} $t$-structures (combining
\cite[Thm~1.1]{hrbe-naka-21} and \cite[Thm.~4.6]{laki-20}). Many of the references
we are going to use are formulated in this equivalent language.

We immediately observe that, since locally noetherian Grothendieck categories
are semi-noetherian \cite[Prop.~IV.7]{gabr-62}, this is the case for $\Mod(R)$,
and therefore for the hearts associated to bounded cosilting complexes, by
Corollary~\ref{cor:mutation-semi-noetherianity}. In the following, we will
implicitely use this fact to apply results from the first section, without mentioning
it every time.
In particular, using Lemma~\ref{lemma:no-superdecomposables} and the equivalence
$H^0_c\colon\Prod(c)\simeq\Inj(\Hcal_c)$, we have the following statement.

\begin{cor}
	For any bounded cosilting complex $c\in\D(R)$, there are no superdecomposable
	objects in $\Prod(c)$.
\end{cor}

This is the case despite the fact that there may very well be superdecomposable
pure-injective objects in $\D(R)$ \cite[Summary~8.17]{jens-lenz-89}.

In this context, we will be able to improve on
Theorem~\ref{thm:piecewise-homeomorphism}, showing that the map $\Theta$
underlying a right mutation between such cosilting objects is
an open map between their Gabriel spectra. The structure of this section is as
follows.
\begin{enumerate}[label=\S\textbf{3.\arabic*})]
	\item We start by showing that there is a canonical bijection between
		$\GSpec(c)$, for $c$ a bounded cosilting complex, and the prime spectrum
		$\Spec R$.
	\item We then show that suitable shifts of the residue fields are the
		Gabriel-simples in the hearts of bounded cosilting complexes.
	\item We use the previous subsections to deduce some topological properties of
		$\GSpec(c)$.
	\item Finally, we prove the desired theorem about right mutation inducing open
		maps.
\end{enumerate}

\subsection{Matlis' correspondence for cosilting complexes}\label{subsubsec:GSpec-Spec}

It is a classical theorem of Matlis \cite[Thm.~2.4]{matl-58} that the Gabriel
spectrum of $\Mod(R)$ (which coincides with $\GSpec(\omega)$) is in bijection with the prime spectrum
$\Spec(R)$. The bijection associates to any prime $\pf\in\Spec(R)$ the injective
envelope $\omega(\pf):=E(R/\pf)\in\GSpec(\omega)$ of the cyclic module $R/\pf$.

By Theorem~\ref{thm:cosilting-mutation}(4), by composing the maps $\Theta$
underlying each mutation, the Gabriel spectra of the pure-injective cosilting
objects $c$ which can be obtained from $\omega$ by a finite chain of mutations
are in bijection with $\GSpec(\omega)\simeq\Spec(R)$. Outside of the commutative
noetherian setting, these bijections are not canonically defined, but
rather they would depend on the chain of mutations used to link $c$ and
$\omega$, as demonstrated in the following example.

\begin{ex}\label{ex:different-bijections}
	Fix a field $k$, and consider the hereditary finite-dimensional algebra
	$\Lambda=k\Abb_2$. The
	indecomposable objects of $\Mod\Lambda$ (and, up to shift, the only indecomposable objects of $\D(\Lambda)$) are the two simples $s_2$ and
	$s_1$ (projective and injective respectively) and the projective-injective
	$p$, sitting in a short exact sequence $0\to s_2\to p\to s_1\to 0$. Let
	$\omega=p\oplus s_1$ be the injective cogenerator of $\Mod \Lambda$, which is
	a cotilting object of $\D(\Lambda)$. We have two routes to obtain the cosilting object
	$\omega[-1]$ from $\omega$ by iterated
	right mutation. One way is the trivial mutation at $0$, so that every summand
	is shifted: hence $\Theta(s_1)=s_1[-1]$ and $\Theta(p)=p[-1]$. Another way is
	to mutate first at $p$, to obtain $\Theta_1(p)=p$ and $\Theta_1(s_1)=s_2$;
	then at $s_2$, to obtain $\Theta_2(s_2)=s_2$ and
	$\Theta_2(p)=s_1[-1]$; lastly at $s_1[-1]$, obtaining
	$\Theta_3(s_1[-1])=s_1[-1]$ and $\Theta_3(s_2)=p[-1]$.

	Hence $\Theta_3\Theta_2\Theta_1(s_1)=p[-1]$ and
	$\Theta_3\Theta_2\Theta_1(p)=s_1[-1]$: these two different chains of
	right mutations have yielded different bijections $\Theta$ and
	$\Theta_3\Theta_2\Theta_1$ between $\GSpec(\omega)$ and $\GSpec(\omega[-1])$.
\end{ex}

In this subsection we show that, remarkably, in the commutative noetherian
setting this phenomenon does not occurr: any chain of mutations from $\omega$ to
$c$ yields the same bijection between $\GSpec(\omega)\simeq\Spec(R)$ and
$\GSpec(c)$.

Recall that a \emph{localising subcategory} of a triangulated category is a
triangulated subcategory closed under coproducts. In $\D(R)$, these are
classified in terms of prime ideals \cite{neem-92a}. For an object $x\in\D(R)$,
its \emph{support} is the set $\supp(x)=\{\,\pf\in\Spec(R)\mid x\otimes_R^\Lbb
k(\pf)\neq 0\,\}$, where $k(\pf):=R_\pf/\pf R_\pf$ denotes the residue field at
$\pf$. Then the localising subcategories of $\D(R)$ are those of the form
$\Lcal_U:=\supp^{-1}(U)$, for a set $U\subseteq\Spec(R)$
\cite[Thm.~2.8]{neem-92a}. From the classification it follows that these
subcategories are (triangulated) aisles of stable $t$-structures
$(\Lcal_U,\Bcal_U:=\Lcal_U^\bot)$. Indeed, $\Lcal_U$ is the smallest localising
subcategory containing a set of objects $\{\,k(\pf)\mid \pf\in U\,\}$, and
therefore it is the aisle of a $t$-structure by \cite[Thm.~2.3]{neem-21}.
Recall that a
subset $W\subseteq\Spec(R)$ is \emph{specialisation-closed} if $\pf\subseteq\qf$
and $\pf\in W$ imply $\qf\in W$. In other words, these are the upper-sets with
respect to the inclusion order $\subseteq$. Via the classification, the
specialisation-closed sets
$W\subseteq\Spec(R)$ correspond to smashing stable $t$-structures:
\[(\Lcal_W,\Bcal_W):=(\supp^{-1}W,\supp^{-1}W^\cpl)\]
by \cite[Thm~3.3]{neem-92a}. The description of $\Bcal_W$ is due to the fact
that if $\Lcal_W$ is smashing, then $\Bcal_W=\Lcal_{W^\cpl}$ is
itself localising.
Smashing localisations interact nicely with respect to the
$t$-structures associated to bounded cosilting complexes.

\begin{prop}\label{prop:smashing-cohomological}
	Let $c\in D(R)$ be a bounded cosilting complex, with heart $\Hcal$, and
	$W\subseteq\Spec(R)$ a specialisation-closed subset.
	Then $\Tcal_W:=\Lcal_W\cap\Hcal$ is a hereditary torsion class in $\Hcal$,
	and:
	\[\Lcal_W=\{\,x\in\D(R)\mid H_c^\ast(x)\in\Tcal_W\,\}.\]
\end{prop}

\begin{proof}
	The fact that $\Tcal_W$ is a hereditary torsion class is
	\cite[Thm.~4.5(1)]{pavo-vito-21}. Then by 
	\cite[Prop.~4.1(3)]{pavo-vito-21} the right-hand side of the equality is a
	localising subcategory of $\D(R)$ whose support is $W$. Therefore it coincides
	with $\Lcal_W$ by Neeman's classification \cite{neem-92a}.
\end{proof}

In the following we will denote the torsion-free class of $\Tcal_W$ by
$\Fcal_W:=\Tcal_W^\bot\subseteq\Hcal$. To these torsion pairs we will apply the
following general observation.

\begin{lemma}\label{lemma:BW-FW}
	Let $\Dcal$ be a compactly generated triangulated category, and $c\in\Dcal$ be
	a pure-injective cosilting object, with heart $\Hcal$. Consider the localising subcategory
	$\Lcal=\{\,x\in\Dcal\mid H_c^\ast(x)\in\Lcal\cap\Hcal\,\}$ of $\Dcal$. Then, for any
	$u\in\GSpec(c)$, we have:
	\[u\in\Lcal^\bot\subseteq\Dcal \iff
	H_c^0(u)\in(\Lcal\cap\Hcal)^\bot\subseteq\Hcal.\]
\end{lemma}

\begin{proof}
	By the assumption on $\Lcal$, we have that $H_c^0(\Lcal)=\Lcal\cap\Hcal$. Then
	the claim follows from the natural isomorphism
	$\Dcal(-,u)\simeq\Hcal(H_c^0(-),H_c^0(u))$ and the assumption on $\Lcal$.
\end{proof}

Since we are mentioning the torsion pairs $(\Tcal_W,\Fcal_W)$, we pause for a
moment to show that
they are of finite type. This generalises \cite[Thm.~4.5(1.b)]{pavo-vito-21}, by
dropping the assumption that the $t$-structure is of derived type.
In fact, we prove that they are even Giraud-finite (see the paragraph above
Proposition~\ref{prop:semi-artinian-perfect}).

First, we need to briefly recall the notion of \emph{homotopy colimit} of a
directed \emph{coherent diagram}, \emph{i.e.} an object of $\D((\Mod R)^I)$, for a
directed small category $I$. 
In $\D(R)$, the homotopy colimit
is the left derived functor of the exact direct limit functor $\varinjlim$ of $\Mod(R)$.
Note that a direct system of shape $I$ in $\D(R)$ is an object of $D(R)^I$.
Therefore, we can only compute its homotopy colimit if it lies in the
essential image of the canonical functor $\D((\Mod R)^I)\to \D(R)^I$. In that
case, this directed homotopy colimit is computed as a termwise direct limit of
complexes.

The aisle of a $t$-structure is always closed under homotopy colimits of direct
systems; if the coaisle also has this property, the $t$-structure is said to
be \emph{homotopically smashing}.
Any smashing localisation is
homotopically smashing (see \cite[\S{}A.5]{hrbe-naka-21} and references
therein), so this is the case for $(\Lcal_W,\Bcal_W)$. The $t$-structure
associated to a pure-injective cosilting object is also homotopically
smashing \cite[Thm.~4.6]{laki-20}.
A direct system of objects in the heart of a homotopically smashing $t$-structure
always lifts to a coherent diagram, and its homotopy colimit coincides with the
direct limit computed in the heart \cite[Cor.~7.4]{saor-stov-viri-17}. Moreover,
the cohomology functor associated to such a $t$-structure commutes with directed
homotopy colimits.

\begin{prop}\label{prop:W-giraud-finite}
	Let $c\in\D(R)$ be a bounded cosilting complex with heart $\Hcal$. For
	any speciali\-sa\-tion-closed subset $W\subseteq\Spec(R)$, the torsion pair
	$(\Tcal_W,\Fcal_W)$ in $\Hcal$ is Giraud-finite.
\end{prop}

\begin{proof}
	First, we are going to prove that $(\Tcal_W,\Fcal_W)$ is of finite type. We
	show that $\Fcal_W$ is closed under taking reduced products of any set
	$\Scal:=\{\,e_i\mid i\in I\,\}\subseteq\Fcal_W\cap\Inj(\Hcal)$, and use 
	the characterisation of Corollary~\ref{cor:finite-type-inj}.
	Let $\tilde\Scal:=\{\,c_i\mid i\in I\,\}$ be the
	corresponding set of objects of $\Prod(c)$, through the equivalence
	$H_c^0\colon \Prod(c)\simeq\Inj(\Hcal)$. By Lemma~\ref{lemma:BW-FW}, we have
	that $\tilde\Scal\subseteq\Bcal_W$. Consider a filter $\Phi$ on $I$.
	By \cite[Prop.~2.7]{laki-20}, the reduced product diagram of $\tilde\Scal$ over
	$\Phi$ lifts to a \emph{coherent} reduced product diagram, which
	following the reference we denote by $\Red_{\Phi}(\tilde\Scal)$. Since
	$(\Lcal_W,\Bcal_W)$ and the cosilting $t$-structure of $c$ are homotopically
	smashing, by combining
	\cite[Thm.~4.6 and Thm.~3.12]{laki-20}, we obtain that both $\Bcal_W$ and the
	coaisle $\Vcal:={}^{\bot_{>0}}c$ are closed under reduced products, and
	therefore contain $\hocolim \Red_{\Phi}(\tilde\Scal)$. It is clear that
	$H^0_c(\Red_{\Phi}(\tilde\Scal))$ is the reduced product diagram of
	$\Scal$ over $\Phi$, and thus 
	we have $H^0_c(\hocolim\Red_{\Phi}(\tilde\Scal))\simeq
	\varinjlim H^0_c(\Red_{\Phi}(\tilde\Scal))=\prod\Scal/\Phi$. 

	We now have:
	\begin{align*}\textstyle
		&\Hcal\left(\Tcal_W,\prod\Scal/\Phi\right) \simeq
		\Hcal\left(\Tcal_W,H_c^0(\hocolim \Red_{\tilde\Phi}(\tilde\Scal))\right) \simeq \\
		\simeq \;&\Dcal\left(\Tcal_W,\hocolim \Red_{\tilde\Phi}(\tilde\Scal)\right)\subseteq
		\Dcal\left(\Lcal_W,\Bcal_W\right)=0,
	\end{align*}
	which shows that $\Prod\Scal/\Phi\in\Tcal_W^\bot=\Fcal_W$, as wanted.

	Now we prove Giraud-finiteness. By the first part, $(\Tcal_W,\Fcal_W)$ is
	a hereditary torsion pair of finite type; hence we can right-mutate $\Hcal$ at
	it, obtaining a Grothendieck heart $\Hcal_t:=\Fcal_W\ast\Tcal_W[-1]$. In this heart,
	Proposition~\ref{prop:smashing-cohomological} gives a hereditary torsion pair
	$(\Tcal_W',\Fcal_W')$, which is again of finite type by the previous paragraph. Observe that:
	\[\Tcal_W'=\Lcal_W\cap \Hcal_t=\{\,x\in\Dcal\mid
	H_c^\ast(x)\in\Tcal_W\,\}\cap\Hcal_t=\Tcal_W[-1].\]
	Now, we have proved that $\Fcal_W$ and $\Fcal_W'$ are closed under direct
	limits, in $\Hcal$ and $\Hcal_t$ respectively, and this is equivalent to the fact that
	they are closed under homotopy colimits in $\Dcal$.
	Since the Giraud subcategory
	$\Ccal_W:=\Tcal_W^{\bot_{0,1}}\subseteq\Hcal$ is equal 
	to the following orthogonal (in $\Dcal$):
	\[\Ccal_W=\Hcal\cap(\Tcal_W^\bot\cap(\Tcal_W[-1])^\bot)=
		\Hcal\cap(\Tcal_W^\bot\cap\Tcal_W'^\bot)=\Fcal_W\cap \Tcal_W'^\bot=\Fcal_W\cap\Fcal_W',\]
	we deduce that $\Ccal_W$ is closed under homotopy colimits, and, thus,
	under direct limits in $\Hcal$.
\end{proof}

\begin{rmk}
	By \cite[Thm.~4.5(2)]{pavo-vito-21}, these are the \emph{only} torsion pairs of
	finite type in the heart of \emph{restrictable} $t$-structures of $\D(R)$.
\end{rmk}

Now we turn back to our main goal. We are going to show that the subcategories
$\Bcal_W$ are `closed under mutation', in the sense of the following
proposition.

\begin{prop}\label{prop:mutation-BW}
	Let $c\in\D(R)$ be a bounded cosilting complex, and $c'$ a right mutation of $c$.
	Let $W\subseteq\Spec(R)$ be any specialisation-closed subset.
	Then $u\in\GSpec(c)$ belongs to $\Bcal_W$ if and only if so does
	$\Theta(u)\in\GSpec(c')$.
\end{prop}

\begin{proof}
	Let $\Hcal$ and $\Hcal_t$ be the hearts of $c$ and $c'$. We have
	$\Hcal_t=\Ycal\ast\Xcal[-1]$, for the hereditary torsion pair $t=(\Xcal,\Ycal)$
	in $\Hcal$ underlying the mutation. We denote by $H^0_c$ and $H^0_t$ the
	associated cohomology functors, respectively. Consider the torsion pairs
	$(\Tcal_W,\Fcal_W)$ in $\Hcal$ and $(\Tcal_W',\Fcal_W')$ in $\Hcal_t$ given by
	Proposition~\ref{prop:smashing-cohomological}. Observe that since as before:
	\[\Tcal_W'=\Lcal_W\cap\Hcal_t=\{\,x\in\D(R)\mid
	H_c^\ast(x)\in\Tcal_W\,\}\cap\Hcal_t,\]
	we have $\Tcal_W'=(\Ycal\cap\Tcal_W)\ast(\Xcal\cap\Tcal_W)[-1]$.

	Now let $u\in\GSpec(c)$, and consider $\Theta(u)\in\GSpec(c')$. If $u$ belongs
	to the class $\Ecal\subseteq\Prod(c)$ at which we have right-mutated, then
	$\Theta(u)=u$, and there is nothing to prove. Therefore, we assume that
	$u\notin\Ecal$. Using the characterisation of Lemma~\ref{lemma:BW-FW}, we
	translate the claim to injectives via cohomology. Let $e:=H_c^0(u)$ and
	$\Theta(e):=H^0_t(\Theta(u))$: we need to prove that $e\in\Fcal_W$ if and only
	if $\Theta(e)\in\Fcal_W'$. By the description of
	Proposition~\ref{prop:theta-GSupp(H)}, the latter is equivalent to
	$t(e)[-1]\in\Fcal_W'$, as this class is closed under subobjects and
	injective envelopes.

	Observe that we have:
	\[\Hcal_t(\Ycal\cap\Tcal_W,t(e)[-1])\subseteq\Hcal_t(\Ycal,\Xcal[-1])=0;\]
	hence $t(e)[-1]$ belongs to
	$\Fcal_W'=\Tcal_W'^\bot=[(\Ycal\cap\Tcal_W)\ast(\Xcal\cap\Tcal_W)[-1]]^\bot$ if and
	only if:
	\[ 0 =
	\Hcal_t((\Xcal\cap\Tcal_W)[-1],t(e)[-1])\simeq\Hcal(\Xcal\cap\Tcal_W,t(e)).\]
	Now, we show that this is equivalent to $\Hcal(\Tcal_W,t(e))=0$. Indeed, one
	implication is clear, as $\Xcal\cap\Tcal_W\subseteq\Tcal_W$. For the converse,
	assume $\Hcal(\Xcal\cap\Tcal_W,t(e))=0$; in particular, $t(e)\in\Xcal$ has no
	nonzero subobjects belonging to $\Xcal\cap\Tcal_W$. Since $\Tcal_W$ is closed
	under quotients and $\Xcal$ is closed under subobjects, it follows that any
	morphism with source in $\Tcal_W$ and target $t(e)$ has vanishing image, and
	so it is zero.

	So far we have proved that $\Theta(e)\in\Fcal_W'$ if and only if
	$t(e)\in\Fcal_W$; this is in turn equivalent to $e\in\Fcal_W$, as this class
	is closed under subobjects and injective envelopes.
\end{proof}

Now we are ready to prove the promised result about the canonical 
bijection between $\GSpec(c)$ and $\Spec(R)$ generalising Matlis'
correspondence. For a prime $\pf\in\Spec(R)$, we will write:
\[\spcl(\pf):=\{\,\qf\in\Spec(R)\mid \pf\subseteq\qf\,\}\quad\text{and}\quad
\gncl(\pf):=\{\,\qf\in\Spec(R)\mid \qf\subseteq\pf\,\}.\]
Later on, we will need to use the notations
$\spcl_\subseteq(\pf):=\spcl(\pf)$ and $\gncl_\subseteq(\pf):=\gncl(\pf)$ for
these sets, to distinguish them from similar notions with respect to other
partial orders on $\Spec R$.

\begin{thm}\label{thm:c(p)-support}
	Let $c\in\D(R)$ be a bounded cosilting complex. Then:
	\begin{enumerate}
		\item for every prime $\pf\in\Spec(R)$, there is a unique element $c(\pf)\in\GSpec(c)$ such that 
	$\pf\in\supp(c(\pf))\subseteq\gncl(\pf)$;
		\item	the assignment $\pf\mapsto c(\pf)$ gives a bijection $\Phi_c\colon \Spec(R)\simeq \GSpec(c)$;
		\item if $c'$ is a mutation of $c$, we have $\Phi_{c'}=\Theta\circ\Phi_c$.
	\end{enumerate}
\end{thm}

\begin{proof}
	For a prime $\pf$, we write $W_\pf:=\Spec(R)\setminus \gncl(\pf)$ for the
	largest specialisation closed subset of $\Spec(R)$ not containing $\pf$.
	Note that $\Bcal_{W_\pf}=\supp^{-1}(\gncl(\pf))$, while
	$\Bcal_{W_\pf\cup\{\pf\}}=\supp^{-1}(\gncl(\pf)\setminus\{\pf\})$.
	Therefore, for an object $x\in\D(R)$ we have $\pf\in
	\supp(x)\subseteq\gncl(\pf)$ if and only if
	$x\in\Bcal_{W_\pf}\setminus\Bcal_{W_\pf\cup\{\pf\}}$.

	Up to shift, we choose a finite chain of right mutations linking $\omega$ to
	$c$, and we argue by induction on its length to prove (1--2). We first
	check the base case, then we prove (3) in general, followed by the final
	step of induction.
	For $c=\omega$, we have already observed that the elements of $\GSpec(\omega)$
	are of the form $\omega(\pf):=E(R/\pf)$, and we have
	$\supp(E(R/\pf))=\{\pf\}$. This shows (1--2) for this base case.
	Now, if $c'$ is a right mutation of $c$, by Proposition~\ref{prop:mutation-BW} we
	have that $u\in\GSpec(c)$ lies in
	$\Bcal_{W_\pf}\setminus\Bcal_{W_\pf\setminus\{\pf\}}$ if and only if so does
	$\Theta(u)\in\GSpec(c')$; this shows (3).
	Arguing on the length of a chain of mutations linking $\omega$ and $c$, and
	using this same observation in the induction step, one concludes the
	proof of (1--2).
\end{proof}

\begin{rmk}
	Note that, while $\supp(\omega(\pf))=\{\pf\}$, it is not true in general
	that $\supp(c(\pf))=\{\pf\}$; by the Proposition, when mutating the
	support can leak downwards. For example, taking $R:=\Zbb$, so that
	$\omega=\Qbb/\Zbb$, and right-mutating at the class
	$\Prod(\Qbb)\subseteq\Prod(\omega)$, we obtain $c(p)=\widehat{\Zbb_{p}}$ for
	every nonzero prime $p\in\Zbb$ \cite[Ex.~5.5]{hrbe-naka-stov-24}. We then have
	$\supp(c(p))=\{0,p\}$.
\end{rmk}

In what follows, we will often use $\Phi_c$ to identify the sets $\Spec(R)$ and
$\GSpec(c)$.
After this identification, each particular choice of a bounded cosilting complex
$c$ will give a possibly different topology on $\Spec(R)$. For example, the
topology induced by $c=\omega$ is the \emph{Hochster topology} on $\Spec(R)$,
whose open sets are the specialisation-closed subsets.
As mentioned, these are the upper-sets of $\subseteq$, so this topology is the
$(\Tsf0)$ Alexandrov topology induced by the partial order given by inclusion
(see \S\ref{subsec:topology-gabriel-spectrum}).

\subsection{Residue fields as
Gabriel-simples}\label{subsec:residues-gabriel-simples}

For each prime $\pf\in\Spec(R)$, we denote by $k(\pf)$ the residue field $R_\pf/\pf R_\pf$, as above. These objects have the remarkable property of belonging, up to
shift, to the heart of any nondegenerate $t$-structure.

\begin{lemma}[{\cite[Cor.~2.7]{hrbe-naka-21},\cite[Prop.~4.1(1)]{pavo-vito-21}}]
	\label{lemma:s(p)-torsion-simple}
	For any $t$-structure $(\Ucal,\Vcal)$ in $\D(R)$, we have that either
	$k(\pf)\in\Ucal$ or $k(\pf)\in\Vcal$. In particular,
	any nondegenerate $t$-structure of $\D(R)$ with heart $\Hcal$ determines a
	(unique) function $f\colon \Spec(R)\to \Zbb$ such that $k(\pf)[-f(\pf)]\in\Hcal$.
\end{lemma}

For the $t$-structure $\tau=(\Ucal,\Vcal)$ associated to a pure-injective cosilting
object $c$, there is a more explicit description of this function $f$. As
mentioned at the start of the section, such a $t$-structure is compactly
generated, and thus, by
\cite[Thm.~3.11]{alon-jere-saor-10}, there is a decreasing chain
$\cdots\supseteq V_i\supseteq V_{i+1}\supseteq\cdots$ of specialisation-closed
subsets $V_i\subseteq\Spec(R)$ such that:
\begin{equation}\label{eq:aisle-coaisle}\tag{\textsf{AJS}}
\begin{aligned}
	\Ucal[-1] &=\{\,x\in\D(R)\mid \supp(H^i(x))\subseteq V_i\,\}\\
	\Vcal &=\{\,x\in\D(R)\mid \Rbb\Gamma_{V_i}(x)\in\D^{\geq i}\,\},
\end{aligned}
\end{equation}
where here by $H^\ast$ we denote the cohomology with respect to the standard
$t$-structure $(\D^{<0},\D^{\geq 0})$ of $\D(R)$. Such a chain is called an
\emph{sp-filtration}. The \emph{local cohomology
functor} $\Rbb\Gamma_{V_i}$ is the left-truncation with respect to the smashing
localisation $(\Lcal_{V_i},\Bcal_{V_i})$, and it is computed as the
right-derived functor of the torsion radical $\Gamma_{V_i}$ with respect to the
hereditary torsion pair $(\Tcal_{V_i},\Fcal_{V_i})$ of $\Mod(R)$.
Then, in this notation, we have \cite[Prop.~4.7]{pavo-vito-21}:
\[f(\pf)=\max\{\,i\in\Zbb\mid \pf\in V_i\,\}.\]
Notice that, conversely, $f$ determines the corresponding sp-filtration and, therefore, the $t$-structure $\tau$ associated to $c$, by taking
\[V_i:=\{\,\pf\in\Spec(R)\mid i\leq f(\pf)\,\}.\]

In the rest of this subsection, once we fix a nondegenerate $t$-structure with heart $\Hcal$,
we will denote by $s(\pf):=k(\pf)[-f(\pf)]\in\Hcal$ the corresponding shift of
the residue field of $\pf$. We immediately observe that these object are
(hereditary-)torsion-simple, as per Definition~\ref{dfn:torsion-simple}.

\begin{lemma}\label{lemma:s(p)-hereditary-torsion-simple}
	Let $\Hcal$ be the heart of a nondegenerate $t$-structure in $\D(R)$. Then:
	\begin{enumerate}
		\item for any torsion pair $t=(\Tcal,\Fcal)$ in $\Hcal$, $s(\pf)$ belongs to
			either $\Tcal$ or $\Fcal$;
		\item	if moreover $t$ is hereditary, then $s(\pf)$ belongs to
	either $\Tcal$ or $\Tcal^{\bot_{0,1}}\subseteq\Fcal$.
	\end{enumerate}
\end{lemma}

\begin{proof}
	(1) Consider the HRS-tilted heart
 	$\Hcal_t:=\Fcal\ast\Tcal[-1]$, which is again the heart of a nondegenerate
	$t$-structure. Then some shift
 	of $s(\pf)$ must lie in $\Hcal_t$, by the lemma above. In other words, we must have
 	$s(\pf)\in\Hcal\cap \Hcal_t[i]$ for some integer $i\in\Zbb$. The only two
 	integers for which this intersection is nonzero are $i=0$, for which we get
 	$\Hcal\cap\Hcal_t=\Fcal$, and $i=1$, for which we get
 	$\Hcal\cap\Hcal_t[1]=\Tcal$.

	(2) Denote by $H^0\colon \D(R)\to \Hcal$ the $\Hcal$-cohomology functor. If
	$t$ is hereditary, consider the subcategory $\Lcal:=\{\,x\in\D(R)\mid
	H^\ast(x)\in\Tcal\,\}$ of $\D(R)$. Note that this is a localising subcategory
	of $\D(R)$, despite the fact that coproducts in $\Hcal$ may be computed
	differently than in $\D(R)$ (see \cite[Prop.~3.2]{parr-saor-15}). Applying the Lemma to the $t$-structure
	$(\Lcal,\Lcal^\bot)$, we have either $s(\pf)\in\Lcal$, which means
	$s(\pf)\in\Tcal$, or $s(\pf)\in\Lcal^\bot$, which in particular implies
	$s(\pf)\in\Tcal^{\bot_{0,1}}$.
\end{proof}

Now we need an important property of the cosilting $t$-structures in $\D(R)$.

\begin{lemma}\label{lemma:heart-localisation}
	Let $c$ be a pure-injective cosilting object, and $(\Ucal,\Vcal)$ the
	associated $t$-structure, with heart $\Hcal$. Then for any prime $\pf\in\Spec(R)$, we have:
	\begin{enumerate}
		\item $\Ucal\otimes R_\pf\subseteq\Ucal$ and $\Vcal\otimes
			R_\pf\subseteq\Vcal$; and 
		\item $\Hcal\otimes R_\pf\subseteq \Hcal$, and $H_c^0(-\otimes R_\pf)\simeq
			H_c^0(-)\otimes R_\pf$.
	\end{enumerate}
\end{lemma}

\begin{proof}
	(1) is \cite[Lemma~2.11]{hrbe-naka-21}, but it can also be shown directly from
	the description \eqref{eq:aisle-coaisle} of $\Ucal$ and $\Vcal$ given
	above. (2) follows from (1).
\end{proof}

Observe that $(-)_\pf:=-\otimes R_\pf$ is the localisation functor (=right truncation)
with respect to the smashing localisation $(\Lcal_{W_\pf},\Bcal_{W_\pf})$ of
$\D(R)$, for $W_\pf:=\Spec(R)\setminus \gncl(\pf)$. Therefore, its restriction
to $\Hcal$ can be checked to be the functor $\Rsf\Lsf$ induced by the hereditary
torsion pair $(\Tcal_{W_\pf},\Fcal_{W_\pf})$, in the notation of
\S\ref{sub:localising}, and it is actually exact.
This property has two consequences for us. First, it shows that hereditary
torsion and torsion-free classes of finite type are closed under localisation.

\begin{lemma}\label{lemma:torsion-localisation-at-p}
	Let $c\in\D(R)$ be a pure-injective cosilting object, with heart $\Hcal$. If
	$(\Tcal,\Fcal)$ is a hereditary torsion pair of finite type in $\Hcal$, then
	$\Tcal\otimes R_\pf\subseteq\Tcal$ and $\Fcal\otimes R_\pf\subseteq\Fcal$.
\end{lemma}

\begin{proof}
	The tilted heart $\Hcal_t:=\Fcal\ast\Tcal[-1]$ is also associated to a 
	pure-injective cosilting object, and therefore it is closed under $-\otimes R_\pf$. We
	deduce that the intersections $\Tcal=\Hcal\cap\Hcal_t[1]$ and
	$\Fcal=\Hcal\cap\Hcal_t$ are also closed under $-\otimes R_\pf$.
\end{proof}

The second consequence is the following technical lemma.

\begin{lemma}\label{lemma:M(p)}
	Let $c\in\D(R)$ be a bounded cosilting complex with heart $\Hcal$.
	\begin{enumerate}
		\item $\Hcal\cap\supp^{-1}(\{\pf\}) =
			(\Mod(R)\cap\supp^{-1}(\{\pf\}))[-f(\pf)]$.
		\item If $0\to x\to s(\pf)\to y\to 0$ is a short exact sequence in $\Hcal$, then:
		\begin{enumerate}
			\item $\supp(x)$ and $\supp(y)$ are contained in $V(\pf)$; and
			\item if $x\neq 0$, $\supp(y)\subseteq V(\pf)\setminus\{\pf\}$.
		\end{enumerate}
	\end{enumerate}
\end{lemma}

\begin{proof}
	(1) ($\subseteq$) Let $(\Ucal,\Vcal)$ be the $t$-structure associated to $c$, and let
	$x\in\Hcal\cap\supp^{-1}(\{\pf\})$. In particular, $x\in\Bcal_{W_\pf}$, so we
	have $x\simeq x_\pf$. Applying
	Lemma~\ref{lemma:heart-localisation} to the standard $t$-structure, for every
	index $i$ we deduce that $H^i(x)_\pf\simeq H^i(x_\pf)\simeq H^i(x)$, where by
	$H^\ast$ here we denote the standard cohomology. Now, in the notation of
	\eqref{eq:aisle-coaisle}, we are assuming that:
	\[x\in\Hcal\subseteq\Ucal[-1]=\{\,x\in\D(R)\mid \supp(H^i(x))\subseteq
	V_i\,\}\]
	which implies that for every $i>f(\pf)$ we have $\pf\notin V(\supp(H^i(x)))$,
	so $H^i(x)=H^i(x)_\pf=0$.

	We have shown that $x\in\D^{\leq f(\pf)}$; now we show that $x\in\D^{\geq
	f(\pf)}$ as well. Indeed, for any specialisation-closed subset  $W\subseteq\Spec(R)$, if $\supp(x)\subseteq\{\pf\}$ one has:
	\[\Rbb\Gamma_W(x)=\begin{cases}x&\text{if }\pf\in W\\0&\text{otherwise}\end{cases}\]
	by using, for example, the fact that $\Rbb\Gamma_W$ is the left truncation with respect to
	$(\Lcal_W,\Bcal_W)$.
	In particular, since we are assuming that $x\in\Vcal$, using the description
	\eqref{eq:aisle-coaisle} we obtain that:
	\[x\simeq \Rbb\Gamma_{V_{f(\pf)}}(x)\in\D^{\geq f(\pf)}.\]
	
	($\supseteq$) In $\Mod(R)$, every module $M$ with $\supp(M)\subseteq\{\pf\}$ is
	a transfinite extension of copies of $k(\pf)$. This is because these are the
	semi-artinian objects of the Giraud subcategory 
	\[\Mod(R_\pf)=\Tcal_{W_\pf}^{\bot_{0,1}}\subseteq\Mod(R),\] 
	and this class is closed
	under extensions and direct limits. Therefore, $M$ is in the smallest
	subcategory of $\D(R)$ containing $k(\pf)$ and closed under extensions and
	directed homotopy colimits. This implies that $M[-f(\pf)]$ lies in $\Hcal$.

	(2.a) follows from the fact that $\Hcal\cap\supp^{-1}(V(\pf))=:\Tcal_{V(\pf)}$ is
	a hereditary torsion class of $\Hcal$, by
	Proposition~\ref{prop:smashing-cohomological}, and it contains $s(\pf)$.

	For (2.b), consider the localisation $0\to x_\pf\to s(\pf)\to y_\pf\to 0$, which
	is again an exact sequence in $\Hcal$ by
	Lemma~\ref{lemma:heart-localisation}(2). Because of (2.a), both
	$\supp(x_\pf)$ and $\supp(y_\pf)$ are contained in
	$V(\pf)\cap\gncl(\pf)=\{\pf\}$. Therefore, by (1), up to shifting by
	$f(\pf)$ this is a short exact sequence of $R_\pf$-modules. In $\Mod(R_\pf)$,
	the residue field $k(\pf)$ is simple by definition, which shows that either
	$x_\pf=0$ or $y_\pf=0$. In the first case, we deduce that $\supp(x)\subseteq
	V(\pf)\setminus\{\pf\}$; but then
	$\D(R)(x,s(\pf))\subseteq\D(R)(\Lcal_{V(\pf)\setminus\{\pf\}},\Bcal_{V(\pf)\setminus\{\pf\}})=0$,
	and therefore $x=0$. In the second case, we have as wanted $\supp(y)\subseteq
	V(\pf)\setminus\{\pf\}$.
\end{proof}

Recall the notion of the Gabriel support $\GSupp$ of a hereditary torsion class from
Definition~\ref{dfn:gsupp}. As a first application of the previous lemma, we
show that the identification $\Phi_c$ between $\Spec(R)$ and
$\GSpec(c)\simeq\GSpec(\Hcal)$ identifies the commutative support $\supp$ with the
Gabriel support $\GSupp$. This is also where our terminology comes from.

\begin{prop}\label{prop:s(p)-inj-envelope}
	Let $c\in\D(R)$ be a bounded cosilting complex, with heart $\Hcal$. Then:
	\begin{enumerate}
		\item In $\Hcal$, the injective envelope of $s(\pf)$ is $H_c^0(c(\pf))$.
		\item For a hereditary torsion class $\Tcal$ in $\Hcal$, we have
			$H_c^0(\Phi_c(\supp(\Tcal)))=\GSupp(\Tcal)$. In other words,
			$\pf\in\supp(\Tcal)$ if and only if $H_c^0(c(\pf))\in\GSupp(\Tcal)$.
	\end{enumerate}
\end{prop}

\begin{proof}
	(1) We start by showing that:
	\[\Hcal(s(\pf),H_c^0(c(\pf)))\simeq\Dcal(s(\pf),c(\pf))\neq 0.\]
	Observe that since $s(\pf)$, which is a shift of $k(\pf)$, lies in
	$\Hcal={}^{\bot_{\neq0}}c$,
	this vanishes if and only if $k(\pf)\in{}^{\bot_\Zbb}c(\pf)$; we show that
	this is not the case. We have
	$c(\pf)\in\Bcal_{W_\pf}\setminus\Bcal_{W_\pf\cup\{\pf\}}$, for
	$W_\pf:=\Spec(R)\setminus \gncl(\pf)$, by Theorem~\ref{thm:c(p)-support}. Taking the localising
	subcategory $\Lcal:={}^{\bot_\Zbb}c(\pf)\subseteq\D(R)$, we then have that $\Lcal_{W_\pf}\subseteq\Lcal$. Then we must have $k(\pf)\notin\Lcal$,
	otherwise we would have $\Lcal_{W_\pf\cup\{\pf\}}\subseteq \Lcal$ and
	$c(\pf)\in\Bcal_{W_\pf\cup\{\pf\}}$, reaching a contradiction.

	Now, consider a nonzero morphism $s(\pf)\to H_c^0(c(\pf))$. If it is not monic
	in $\Hcal$, its image $y$ is a proper quotient of $s(\pf)$; therefore, by
	Lemma~\ref{lemma:M(p)}(2.b), it is contained in
	$\Lcal_{V(\pf)\setminus\{\pf\}}\subseteq\Lcal_{W_\pf}$, since
	$W_\pf=\Spec(R)\setminus \gncl(\pf)\supseteq V(\pf)\setminus\{\pf\}$. But then:
	\[\Hcal(y,H_c^0(c))\simeq\Dcal(y,c(\pf))
		\subseteq\Dcal(\Lcal_{W_\pf},\Bcal_{W_\pf})=0\]
	by Theorem~\ref{thm:c(p)-support}, and we would conclude that $y=0$, a
	contradiction. Hence $s(\pf)$ is a subobject of $H_c^0(c(\pf))$, which is then
	its injective envelope.

	(2) For a hereditary torsion pair $(\Tcal,\Fcal)$ in $\Hcal$, we have:
	\begin{align*}
		&\pf\in\supp\Tcal\iff s(\pf)\in\Tcal\iff s(\pf)\notin\Fcal\iff \\
	\iff &H_c^0(c(\pf))=E(s(\pf))\notin \Fcal \iff H_c^0(c(\pf))\in\GSupp(\Tcal).
	\end{align*}
	These implications are, in order, due to: (a) \cite[Thm.~4.1(3)]{pavo-vito-21},
	(b) Lemma~\ref{lemma:s(p)-hereditary-torsion-simple}, (c) the fact that
	$\Fcal$ is closed under subobjects and injective envelopes, and (d) the
	definition of $\GSupp(\Tcal)$.
\end{proof}

We are finally ready to prove the main result of this subsection.

\begin{prop}\label{prop:residues-gabriel-simples}
	Let $c$ be a bounded cosilting complex, with heart $\Hcal$. Then the objects
	\[\{\,s(\pf)\in\Hcal\mid \pf\in\Spec(R)\,\}\]
	are the Gabriel-simples of $\Hcal$.
\end{prop}

\begin{proof}
	As mentioned at the beginning of this section, we know that $\Hcal$ is
	semi-noetherian, so its Gabriel spectrum $\GSpec(\Hcal)$ (which we identify with
	$\Spec(R)$ via $\Phi_c$) has Cantor--Bendixson rank. Denote by
	$\emptyset=X_{-1}\subseteq X_0\subseteq\cdots\subseteq X_\delta=\Spec(R)$ the
	Cantor--Bendixson filtration. Then the steps of the Gabriel filtration of
	$\Hcal$ are $\Hcal_\alpha:=\GSupp^{-1}(X_\alpha)=\supp^{-1}(X_\alpha)\cap\Hcal$, by
	Proposition~\ref{prop:cantor-bendixson-gabriel} and
	Proposition~\ref{prop:s(p)-inj-envelope}(2).
	For every prime $\pf\in\Spec(R)$, let
	$0\leq\alpha\leq \delta$ be
	the minimum ordinal for which $\pf\in X_\alpha$: it is a successor ordinal,
	and it is such that $\pf$ is an open point in $\Spec(R)\setminus X_{\alpha-1}$.
	In other words, $X_{\alpha-1}\sqcup\{\pf\}$ is open in $\Spec(R)$, with respect
	to the Gabriel topology of $\Hcal$. Let
	$(\Tcal:=\supp^{-1}(X_{\alpha-1}\sqcup\{\pf\})\cap\Hcal,\Fcal)$ be the corresponding
	hereditary torsion pair in $\Hcal$. Since
	$s(\pf)\in\Tcal$,
	for every short exact sequence $0\to x\to s(\pf)\to y\to 0$ in $\Hcal$ with
	$x\neq 0$ we
	also have $y\in\Tcal$. On the other hand, by Lemma~\ref{lemma:M(p)}(2.b), we actually have
	$\pf\notin\supp(y)$, so $\supp(y)\subseteq X_{\alpha-1}$ and therefore $y\in\Hcal_{\alpha-1}$. Since
	instead $\pf\notin X_{\alpha-1}$ by construction, we have $s(\pf)\in\Hcal_{\alpha-1}^{\bot_{0,1}}$ by
	Lemma~\ref{lemma:s(p)-hereditary-torsion-simple}(2). We have proved that $s(\pf)$ is a
	simple object of the Giraud subcategory corresponding to $\Hcal_{\alpha-1}$;
	this means that $s(\pf)$ is a Gabriel-simple object of $\Hcal$ of dimension
	$\alpha$. Since the
	assignment $s(\pf)\mapsto E(s(\pf))=H_c^0(c(\pf))$ is surjective on
	$\GSpec(\Hcal)$, by Lemma~\ref{lemma:gabriel-simple-subobject} these are all
	the Gabriel-simple objects of $\Hcal$.
\end{proof}

\subsection{Topological properties of the Gabriel spectra}

From the results proved so far, we immediately obtain the following
topological property of $\GSpec(c)$.

\begin{cor}\label{cor:T0-alexandrov}
	For a bounded cosilting complex $c\in\D(R)$, $\GSpec(c)$ is Alexandrov.
\end{cor}

\begin{proof}
	Combining Lemma~\ref{lemma:s(p)-hereditary-torsion-simple} and
	Proposition~\ref{prop:residues-gabriel-simples} we see that the
	Gabriel-simples of $\Hcal$ are hereditary-torsion-simple. Since any
	indecomposable injective contains a Gabriel-simple as a subobject by
	Lemma~\ref{lemma:gabriel-simple-subobject}, we apply
	Theorem~\ref{thm:torsion-simple-socle} to conclude that $\GSpec(c)$ is
	Alexandrov.
\end{proof}

It follows that the Gabriel topology on $\GSpec(c)$ is determined by its
closure preorder (see \S\ref{subsec:topology-gabriel-spectrum}). Once we fix a bounded cosilting
complex $c$ and identify $\GSpec(c)$ with $\Spec(R)$, we will denote by $\preceq$
the closure preorder on $\Spec(R)$ induced by the topology of $\GSpec(c)$.
For an element $\pf\in\Spec(R)$, we denote the smallest upper-set (respectively, lower-set) of $\Spec(R)$ with respect to
$\preceq$ which contains $\pf$ as follows:
\[\spcl_\preceq(\pf):=\{\,\qf\in\Spec(R)\mid \pf\preceq \qf\,\} \quad\text{and}\quad
	\gncl_\preceq(\pf):=\{\,\qf\in\Spec(R)\mid \qf\preceq \pf\,\}.\]
This is then the smallest open
(respectively, closed) set of $\Spec(R)$ for the Gabriel topology of $\GSpec(c)$
which contains
$\pf$. The preorder $\preceq$ is actually a partial order (that is, it is
antisymmetric), as a consequence of the following fact.

\begin{lemma}\label{lemma:preceq-finer}
	The (pre)order $\preceq$ is finer than $\subseteq$, that is, $\pf\preceq\qf$
	implies $\pf\subseteq\qf$. In other words, the map
	$\Phi_c\colon\GSpec(\omega)\simeq \Spec(R)\to \GSpec(c)$ is open.
\end{lemma}

\begin{proof}
	By Proposition~\ref{prop:smashing-cohomological}, the specialisation-closed
	subsets $W\subseteq\Spec(R)$ are open sets for the Gabriel topology of
	$\GSpec(c)$. This means that they are upper-sets for $\preceq$. In particular
	$\spcl_{\preceq}(\pf)\subseteq\spcl_{\subseteq}(\pf)$ for every
	$\pf\in\Spec(R)$, which is a rephrasing of the statement.
\end{proof}

Recall that, in a topological space, a non-empty closed subset is \emph{irreducible}
if it cannot be written as the union of two smaller
non-empty closed subsets. In an Alexandrov topological space, this means a
non-empty lower-set with respect to the closure preorder which cannot be written
as the union of two smaller non-empty lower-sets.
A topological space is
\emph{sober} if every irreducible closed subset is the closure of exactly one
point. In this case there is a bijection between points and irreducible closed
subsets, under which the closure (pre)order on points corresponds to inclusion
of closed subsets.

Recall also that a poset is \emph{artinian} (respectively, \emph{noetherian}) if it
satisfies the descending (respectively, ascending) chain condition.

\begin{cor}\label{cor:topological-GSpec(c)}
	For a bounded cosilting complex $c\in\D(R)$, we have that:
	\begin{enumerate}
		\item $\preceq$ is a partial order;
		\item $\GSpec(c)$ is $(\Tsf_0)$;
		\item $(\Spec(R),\preceq)$ is an artinian and noetherian poset;
		\item $\GSpec(c)$ is a sober topological space;
	\end{enumerate}
\end{cor}

\begin{proof}
	(1) Since $\preceq$ is finer than $\subseteq$ and the latter is antisymmetric,
	$\preceq$ must be as well.

	(2) An Alexandrov topology is $(\Tsf_0)$ if and
	only if its closure preorder is a partial order, so this is a rephrasing of
	(1).

	(3) This follows from the fact that $(\Spec(R),\subseteq)$ is both artinian and
	noetherian. Indeed, noetherianity of this poset follows from the noetherianity
	of $R$. Artinianity, instead, follows from Krull's Principal Ideal Theorem: a
	prime ideal $\pf$ generated by $n$ elements has height at most $n$. Since any
	prime ideal is finitely generated, it has finite height, and therefore there
	are no infinite strictly decreasing chains of primes below it. Now, since
	$\preceq$ is finer than $\subseteq$, any chain for $\preceq$ is also a chain
	for $\subseteq$, and therefore $(\Spec(R),\preceq)$ satisfies both chain
	conditions as well.

	(4) So far we know that $\GSpec(c)$ is $(\Tsf_0)$ Alexandrov and that the
	closure order $\preceq$ is noetherian. For a subset $\Qcal\subseteq\Spec(R)$,
	denote by $\max_\preceq\Qcal$ the subset of $\preceq$-maximal elements of
	$\Qcal$: then by noetherianity we have that every element of $\Qcal$ lies
	below a maximal element, that is,
	$\Qcal\subseteq\bigcup_{\qf\in\max_{\preceq}\Qcal}\gncl_{\preceq}(\qf)$.
	Observe that if $\Qcal$ is closed, that is, a lower-set for $\preceq$, this inclusion is an
	equality. Then a closed subset $\Qcal$ is irreducible if and only if
	$\max_{\preceq}\Qcal=\{\qf\}$, that is, $\Qcal=\gncl_{\preceq}(\qf)$ is the
	closure of a (unique) point.
\end{proof}

\subsection{Right mutation induces an open map between Gabriel spectra}

We are now ready to prove the main result of this section. We are going to
prove the following.

\begin{thm}\label{thm:theta-open}
	Let $c\in\D(R)$ be a bounded cosilting object, and let $c'$ be a
	right mutation of $c$. Then the map $\Theta\colon \GSpec(c)\to \GSpec(c')$ is
	open.
\end{thm}

The proof will occupy the rest of this subsection, involving a few reduction and
reformulation steps, as well as some auxiliary lemmas.

We start by fixing some notation. 
Let $c$ and $c'$ be as in the statement, and $\Ecal\subseteq\Prod(c)$ be the
class at which we have right-mutated. Denote by $\Ocal$ and
$\Ocal'$ the two topologies induced on $\Spec(R)$ by the identifications
$\Phi_c\colon\Spec(R)\simeq \GSpec(c)$ and
$\Phi_{c'}\colon\Spec(R)\simeq\GSpec(c')$. These are $(\Tsf0)$ Alexandrov
topologies, by Corollary~\ref{cor:T0-alexandrov}; denote by $\preceq$ and
$\preceq'$ the corresponding closure partial orders on $\Spec(R)$. Denote also
by $E$ the subset of $\Spec(R)$ for which $\Phi_c(E)=\Ind\Ecal=\Phi_{c'}(E)$.
For a set $\Qcal\subseteq\Spec(R)$, denote by $\min_\preceq\Qcal\subseteq\Qcal$ the
set of elements of $\Qcal$ which are minimal with respect to $\preceq$. Since
$(\Spec(R),\preceq)$ is artinian by
Corollary~\ref{cor:topological-GSpec(c)}(3), we have:
\[\textstyle \Qcal\subseteq \bigcup_{\qf\in\min_\preceq\Qcal}\spcl_\preceq(\qf).\]
Translating the claim of Theorem~\ref{thm:theta-open} in these terms, we get the
following lemma.

\begin{lemma}\label{lemma:reformulation}
	The following are equivalent:
	\begin{enumerate}
		\item $\Theta\colon \GSpec(c)\to \GSpec(c')$ is open;
		\item $\Ocal'$ is finer than $\Ocal$, that is, $\Ocal\subseteq\Ocal'$;
		\item $\preceq'$ is finer that $\preceq$, that is, $\pf\preceq'\qf\implies
			\pf\preceq\qf$ for every $\pf,\qf\in\Spec(R)$;
		\item $\pf\preceq'\qf\implies \pf\preceq\qf$ for every $\pf\in E,\;\qf\notin E$;
		\item $\min_\preceq(\spcl_{\preceq'}(\pf)\cap E^\cpl)\subseteq
			\spcl_\preceq(\pf)$ for every $\pf\in E$.
	\end{enumerate}
\end{lemma}

\begin{proof}
	The first three items are just reformulations.
	By Theorem~\ref{thm:c(p)-support}(3), we have $\Phi_{c'}=\Theta\circ\Phi_c$,
	which means that after identifying $\GSpec(c)$ and $\GSpec(c')$ with
	$\Spec(R)$, the map $\Theta$ corresponds to the identity. Hence it is open if
	and only if the target topology $\Ocal'$ is finer than the source topology
	$\Ocal$.
	($2\Leftrightarrow3$) is clear, as open sets for an Alexandrov topology are
	the upper-sets of the closure order.

	The implication $(3\Rightarrow 4)$ is obvious, while the converse
	$(4\Rightarrow 3)$ follows from the general
	results proved in earlier sections. Indeed, by Lemma~\ref{lemma:two-cases},
	$E$ is closed with respect to both $\Ocal$ and $\Ocal'$, so it is a lower set
	for both $\preceq$ and $\preceq'$. In particular, neither of the sides of the
	implication in~(3) can hold with $\pf\notin E$
	and $\qf\in E$.
	Moreover, by Theorem~\ref{thm:piecewise-homeomorphism}, $\Ocal$ and $\Ocal'$
	induce the same subspace topologies on $E$ and its complement $E^\cpl$, which
	means that the two restriction of $\preceq$ and $\preceq'$ on these subsets
	coincide. Therefore the implication in~(3) holds (and it is actually an
	equivalence) when either $\pf,\qf\in E$ or $\pf,\qf\notin E$. Hence (4) is the
	only case of (3) which is not already established by earlier results.

	($4\Leftrightarrow 5$) The complement $E^\cpl$ is open for both $\Ocal$ and
	$\Ocal'$ by Lemma~\ref{lemma:two-cases}. Observe that item (4) can be
	restated as: $\spcl_{\preceq'}(\pf)\cap E^\cpl\subseteq\spcl_{\preceq}(\pf)$
	for every $\pf\in E$. We notice that $\spcl_{\preceq'}(\pf)\cap E^\cpl$ is
	open for $\Ocal'$, and since it is contained in $E^\cpl$, is also open for
	$\Ocal$, again by Lemma~\ref{lemma:two-cases}. In other words, it is an
	upper-set for $\preceq$, and thus we actually have an equality:
	\[\textstyle\spcl_{\preceq'}(\pf)\cap
	E^\cpl=\bigcup_{\qf\in\min_{\preceq}(\spcl_{\preceq'}(\pf)\cap
	E^\cpl)}\spcl_\preceq(\qf).\]
	This gives the equivalence with (5).
\end{proof}

As in \S\ref{subsec:residues-gabriel-simples}, denote by $s(\pf)$ and
$s'(\pf)$ the Gabriel-simple objects of $\Hcal$ and $\Hcal_t$, for
$\pf\in\Spec(R)$. Denote by $(\Tcal:=\supp^{-1}(E^\cpl),\Fcal)$ the torsion pair of finite type in $\Hcal$ at
which we have right-mutated, so that $\Hcal_t=\Fcal\ast\Tcal[-1]$. Then, by
definition of $s'(\pf)$ we have:
\[s'(\pf)=\begin{cases}
	s(\pf)&\text{if }s(\pf)\in\Fcal,\text{ that is if }\pf\in E\\
	s(\pf)[-1]&\text{if }s(\pf)\in\Tcal,\text{ that is if }\pf\in E^\cpl,
\end{cases}\]
where these are the only cases that can occurr, by
Lemma~\ref{lemma:s(p)-hereditary-torsion-simple}.

\medskip
From now on, we fix $\pf\in E$ and $\qf\in\min_\preceq(\spcl_{\preceq'}(\pf)\cap
E^\cpl)$, towards proving that Lemma~\ref{lemma:reformulation}(5) holds. We
therefore have $s'(\pf)=s(\pf)$ and $s'(\qf)=s(\qf)[-1]$.
The assumption of minimality of $\qf$ will be used as follows. We have said that
$\spcl_{\preceq'}(\pf)\cap E^\cpl$ is an upper-set for $\preceq$; since $\qf$ is
$\preceq$-minimal in it, the difference $(\spcl_{\preceq'}(\pf)\cap
E^\cpl)\setminus\{\qf\}$ is also $\preceq$-upper. We therefore have localising
subcategories of $\Hcal$:
\[\tilde\Xcal:=\supp^{-1}(\spcl_{\preceq'}(\pf)\cap E^\cpl)\cap\Hcal\quad\text{and}\quad
\Xcal:=\supp^{-1}(\spcl_{\preceq'}(\pf)\cap E^\cpl\setminus\{\qf\})\cap\Hcal,\]
whose supports only differ by $\qf$.
If we write $\begin{tikzcd}[cramped,sep=small]\Lsf\colon \Hcal\arrow[shift
left]{r}&\arrow[shift left]{l}\Hcal/\Xcal:\!\Rsf\end{tikzcd}$ for the exact
localisation functor associated to $\Xcal$ and its right adjoint, then we have
the following technical lemma.

\begin{lemma}\label{lemma:q-minimal}
	We have that:
	\begin{enumerate}
		\item $s(\qf)\in\Xcal^{\bot_{0,1}}$, so that $s(\qf)\simeq \Rsf\Lsf(s(\qf))$;
		\item $\Lsf(s(\qf))$ is simple in $\Hcal/\Xcal$;
		\item $\tilde\Xcal/\Xcal:=\Lsf(\tilde\Xcal)=\Loc(\Lsf(s(\qf)))\subseteq\Hcal/\Xcal$.
		\item for every object $z\in\tilde\Xcal\setminus\Xcal$, there is a
			monomorphism $\Lsf(s(\qf))\hookrightarrow \Lsf(z)$ in $\Hcal/\Xcal$.
	\end{enumerate}
\end{lemma}

\begin{proof}
	(1) By construction of $\Xcal$ we have
	$s(\qf)\notin\Xcal$. Therefore $s(\qf)\in\Xcal^{\bot_{0,1}}$ by
	Lemma~\ref{lemma:s(p)-hereditary-torsion-simple}(2). 

	(2) Using (1), to show that $\Lsf(s(\qf))$ is simple we can equivalently show
	that all the proper quotients $x$ of $s(\qf)$ in $\Hcal$ belong to $\Xcal$.
	Since $s(\qf)\in\tilde\Xcal$, we have $x\in\tilde\Xcal$; on the other hand, by
	Lemma~\ref{lemma:M(p)}(2.b) we have $\qf\notin\supp(x)$, hence actually
	$x\in\Xcal$.

	(3) By looking at supports, the inclusion
	$\Xcal\subseteq\tilde\Xcal$ is a minimal inclusion of localising subcategories
	of $\Hcal$. Since the assignment $\Lcal\mapsto \Lcal/\Xcal$ gives an
	inclusion-preserving bijection between localising subcategories $\Lcal$ of $\Hcal$
	containing $\Xcal$ and localising subcategories of $\Hcal/\Xcal$ by
	Proposition~\ref{prop:spectrum-quotients}(1), we deduce that
	$\tilde\Xcal/\Xcal$ is a minimal nonzero localising subcategory of
	$\Hcal/\Xcal$. Since it contains
	$\Lsf(s(\qf))$, it must then coincide with $\Loc(\Lsf(s(\qf))$.

	(4) For $z\in\tilde\Xcal\setminus\Xcal$, we have that $0\neq
	\Lsf(z)\in\tilde\Xcal/\Xcal=\Loc(\Lsf(s(\qf)))$, by (3). Then $\Lsf(z)$ must have a
	monomorphism from the simple object $\Lsf(s(\qf))$ in $\Hcal/\Xcal$, otherwise it
	would lie in the torsion-free class
	$(\Lsf(s(\qf)))^\bot=(\Loc(\Lsf(s(\qf))))^\bot$.
\end{proof}

Now we are ready to prove that Lemma~\ref{lemma:reformulation}(5) holds.

\begin{lemma}\label{lemma:meat}
	We have $\min_{\preceq}(\spcl_{\preceq'}(\pf)\cap
	E^\cpl)\subseteq\spcl_{\preceq}(\pf)$ for every $\pf\in E$.
\end{lemma}

\begin{proof}
	We fix $\pf\in E$ and $\qf\in\min_{\preceq}(\spcl_{\preceq'}(\pf)\cap E^\cpl)$,
	and keep the notations introduced so far. In order to show that $\qf\in
	\spcl_\preceq(\pf)$, we will start by constructing a quotient $w$ of $s(\pf)$
	in $\Hcal$. Since $w$ will lie in every localising subcategory of
	$\Hcal$ containing $s(\pf)$, its support will be contained in the smallest open set
	containing $\supp(s(\pf))=\{\pf\}$, that is,
	$\supp(w)\subseteq\spcl_\preceq(\pf)$. Then we will show that
	$\qf\in\supp(w)$, concluding the proof.

	\smallskip\noindent\textbf{Step 1.}\ \emph{Construction of the quotient $w$}. Recall that the Gabriel-simples are
	hereditary-torsion-simple subobjects of the indecomposable injectives of
	$\Hcal_t$, by Lemma~\ref{lemma:s(p)-hereditary-torsion-simple}. Therefore, by 
	Lemma~\ref{lemma:hts-topology-characterisation}, our assumption that
	$\pf\preceq'\qf$ translates into:
	\[\exists\, x\subseteq s'(\pf)\text{ in }\Hcal_t\text{ such that
	}\Hcal_t(x,s'(\qf))\neq 0.\]
	Now, as mentioned, we have $s'(\pf)=s(\pf)$ and $s'(\qf)=s(\qf)[-1]$.
	The triangle witnessing that $x$ is a subobject of $s'(\pf)$ in $\Hcal_t$ and its
	long exact sequence of $\Hcal$-cohomology are then:
	\[
		x\to s(\pf)\to y\to x[1],\qquad 0\to H_c^0(x)\to s(\pf)\to y\to H_c^1(x)\to 0,
	\]
	with $y\in\Fcal$.
	Since $s'(\qf)=s(\qf)[-1]$ is torsion-free for the pair $(\Fcal,\Tcal[-1])$ in
	$\Hcal_t$, we have:
	\[0\neq \Hcal_t(x,s'(\qf))\simeq \Hcal_t(H_c^1(x)[-1],s(\qf)[-1])\simeq
	\Hcal(H_c^1(x),s(\qf)).\]
	We let
	$z:=H_c^1(x)\in\Tcal$, and we denote by $w$ the image (in $\Hcal$) of the
	morphism $s(\pf)\to y$, so that there are short exact sequences:
	\[0\to w\to y\to z\to 0\quad\text{in }\Hcal\qquad\text{and}\qquad
		0\to \Lsf(w)\to \Lsf(y)\to \Lsf(z)\to 0\quad\text{in }\Hcal/\Xcal.\]

	\smallskip\noindent\textbf{Step 2.}\ \emph{Proof that $\qf$ lies in
	$\supp(w)$}.
	Since $z[-1]$ is a quotient of the subobject
	$x$ of $s(\pf)$ in $\Hcal_t$, it lies in every localising subcategory of
	$\Hcal_t$ containing $s(\pf)$. Hence, as before, we have:
	\[\supp(z)=\supp(z[-1])\subseteq\spcl_{\preceq'}(\pf).\]
	On the other hand, by construction $z\in\Tcal=\supp^{-1}(E^\cpl)\subseteq\Hcal$,
	so in fact we have:
	\[\supp(z)\subseteq\spcl_{\preceq'}(\pf)\cap E^\cpl,\quad\text{that is,}\quad z\in\tilde\Xcal.\]
	Moreover, since $\Hcal(z,s(\qf))\neq 0$ and $s(\qf)\in\Xcal^{\bot_{0,1}}$, we
	must have $z\notin\Xcal$.
	We are in a position to apply Lemma~\ref{lemma:q-minimal}(4), which yields a monomorphism
	$\Lsf(s(\qf))\hookrightarrow \Lsf(z)$ in $\Hcal/\Xcal$. We use it to compute a
	pullback from the sequence above:
	\[\begin{tikzcd}[sep=1pc]
		0\arrow{r}& \Lsf(w) \arrow{r} & \Lsf(y) \arrow{r} & \Lsf(z)\arrow{r} & 0\\
		0\arrow{r}& \Lsf(w) \arrow[equal]{u} \arrow{r} & P \arrow[hook]{u} \arrow{r}
		\arrow[phantom,description, "\urcorner", very near start]{ur} &
			\Lsf(s(\qf))	\arrow[hook]{u} \arrow{r} & 0
	\end{tikzcd}\qquad \text{in }\Hcal/\Xcal.\]
	Applying the left exact functor $\Rsf$ we obtain an exact sequence:
	\[0\to \Rsf\Lsf(w)\to \Rsf(P)\to s(\qf)\to t\to 0\text{ in }\Hcal\]
	with $t\in\Xcal$. Since instead $s(\qf)\notin\Xcal$, $t$ must be a proper
	quotient of $s(\qf)$, and therefore  we have
	$\supp(t)\subseteq\spcl_{\subseteq}(\qf)\setminus\{\qf\}$ by
	Lemma~\ref{lemma:M(p)}(2.b).
	We are going to perform a last manipulation of this sequence, localising it
	at $\qf$ by applying the functor $-\otimes R_\qf$. This yields an exact
	sequence in $\Hcal$ by Lemma~\ref{lemma:heart-localisation}. Moreover, by our last
	observation on the support of $t$, we have $t\otimes R_\qf=0$. The (short)
	exact sequence we obtain is therefore:
	\[0\to (\Rsf\Lsf(w))_\qf\to (\Rsf(P))_\qf\to s(\qf)\to 0\quad\text{in }\Hcal.\]
	Now we show that this sequence cannot split.
	Since $\Rsf$ is left-exact, from the monomorphism $P\hookrightarrow
	\Lsf(y)$ in $\Hcal/\Xcal$ we deduce that $\Rsf(P)\hookrightarrow \Rsf\Lsf(y)$ in $\Hcal$.
	Since $\spcl_{\preceq'}(\pf)\cap E^\cpl\subseteq E^\cpl$, we have an inclusion of
	the corresponding localising subcategories $\Xcal\subseteq\Tcal$; therefore,
	by Lemma~\ref{lemma:torsion-localisation}(1), the torsion-free class $\Fcal$ is
	closed under $\Rsf\Lsf$. By construction we have $y\in\Fcal$, from which we deduce
	that $\Rsf\Lsf(y)$ and then $\Rsf(P)$ both belong to $\Fcal$ as well. Moreover,
	$(\Tcal,\Fcal)$ is of finite type by assumption, so by
	Lemma~\ref{lemma:torsion-localisation-at-p} the torsion-free class $\Fcal$ is
	closed under localisation at $\qf$. We finally deduce that
	$(\Rsf(P))_\qf\in\Fcal$; therefore the sequence cannot split,
	as $s(\qf)\in\Tcal$.

	We conclude that we must have $\qf\in\supp((\Rsf\Lsf(w))_\qf)$.
	Indeed, if we write as in earlier sections $W_\qf:=\Spec(R)\setminus\gncl_{\subseteq}(\qf)$, we have by definition that
	$(\Rsf\Lsf(w))_\qf\in\Bcal_{W_\qf}$; if we had $\qf\notin \supp((\Rsf\Lsf(w))_\qf)$, then
	we would actually have $(\Rsf\Lsf(w))_\qf\in\Bcal_{W_\qf\cup\{\qf\}}$. However,
	since $s(\qf)\in\Lcal_{W_\qf\cup\{\qf\}}$, this would make the connecting map
	$s(\qf)\to (\Rsf\Lsf(w))_\qf[1]$ vanish, splitting the sequence.

	Therefore we have $\qf\in\supp((\Rsf\Lsf(w))_\qf)\subseteq\supp(\Rsf\Lsf(w))$.  In view of
	the sequence:
	\[0\to x_0\to w\to \Rsf\Lsf(w)\to x_1\to 0\]
	with $x_0,x_1\in\Xcal$, we deduce that:
	\[\qf\in\supp(\Rsf\Lsf(w))\subseteq\supp(x_0)\cup\supp(w)\cup\supp(x_1)\subseteq\supp(w)\cup\supp(\Xcal).\]
	Since $\qf\notin\supp(\Xcal)$ by construction, we finally conclude that
	$\qf\in\supp(w)$. This, as mentioned earlier, proves that
	$\qf\in\spcl_\preceq(\qf)$.
\end{proof}

Summarising the argument:

\begin{proof}[Proof of Theorem~\ref{thm:theta-open}]
	The statement of Theorem~\ref{thm:theta-open} is item (1) of
	Lemma~\ref{lemma:reformulation}. This is equivalent to item (5) of the Lemma,
	which holds by Lemma~\ref{lemma:meat}.
\end{proof}

\section{Computations of Gabriel spectra}
\label{sec:concrete-computations}

Let $R$ be a commutative noetherian ring. The goal of this section is to
compute explicitly the Gabriel spectra of the hearts $\Hcal$ associated to certain
bounded cosilting complexes $c$ of $\D(R)$, and the corresponding closure orders on
$\Spec(R)$. As seen before, there is a canonical bijection
$H^0_c(\Phi_c(-))\colon \Spec(R)\to \GSpec(\Hcal)$, which we use to
identify these sets without further mention.

We start by fixing some notation for the rest of the section. Let $c\in\D(R)$ be
a bounded cosilting complex, and $\tau$ be its associated homotopically smashing
$t$-structure. As explained at the beginning of
\S\ref{subsec:residues-gabriel-simples}, $\tau$ is described by an sp-filtration
$\cdots\supseteq V_{i-1}\supseteq V_i\supseteq\cdots$. Since $c$ is bounded,
this filtration is \emph{intermediate}, that is, we have $V_{-i}=\Spec R$ and
$V_i=\emptyset$ for $i\gg 0$. Up to shifting, we may assume that $\tau$
corresponds to a filtration of the form
\[\Spec(R)=V_{-1}\supsetneq V_0\supseteq\cdots\supseteq V_{n-1}\supsetneq
V_n=\emptyset,\]
for $n\geq 0$ (the \emph{length} of the filtration). In this case, $\tau$ can be
obtained from the standard $t$-structure by a chain of right mutations
\cite[Lemma~6.8]{pavo-vito-21}, whose construction we now recall.

\begin{construction}\label{con:chain}
	Let $\tau$ be the $t$-structure associated to the filtration
	\[\Spec R\supsetneq V_0\supseteq\cdots\supseteq V_{n-1}\supsetneq \emptyset.\]
	For every $0\leq j\leq n$, consider the $t$-structures $\tau_j$, with hearts
	$\Hcal_j$, associated to the filtrations:
	\begin{align*}
		\tau_0 \qquad &\Spec(R)\supsetneq \emptyset \\
		\tau_1 \qquad &\Spec(R)\supsetneq V_0\supsetneq \emptyset \\
		\tau_2 \qquad &\Spec(R)\supsetneq V_0 \supseteq V_1 \supsetneq \emptyset \\
		\vdots\qquad\;&\qquad\vdots\\
		\tau_n \qquad &\Spec(R)\supsetneq V_0 \supseteq V_1 \supseteq\cdots\supseteq V_{n-1}\supsetneq \emptyset
	\end{align*}
	Then $\tau_0$ is the standard $t$-structure, $\tau_n$ equals
	$\tau$, and $\tau_i$ is obtained from $\tau_{i-1}$ by right mutation at the
	closed set $V_{i-1}^\cpl\subseteq\GSpec(\Hcal_{i-1})$, for every $1\leq i\leq n$. 
\end{construction}

When referring to this chain of mutations, for every $0\leq i\leq n$ we denote
by $s_i(\pf)$ the Gabriel-simples of $\Hcal_i$, and by
$\preceq_i$ the closure order on $\Spec(R)$ induced by $\GSpec(\Hcal_i)$.
In the heart $\Hcal_i$, we write $t_i:=(\Tcal_{V_i},\Fcal_{V_i})$ for the
hereditary torsion pair of finite type given by Proposition~\ref{prop:smashing-cohomological}, which is the
one involved in the mutation to obtain $\tau_{i+1}$.
Since $\tau_n=\tau$, we will also write $\Hcal:=\Hcal_n$ for its heart and
$\preceq:=\preceq_n$ for its closure order, which is what we want to determine.

\subsection{Preliminary observations}\label{subsec:preliminary-computations}

Our first observation is that Theorem~\ref{thm:piecewise-homeomorphism} gives an
upper bound on how fine the Gabriel topology on $\GSpec(\Hcal)$ and its closure
order $\preceq$ can be.

\begin{lemma}\label{lemma:bound-topology}
	For every $0\leq i\leq n$, the topologies induced on the differences
	$V_{i-1}\setminus V_{i}$ by $\GSpec(\Hcal)$ and $\Spec(R)$ coincide. In other
	words, for $\pf,\qf\in V_{i-1}\setminus V_{i}$ we have $\pf\preceq\qf$ if and
	only if $\pf\subseteq\qf$.
\end{lemma}

\begin{proof}
	In the chain of Construction~\ref{con:chain}, right mutation is performed at
	the closed sets $V_j^\cpl$, for $1\leq j\leq n$. Since $V_{i-1}\setminus V_i$
	always lies entirely either in $V_j^\cpl$ or in $V_j$, by
	Theorem~\ref{thm:piecewise-homeomorphism} each mutation operation does not
	change its subspace topology.
\end{proof}

Together with the fact that the sets $V_i$ are upper-sets for $\preceq$ by
Lemma~\ref{lemma:preceq-finer}, this shows that, in order to have a full
description of $\GSpec(\Hcal)$, we only need to determine the
relations $\pf\preceq\qf$ involving $\pf\notin V_i\ni\qf$ for some $i$.
As a consequence of Proposition~\ref{prop:ttf-simple}, we also have that no such
relation can occur for certain primes $\pf$.

\begin{lemma}\label{lemma:maximal-difference}
	Let $0\leq i\leq n$, and let $\pf\in V_{i-1}\setminus V_{i}$ be
	maximal in this difference (with respect to $\preceq$, or equivalently to
	$\subseteq$). Then $\pf$ is maximal in $\Spec(R)$ with respect to $\preceq$.
\end{lemma}

\begin{proof}
	First, notice that the equivalent reformulation of the assumption follows from
	the fact that $\preceq$ and $\subseteq$ coincide on $V_{i-1}\setminus V_{i}$, by
	the Lemma above.

	We distinguish two cases. If $i=n$, then we are assuming that $\pf$ is maximal
	in the difference $V_{n-1}\setminus V_n=V_{n-1}$; therefore, $\pf$ is a
	maximal prime in $\Spec(R)$ with respect to $\subseteq$. Since the relation
	$\preceq$ is finer by Lemma~\ref{lemma:preceq-finer}, $\pf$ is also maximal
	with respect to $\preceq$.

	Now, consider the case $i<n$.
	If $\pf\in V_{i-1}\setminus V_{i}$ is maximal in this difference,
	the Gabriel-simple $s_{i}(\pf)$ of $\Hcal_{i}$ is a simple object of the
	Giraud subcategory of $t_{i}$, by 
	Lemmas~\ref{lemma:s(p)-hereditary-torsion-simple} and~\ref{lemma:M(p)}(2.b). Therefore
	$s_{i+1}(\pf)=s_{i}(\pf)$ is a simple object of $\Hcal_{i+1}$ by
	Proposition~\ref{prop:ttf-simple}(3). This means that $\pf$ is a maximal
	element with respect to the closure order $\preceq_{i+1}$ of
	$\GSpec(\Hcal_{i+1})$ (\emph{cf.}\
	Proposition~\ref{prop:cantor-bendixson-gabriel}(1)). Since right mutation refines the closure order by
	Theorem~\ref{thm:theta-open}, $\pf$ is maximal with respect to $\preceq_j$ for
	$i+1\leq j\leq n$ as well. In particular, for $j=n$ we obtain the claim.
\end{proof}

This lemma has an interesting consequence, which we explain before moving on.
While there are nontrivial examples of cosilting complexes in $\D(R)$ with
locally noetherian hearts (see \emph{e.g.}~\cite[Ex.~3.16]{hrbe-pavo-21}), in
our experience they do not seem to enjoy the derived type property, \emph{i.e}
they are not cotilting.
As an application of the topological insight obtained so far,
we are able to show that any cotilting complex with a locally
noetherian heart is in fact trivial in the following sense.

\begin{thm}\label{thm:locally-noetherian-rare}
	If $\Gcal$ is a locally noetherian abelian category such
	that there is a triangle equivalence $\D^{\mathsf{b}}(\Gcal) \simeq
	\D^{\mathsf{b}}(\Mod(R))$ of the bounded derived categories, then $\Gcal \simeq
	\Mod(R)$.
\end{thm}

\begin{proof}
	We prove the claim for $R$ connected. The general case follows easily, see for
	example the first paragraph of \cite[Rmk. 4.10]{hrbe-naka-stov-24}, and references therein.
	A locally noetherian abelian category is Grothendieck \cite[\S2.4]{craw-94}. Therefore, if
	there is an equivalence $\Fsf\colon \D^b(\Gcal)\simeq \D^b(\Mod(R))$, the
	image of the standard $t$-structure of $\D^b(\Gcal)$ is a cotilting
	$t$-structure in $\D^b(\Mod(R))$, whose bounded cotilting complex $c$ is the image of the
	injective cogenerator of $\Gcal$ (\emph{cf.}\ \cite[Thm.~5.3]{psar-vito-18}
	and \cite[Thm.~2.3]{hrbe-mart-24}).
	Let $\Hcal\simeq\Gcal$ be the heart associated to $c$. Since it is locally
	noetherian, any hereditary torsion pair is automatically of finite type
	\cite[Cor.~III.4.1]{gabr-62}. In other words, the Gabriel and the Ziegler
	topologies on $\GSpec(\Hcal)$ coincide. By \cite[Cor.~3.14]{hrbe-pavo-21} and
	\cite[Thm.~4.5(2)]{pavo-vito-21}, the supports of the torsion pairs of finite
	type in $\Hcal$ are precisely the specialisation closed sets, and thus $\Phi_c\colon\Spec(R)\to \GSpec(\Hcal)$ is a homeomorphism, when we endow
	$\Spec(R)$
	with its Hochster topology.
	Lemma~\ref{lemma:maximal-difference} asserts that for any prime ideal $\pf$
	which is maximal in any difference set $V_{i-1}\setminus V_i$ of the
	associated sp-filtration the singleton $\{\pf\}$ is open in
	$\GSpec(\Hcal)$, and therefore in $\Spec(R)$; that is, $\pf$ is a maximal
	prime of $\Spec(R)$. It follows
	that every $V_i$ is a clopen subset of $\Spec(R)$ in the Hochster topology,
	that is, $V_i$ is a union of connected components of $\Spec(R)$ for
	any $i \in \Zbb$. Our assumption that $R$ is connected then gives either
	$V_i=\Spec(R)$ or $V_i=\emptyset$, that is, the $t$-structure is a shift of
	the standard one.
\end{proof}

Note that if $R$ is connected, the proof above guarantees that
$\Hcal=\Mod(R)[i]$ for some $i$. Recall that an object $c$ is
\emph{$\Sigma$-pure-injective} if $c^{(I)}$ is pure-injective for all sets $I$.
We can therefore conclude the following.

\begin{cor}
	Let $c$ be a bounded cotilting complex in $\D(R)$, for $R$ a connected
	commutative noetherian ring. Then $c$ is $\Sigma$-pure-injective if and only
	if $c$ is equivalent to a shift of the injective cogenerator $\omega$ of
	$\Mod(R)$.
\end{cor}

\begin{proof}
	By \cite[Prop.~5.6]{laki-20}\footnote{See the arXiv version
	\texttt{arXiv:1804.01326} for a correction of the published statement.},
	$c$ cotilting is $\Sigma$-pure-injective if and only if
	the corresponding heart $\Hcal$ is locally noetherian. By the observation
	above, the Proposition shows that this is the
	case if and only if $\Hcal=\Mod(R)[i]$ for some $i\in\Zbb$, which means that
	$c$ is equivalent to $\omega[i]$.
\end{proof}

\subsection{Perfect mutation revisited}

As explained above, to completely determine the closure order $\preceq$, we need
to characterise, for every $0\leq i\leq n$, the relations $\pf\preceq \qf$ for $\pf\notin V_i\ni \qf$.
A case in which this is easily done is when $V_i$ is a clopen subset for
$\GSpec(\Hcal)$:
indeed, then $V_i^\cpl$ is also open, that is, an upper-set for $\preceq$, and no
such relation can exist. In this subsection we show that this phenomenon is
tightly linked to perfectness of torsion pairs, which will allow us to recognise
it in examples.

Let $W$ be a specialisation-closed subset of $\Spec(R)$, $t=(\Tcal_W,\Fcal_W)$
the associated hereditary torsion pair of finite type in $\Hcal$, and
$\Hcal_t:=\Fcal_W\ast\Tcal_W[-1]$ the right mutation. We know from
Corollary~\ref{cor:strongly-perfect} that if $t$ is strongly perfect, then $W$ is
clopen for $\Hcal_t$. The following result is a partial converse, under the
assumption of derived type.

\begin{prop}\label{prop:perfect-converse}
	Assume that the associated cosilting complex is cotilting, so that $\tau$ is
	of derived type. Then $W$ is clopen in $\GSpec(\Hcal_t)$ if and only if $t$
	is (strongly) perfect.
\end{prop}
\begin{proof}
	The implication $(\Leftarrow)$ is proved in Corollary~\ref{cor:strongly-perfect}, it
	remains to prove the converse.

	Let $\Bcal_W:=\supp^{-1}(W^\cpl)\subseteq\D(R)$. We have an inclusion
	$\Hcal\cap\Bcal_W\subseteq\Ccal_W:=\Tcal_W^{\bot_{0,1}}$. Since $\Bcal_W$ is
	triangulated, $\Hcal\cap\Bcal_W$ is closed under cokernels of monomorphisms;
	therefore if this inclusion is an equality, then $t$ is perfect, and even
	strongly perfect, since we are in the cotilting case (see the paragraph below
	Definition~\ref{dfn:strongly-perfect}). Assume
	instead that $t$ is not perfect, so that there must be $0\neq x\in\Ccal_W$ with
	$\supp(x)\nsubseteq W^\cpl$. Let $e:=E(x)\in\Fcal_W$ be the injective envelope of
	$x$; since $e\in\Prod(c)$, by the cotilting assumption, we have that
	$e\in\Bcal_W$ by Lemma~\ref{lemma:BW-FW}; that is, $\supp(e)\subseteq W^\cpl$. If
	we consider the short exact sequence in $\Hcal$
	\[0\to x\to e\to \mho^1(x)\to 0,\]
	we have that $\mho^1(x)\in\Fcal_W$ by the fact that $x$ and $e$
	belong to $\Ccal_W$. Therefore this is also an exact sequence in $\Hcal_t$.
	Since the support of $e$ is contained in $W^\cpl$ but the support of its
	subobject $x\subseteq e$ is not, this shows that $W^\cpl$ is not open in
	$\GSpec(\Hcal_t)$.
\end{proof}

Recall that if $W\subseteq\Spec(R)$ is specialisation-closed, the torsion pair
$(\Tcal_W,\Fcal_W)$ in $\Mod(R)$ is (strongly) perfect if and only if
$W^\cpl\subseteq\Spec(R)$ is a \emph{coherent} subset \cite[Thm.~3.1]{krau-08}. If this is the
case, then any minimal element of $W$ must have height at most one
\cite[Thm.~1.2]{ange-mark-stov-taka-vito-20}. Any subset of $\Spec(R)$ is
coherent if $R$ is of Krull dimension at most one \cite[Thm.~1.2]{krau-08}.

\begin{cor}\label{cor:aperture-zero}
	Let $(R,\mf)$ be a local noetherian domain, and let $\emptyset\neq V_0\subsetneq\Spec(R)$ be a
	specialisation-closed subset. Let $\tau$ be the $t$-structure associated to the
	filtration
	\[\Spec(R)\supsetneq V_0 \supsetneq\emptyset.\]
	For the zero ideal $\of\in\Spec(R)$, we have
	$\of\preceq\mf$ if and only if $V_0^\cpl\subseteq\Spec(R)$ is not coherent.
\end{cor}

\begin{proof}
	The $t$-structure $\tau$ (with heart denoted by $\Hcal$) is obtained from the
	standard one by an HRS-tilt at a hereditary torsion pair,
	$t=(\Tcal_{V_0},\Fcal_{V_0})$, and therefore it is of derived type by
	\cite[Cor.~5.12]{pavo-vito-21}.
	Then, by Proposition~\ref{prop:perfect-converse},
	$V_0\subseteq\GSpec(\Hcal)$ is clopen
	if and only if $t$ is perfect, that is, if and
	only if $V_0^\cpl\subseteq\Spec(R)$ is coherent. Now, if this is the case, as we have observed
	at the beginning of the subsection, $V_0^\cpl$ is an upper-set for $\preceq$.
	Since by assumption on $V_0$ we have $\of\in V_0^\cpl$ and $\mf\notin
	V_0^\cpl$, we have $\of\npreceq\mf$. If $V_0^\cpl$ is not coherent, on the other
	hand, $V_0^\cpl$ is not an upper-set for $\preceq$, so there exist $\pf\in
	V_0^\cpl$ and $\qf\in V_0$ with $\pf\preceq\qf$. Since on both differences
	$V_0^\cpl=\Spec(R)\setminus V_0$ and $V_0=V_0\setminus\emptyset$ the order
	$\preceq$ coincides with inclusion by Lemma~\ref{lemma:bound-topology}, we have $\of\preceq
	\pf$ and $\qf\preceq\mf$, from which we obtain that $\of\preceq\mf$.
\end{proof}

In Proposition~\ref{prop:discrete-mutation}, we saw that $W$ is clopen for $\Hcal_t$ also if $W^\cpl$ is
a discrete subspace for $\Hcal$. In particular, when $\Hcal$ comes from a
cotilting complex, by Proposition~\ref{prop:perfect-converse}, if
$W^\cpl$ is discrete for $\Hcal$, then $t$ is (strongly) perfect. In fact, this
is also true without the cotilting assumption.

\begin{prop}\label{prop:zerodim-perfect}
	Let $\Hcal$ be the heart associated to a bounded cosilting complex in $\D(R)$.
	Assume that $W^\cpl$ is a discrete subspace of $\GSpec(\Hcal)$. Then $t$ is
	perfect.
\end{prop}

\begin{proof}
	We will show that $t$ is perfect using
	Proposition~\ref{prop:semi-artinian-perfect}. By
	Proposition~\ref{prop:W-giraud-finite}, $t$ is Giraud-finite. It remains to
	check the condition $\Ext_\Hcal^2(\Tcal,\Rsf(s)) = 0$ for any
	simple object $s \in \Gcal/\Tcal$. By
	Proposition~\ref{prop:residues-gabriel-simples}, there is $\pf \in W^\cpl$
	such that $\Rsf(s) \simeq s(\pf)$. Then we have:
	\[\Ext_\Hcal^2(\Tcal,s(\pf))\hookrightarrow \Hom_{\D(R)}(\Tcal,s(\pf)[2])\subseteq
	\Hom_{\D(R)}(\supp^{-1}W,\supp^{-1}W^c)=0.\]
	We conclude that $t$ is perfect.
\end{proof}

Lastly, we have another condition that implies that
a specialisation closed set $W\subseteq\Spec(R)$ is clopen in the
Gabriel spectrum of a cosilting heart $\Hcal$.

\begin{lemma}\label{lemma:loccoh-restrict-open}
	If the local cohomology functor $\Rbb\Gamma_W$ restricts to $\Hcal$, then
	$W$ is clopen in $\GSpec(\Hcal)$.
\end{lemma}

\begin{proof}
	The restriction of $\Rbb\Gamma_W$ to $\Hcal$ is an exact and coproduct
	preserving endofunctor on $\Hcal$ whose kernel is precisely
	$\supp^{-1}(W^\cpl) \cap \Hcal$. It follows that $\supp^{-1}(W^\cpl) \cap
	\Hcal$ is a hereditary torsion class in $\Hcal$, and thus $W^\cpl$ is open in
	$\GSpec(\Hcal)$.
\end{proof}

\begin{ex}\label{example:twodimtwostep}
	Let $(R,\mf)$ be a local ring of Krull dimension two and consider the filtration
	\[(V_{-1}\supseteq V_0\supseteq V_1\supseteq V_2) \;=\; (\Spec(R)\supseteq
	\{\mf\} \supseteq \{\mf\}\supseteq\emptyset).\]
	We claim that the local cohomology functor
	$\Rbb\Gamma_{\{\mf\}}$
	restricts to $\Hcal$. Recalling the description of
	the $t$-structure $\tau$ associated to the filtration (see
	\S\ref{subsec:residues-gabriel-simples}),
	the fact that this functor restricts to the coaisle $\Vcal$ of $\tau$ is immediate, so
	it is sufficient to check that the aisle
	\[\Ucal = \{\,x\in\D(R)\mid \supp(H^i(x))\subseteq V_{i+1}\,\}\]
	is closed under $\Rbb\Gamma_{\{\mf\}}$. Let
	$x \in \Ucal$ and consider the standard truncation triangle
	$$\tau^{<0} x \to x \to \tau^{\geq 0} x \to (\tau^{<0} x)[1]$$ 
	By the description of $\Ucal$, we
	have $\Rbb\Gamma_{\{\mf\}}(\tau^{\geq 0} x) \simeq \tau^{\geq 0} x\in\Ucal$. On the
	other hand, we have 
	\[\Rbb\Gamma_{\{\mf\}}(\tau^{<0} x) \in \D^{<2}\cap\supp^{-1}(\{\mf\})\subseteq\Ucal\] 
	by Grothendieck's Vanishing Theorem \cite[Prop.~20.20.7, 02UZ]{stacks}. This shows that
	$\Rbb\Gamma_{\{\mf\}}(x)\in\Ucal$.

	Therefore, $\Rbb\Gamma_{\{\mf\}}$ restricts to $\Hcal$, and so $\{\mf\}^\cpl$ is open in
	$\GSpec(\Hcal)$ by Lemma~\ref{lemma:loccoh-restrict-open}. Together with
	Lemma~\ref{lemma:bound-topology}, this fully describes the topology on
	$\GSpec(\Hcal)$. Note that, by Proposition~\ref{prop:perfect-converse}, the
	hereditary torsion pair supported on $\{\mf\}$ in the heart $\Hcal_1$
	corresponding to the truncated filtration $\Spec(R) \supseteq V_0
	\supseteq \emptyset$ is perfect. We will compute the Gabriel
	topology of $\GSpec(\Hcal_1)$ in the next subsection (see
	Figure~\ref{figure:2-dim-1-step}).
\end{ex}

\subsection{Spectra of two-term cosilting complexes}

In this subsection, we use Corollary~\ref{cor:aperture-zero} to study the case
of filtrations of length one. The idea is that in order to check whether two
primes $\pf$ and $\qf$ satisfy $\pf\preceq\qf$, we can pass to the local domain
$(R/\pf)_\qf$ and apply the Corollary. It will, thus, be useful to spell out how the
Gabriel topology on $\GSpec(\Hcal)$ behaves with respect to certain
restrictions of scalars. 

Let $\lambda: R \to S$ be a ring epimorphism of
commutative noetherian rings. Then the induced map $\lambda_*: \Spec(S)
\hookrightarrow \Spec(R)$ on Zariski spectra is injective \cite[Lemma~10.107.9,
04VW]{stacks}. Our intermediate sp-filtration $\Spec(R) = V_{-1} \supseteq V_0
\supseteq \cdots \supseteq V_{n-1} \supseteq V_n = \emptyset$ on $\Spec(R)$
induces another intermediate sp-filtration on $\Spec(S)$ by setting
$V^\lambda_i = \lambda_*^{-1}(V_i)$. We denote by $\tau_\lambda$ and
$\Hcal_\lambda$ the induced $t$-structure and heart in $\D(S)$. By \cite[Lemma~10.107.9, 04VW]{stacks}, $\lambda$ induces an isomorphism $k(\pf) \simeq
k(\lambda_*^{-1}(\pf))$ for any $\pf$ in the image of $\lambda_*$. Inspecting
Lemma~\ref{lemma:s(p)-torsion-simple}, we see that the Gabriel-simple objects in
$\Hcal$ and $\Hcal_\lambda$ lying over $\pf$ and $\lambda_*^{-1}(\pf)$,
respectively, may be identified and denoted both simply as $s(\pf)$.
Moreover, by \cite[\S5.2, Thm.~5.7]{brea-hrbe-modo-22}, the $t$-structure $\tau_\lambda$ is associated to a pure-injective cosilting object $c^\lambda := \RHom_R(S,c)$, the derived coextension of scalars of $c$ along $\lambda$.

Under the identification of Theorem~\ref{thm:c(p)-support}(2), the embedding
$\lambda_*$ induces an injective map $\Phi_c\lambda_*\Phi_{c^\lambda}^{-1}\colon \GSpec(c^\lambda)
\hookrightarrow \GSpec(c)$. In the following, we show that this embedding is
topological for a useful class of choices of $\lambda$.

\begin{lemma}\label{lemma:full-support-localisation}
	Let $\qf \in \Spec(R)$, $I \subseteq \qf$ an ideal, and consider the natural
	ring epimorphism $\lambda\colon R \to (R/I)_\qf$. Then the Gabriel
	topology on $\GSpec(c^\lambda)$ is the subspace topology induced by the
	embedding $\Phi_c\lambda_*\Phi_{c_\lambda}^{-1}: \GSpec(c^\lambda) \hookrightarrow \GSpec(c)$.
\end{lemma}

\begin{proof}
	Let us denote by $\preceq_\lambda$ the order on $\Spec((R/I)_\qf)$ associated to
	the Gabriel topology of $\GSpec(c^\lambda)$. Using $\lambda_\ast$, we identify
	$\Spec((R/I)_\qf)$ with the subset $\{\,\pf\in\Spec(R)\mid
	I\subseteq\pf\subseteq\qf\,\}$ of $\Spec(R)$. Given $\pf\subseteq\pf'$
	in this subset, we need to prove that $\pf\preceq\pf'$ if and
	only if $\pf\preceq_\lambda\pf'$.
	Without loss of generality, we may assume that $\pf'=\qf$, otherwise
	consider the composition $\lambda\colon
	R\to (R/I)_\qf\to (R/I)_{\pf'}$.

	To prove the claim, we first show that
	$\Fsf_\lambda:=\Rbb\Hom_R((R/I)_\qf,-)\colon \Prod(c)\to \Prod(c^\lambda)$ sends the
	indecomposable $c(\qf)=\Phi_c(\qf)$ to $c^\lambda(\qf)=\Phi_{c^\lambda}(\qf)$, that is, to the unique
	indecomposable object of $\Prod(c^\lambda)$ having $\qf$ in its support (see
	Theorem~\ref{thm:c(p)-support}).

	To show that $\Fsf_\lambda(c(\qf))$ is indecomposable, we prove that it has a
	local endomorphism ring. We have a chain of natural isomorphisms:
	\begin{align*}
		&\RHom_{(R/I)_\qf}(\Fsf_\lambda(c(\qf)),\Fsf_\lambda(c(\qf))) \simeq \\
		\simeq&\RHom_{R}\left(\RHom_R\left(R/I,c(\qf)\right),c(\qf)\right) \simeq \\
		\simeq&\, R/I \otimes_R^\Lbb \RHom_R\left(c(\qf),c(\qf)\right),
	\end{align*}
	where the first isomorphism comes directly from the derived $\otimes$-$\Hom$ adjunction and
	the second one follows from a standard homological identity (see for example
	\cite[Proposition 2.2]{chri-holm-09}) using the fact that
	$R/I$ is finitely generated and that $c(\qf)$ is isomorphic in $\D(R)$ to a
	bounded complex of injective $R$-modules. As $c(\qf)$ is in $\Prod(c)$, the
	positive cohomology of $\RHom_R(c(\qf),c(\qf))$ vanishes.
	The isomorphism of objects of $\D((R/I)_\qf)$ is in fact an isomorphism of
	differential graded algebras. This can be checked using the explicit
	expressions of the isomorphisms above, noting that, once we fix a
	bounded complex of injectives $E^\bullet$ quasi-isomorphic to $c(\qf)$, we
	have:
	\begin{enumerate}[label=\roman*)]
		\item $\Fsf_\lambda(c(\qf))\simeq \Hom^\bullet_R((R/I)_\qf,E^\bullet)$ is a
			bounded complex of injective $(R/I)_\qf$-modules;
		\item $\RHom_R(c(\qf),c(\qf))\simeq \Hom^\bullet_R(E^\bullet,E^\bullet)$ is a
			bounded complex of flat $R$-modules (in fact, $R_\qf$-modules). Indeed,
			each component of $\Hom^\bullet_R(E^\bullet,E^\bullet)$ is a direct sum of
			$\Hom$-spaces between injectives, which are flat by the well-known module
			version of the homological identity used above.
	\end{enumerate}
	Applying the
	standard cohomology $H^0$ to the chain of isomorphisms above, we thus obtain
	an isomorphism of $R$-algebras:
	\[\End_{\D((R/I)_\qf)}(\Fsf_\lambda(c(\qf))) \simeq R/I \otimes_R
	\End_{\D(R)}(c(\qf)).\]
	Since $\End_{\D(R)}(c(\qf))$ is isomorphic to the endomorphism ring of an
	indecomposable object of the Grothendieck category $\Hcal$, it is a local
	$R$-algebra, and thus also its quotient ring
	\[R/I \otimes_R \End_{\D(R)}(c(\qf)) \simeq
	\End_{\D(R)}(c(\qf))/I\End_{\D(R)}(c(\qf))\]
	is a local ring, establishing finally that
	$\End_{\D((R/I)_\qf)}(\Fsf_\lambda(c(\qf)))$ is local.

	The last task is to show that the maximal ideal $\qf$ of $\Spec((R/I)_\qf)$ belongs to the support of
	$\Fsf_\lambda(c(\qf))\in\D((R/I)_\qf)$.
	By Neeman's classification of smashing subcategories (see
	\S\ref{subsubsec:GSpec-Spec}), for this it suffices to check the non-vanishing
	of $\RHom_{(R/I)_\qf}(k(\qf),c(\qf)^\lambda)$. By adjunction,
	$\RHom_{(R/I)_\qf}(k(\qf),\Fsf_\lambda(c(\qf))) \simeq
	\RHom_{R_\qf}(k(\qf),c(\qf))$, and the latter cannot
	vanish since the maximal ideal $\qf$ of $R_\qf$ lies in the support of $c(\qf)$
	by Theorem~\ref{thm:c(p)-support}.
	
	Using that $s(\pf)$ is naturally a Gabriel-simple object in $\Hcal_\lambda$ as
	discussed above, we obtain using standard adjunction arguments the following
	chain of isomorphisms:
	\begin{align*}
		&\Hom_{\Hcal}(s(\pf),H_c^0(c(\qf)))
			\simeq \Hom_{\D(R)}(s(\pf),c(\qf)) \simeq \\
		\simeq& \Hom_{\D((R/I)_\qf)}(s(\pf),\Fsf_\lambda(c(\qf))) \simeq
			\Hom_{\Hcal_\lambda}(s(\pf),H^0_{c^\lambda}(c(\qf)))).
	\end{align*}
	By Lemma~\ref{lemma:hts-topology-characterisation} and
	Proposition~\ref{prop:s(p)-inj-envelope}, $\pf \preceq \qf$ if and only if the
	first $\Hom$-group does not vanish, while $\pf
	\preceq_\lambda \qf$ if and only if the last $\Hom$-group
	does not vanish, finishing the proof.
\end{proof}

From now on, we restrict to the case of an
sp-filtration $\Spec(R)\supsetneq V_0\supsetneq\emptyset$ of length one.
In another words, $\Hcal$ is the right tilt of the standard heart $\Mod(R)$ at
$V_0$. In this case, Lemma~\ref{lemma:full-support-localisation} gives us a
recipe to compute the Gabriel topology of $\GSpec(\Hcal)$. Recall that in
view of Theorem~\ref{thm:theta-open} and Lemma~\ref{lemma:bound-topology}, to
recover the topology it suffices to determine when $\pf \preceq \qf$ for $\pf
\subseteq \qf$ such that $\pf \not\in V_0$ and $\qf \in V_0$.

\begin{thm}\label{thm:onestep-recipe}
	Let $V_0\subseteq\Spec(R)$ a specialisation-closed subset, and let $\tau$ be
	the $t$-structure, with heart $\Hcal$, associated to the filtration
	\[\Spec(R) \supsetneq V_0\supsetneq\emptyset.\]
	Denote by $\preceq$ the closure order on $\Spec(R)$ induced by
	$\GSpec(\Hcal)$. Given $\pf,\qf\in\Spec(R)$, we have $\pf\preceq\qf$ if and
	only if $\pf\subseteq\qf$ and the following condition holds: for the natural
	ring epimorphism $\lambda\colon R\to (R/\pf)_\qf$, the subset
	$\lambda_\ast^{-1}(V_0)$ has coherent complement in $\Spec((R/\pf)_\qf)$.
\end{thm}

\begin{proof}
	\noindent\textbf{Case $1$}: \textit{$\pf,\qf$ belong both to $V_0$ or both to
	$V_0^\cpl$}. In this case,
	by Lemma~\ref{lemma:bound-topology}, we have $\pf\preceq\qf$ if and only if
	$\pf\subseteq\qf$. Observe that in this case $\lambda_\ast^{-1}(V_0)$ is
	either $\Spec((R/I)_\qf)$ or $\emptyset$, respectively, which always has
	coherent complement.

	\noindent\textbf{Case $2$}: \textit{$\pf\in V_0$ and $\qf\in V_0^\cpl$}. In this case, by 
	Lemma~\ref{lemma:preceq-finer}, we cannot have $\pf\preceq\qf$, as
	$\pf\nsubseteq\qf$.

	\noindent\textbf{Case $3$}: \textit{$\pf\in V_0^\cpl$ and $\qf\in V_0$}.
	By Lemma~\ref{lemma:full-support-localisation}, we have $\pf \preceq \qf$ if
	and only if $\lambda_*^{-1}(\pf) \preceq_\lambda \lambda_*^{-1}(\qf)$. Since
	$\lambda_*^{-1}(\pf)$ coincides with the zero prime and $\lambda_*^{-1}(\qf)$
	with the maximal prime of the local domain $(R/\pf)_\qf$, the rest follows
	from Corollary~\ref{cor:aperture-zero}.
\end{proof}

\begin{rmk}
	Given primes $\pf\subseteq\qf$, the ring epimorphism $\lambda\colon R\to
	(R/\pf)_\qf$ induces a fully faithful embedding
	$\Mod((R/\pf)_\qf)\hookrightarrow \Mod(R)$. If $(\Tcal,\Fcal)$ is a
	(hereditary) torsion
	pair in $\Mod(R)$, its restriction to $\Mod((R/\pf)_\qf)$ by intersection is
	a (hereditary) torsion pair in $\Mod((R/\pf)_\qf)$. Indeed, the torsion part
	of an $(R/\pf)_\qf$-module for the pair $(\Tcal,\Fcal)$ is an
	$(R/\pf)_\qf$-module, as seen in the proof of \cite[Cor.~2.11]{pavo-25}. Then
	so is the torsion-free part, as $\Mod((R/\pf)_\qf)$ is closed under cokernels
	in $\Mod(R)$.

	This restriction applied to the hereditary torsion pair of $\Mod(R)$ corresponding to
	$V_0$ yields the hereditary torsion pair of $\Mod((R/\pf)_\qf)$ corresponding
	to $\lambda_\ast^{-1}(V_0)$, as can be easily seen by checking which residue
	fields lie in the torsion class.
	Hence, the theorem says that
	$\pf\preceq\qf$ if and only if the restriction of $(\Tcal,\Fcal)$ to
	$\Mod((R/\pf)_\qf)$ is perfect in this category.
\end{rmk}

\begin{figure}[h]
\begin{tikzpicture}
	\begin{scope}
		\path[fill=gray!15!white,rounded corners=8pt] (-2,-2) rectangle (2,2);
		\path[clip,rounded corners=8pt] (-2,-2) rectangle (2,2);
		\path[fill=gray!70!white,rounded corners=8pt] (-3,-1) -- (-1,.5) -- (1,-.5) -- (3,1) -- (3,3) -- (-3,3) -- cycle;
		\path[draw,rounded corners=8pt] (-2,-2) rectangle (2,2);
		\node at (-1.5,1.5) {$V_0$};
		\node at (1.5,-1.5) {$V_0^\cpl$};
		\path[draw,fill=white,rounded corners=2pt,shift={+(1,.9)}] (0,-.9) -- (.6,0) -- (0,.9) -- (-.6,0) -- cycle;
		{
			\path[clip,rounded corners=8pt] (-3,-1) -- (-1,.5) -- (1,-.5) -- (3,1) -- (3,3) -- (-3,3) -- cycle;
			\path[fill=black,rounded corners=2pt,shift={+(1,.9)}] (0,-.9) -- (.6,0) -- (0,.9) -- (-.6,0) -- cycle;
		}
	\end{scope}
	\begin{scope}[xshift=5cm]
		\path[fill=gray!15!white,rounded corners=8pt] (-2,-2) rectangle (2,2);
		\path[clip,rounded corners=8pt] (-2,-2) rectangle (2,2);
		\path[fill=gray!70!white,rounded corners=8pt] (-3,-1) -- (-1,.5) -- (1,-.5) -- (3,1) -- (3,3) -- (-3,3) -- cycle;
		\path[draw,rounded corners=8pt] (-2,-2) rectangle (2,2);
		\path[draw,fill=white,rounded corners=2pt] (0,-.9) -- (.6,0) -- (0,.9) -- (-.6,0) -- cycle;
		\node at (-1.5,1.5) {$V_0$};
		\node at (1.5,-1.5) {$V_0^\cpl$};
		{
			\path[clip,rounded corners=8pt] (-3,-1) -- (-1,.5) -- (1,-.5) -- (3,1) -- (3,3) -- (-3,3) -- cycle;
			\path[fill=black,rounded corners=2pt] (0,-.9) -- (.6,0) -- (0,.9) -- (-.6,0) -- cycle;
		}
	\end{scope}
	\begin{scope}[xshift=10cm]
		\path[fill=gray!15!white,rounded corners=8pt] (-2,-2) rectangle (2,2);
		\path[clip,rounded corners=8pt] (-2,-2) rectangle (2,2);
		\path[fill=gray!70!white,rounded corners=8pt] (-3,-1) -- (-1,.5) -- (1,-.5) -- (3,1) -- (3,3) -- (-3,3) -- cycle;
		\path[draw,rounded corners=8pt] (-2,-2) rectangle (2,2);
		\path[draw,fill=white,rounded corners=2pt,shift={+(-1,-.9)}] (0,-.9) -- (.6,0) -- (0,.9) -- (-.6,0) -- cycle;
		\node at (-1.5,1.5) {$V_0$};
		\node at (1.5,-1.5) {$V_0^\cpl$};
		{
			\path[clip,rounded corners=8pt] (-3,-1) -- (-1,.5) -- (1,-.5) -- (3,1) -- (3,3) -- (-3,3) -- cycle;
			\path[fill=black,rounded corners=2pt,shift={+(-1,-.9)}] (0,-.9) -- (.6,0) -- (0,.9) -- (-.6,0) -- cycle;
		}
	\end{scope}
\end{tikzpicture}
	\caption{
		The squares represent $\Spec(R)$, with the specialisation-closed subset
		$V_0$ shaded in a darker gray. The diamonds represent the image of the embedding
		$\lambda_\ast\colon\Spec((R/\pf)_\qf)\hookrightarrow \Spec(R)$, depending on
		the three possible cases:  $\pf,\qf\in V_0$ (left),  $\pf\in V_0^\cpl$ and
		$\qf\in V_0$ (middle), $\pf,\qf\in V_0^\cpl$ (right). According to
		Theorem~\ref{thm:onestep-recipe}, we have $\pf\preceq\qf$ in the heart
		$\Hcal$ of the mutation at $V_0$ if and
		only if the white region is coherent in the diamond. Notice that this is
		trivially true in the first and third case, in accordance with
		Lemma~\ref{lemma:bound-topology}.
	}
	\label{figure:diamonds}
\end{figure}
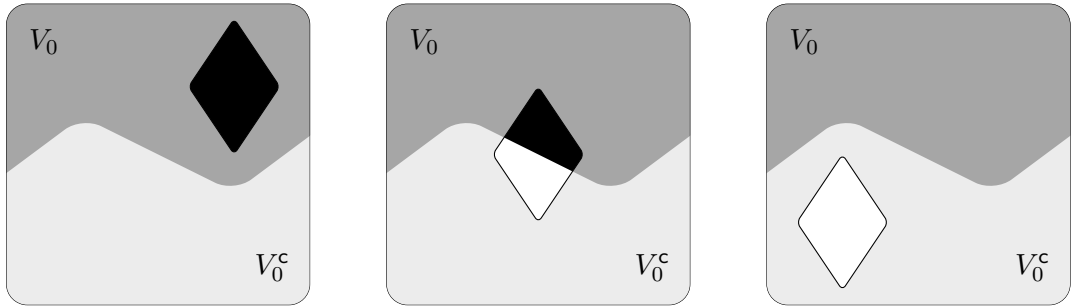

Note that, in general, one cannot read the coherence of a complement of a
specialisation-closed subset just from the Zariski topology, see
Example~\ref{example:Nagata-ring} below. We now use our results to compute the
closure order of the Gabriel spectrum of some two-term cosilting complexes.

\begin{ex}\label{example:H1-twodim}
	Let us compute the Gabriel topology of the one-step heart $\Hcal_1$ of
	Example~\ref{example:twodimtwostep}. Here, $(R,\mf)$ is a
	two-dimensional local ring and $V_0 = V(\mf)$. Let $\pf\neq \mf$ be a prime ideal of $R$. Employing Theorem~\ref{thm:onestep-recipe}, we
	see that $\pf \preceq \mf$ if and only if $\dim(R/\pf) = 2$. Indeed, if
	$\lambda\colon R \to R/\pf$ is the natural projection then $\lambda_*^{-1}(V_0)$ is
	precisely the singleton subset $\{\mf\}\subseteq \Spec(R/\pf)$.
	As mentioned, if the complement of $\{\mf\}$ is coherent, then $\mf$
	must have height at most one, and therefore $\dim(R/\pf)<2$. Conversely, if
	$\dim(R/\pf)<2$, any subset of $\Spec(R/\pf)$ is coherent. Therefore,
	$\pf\preceq\mf$ if and only if $\pf$ is a minimal prime. Together with
	Lemma~\ref{lemma:bound-topology}, this provides us with a complete description of
	the closure order, as illustrated in Figure~\ref{figure:2-dim-1-step}.
\end{ex}

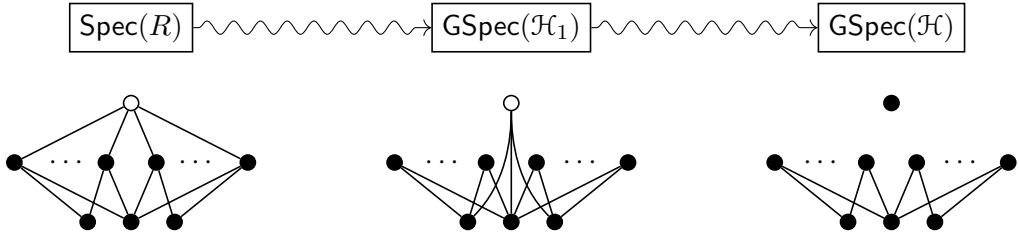
\begin{figure}[h]
	\centering
\usetikzlibrary{positioning}
\usetikzlibrary{decorations.pathmorphing}
\usetikzlibrary{arrows.meta}
\begin{tikzpicture}
	[black/.style={circle,draw=black,semithick,fill=black,inner sep=2pt},
	 white/.style={circle,draw=black,semithick,fill=white,inner sep=2pt},
	 dots/.style={draw=none,fill=none},
	 label/.style={rectangle,draw=black,semithick,},
	 HRS/.style={decoration={snake,pre length=2pt,post length=2pt}},
	 every edge quotes/.style={},
	 node distance=14pt and 10pt]
	\node[label] (R) at (0,2) {$\Spec(R)$};
	\node[white] (m) at (0,1) {};

	\node[black] (q-1) [below left=18pt and 5pt of m] {};
	\node[black] (q1) [below right=18pt and 5pt of m] {};
	\node[black] (q-2) [left=28pt of q-1] {};
	\node[black] (q2) [right=28pt of q1] {};
	\path (q-2) edge [draw=none] node {$\;\;\cdots$} (q-1);
	\path (q2) edge [draw=none] node {$\cdots\;\,$} (q1);

	\node[black] (p0) [below right=18pt and 5pt of q-1] {};
	\node[black] (p1) [right=of p0] {};
	\node[black] (p-1) [left=of p0] {};

	\draw[semithick] (q-1) -- (m);
	\draw[semithick] (q-2) -- (m);
	\draw[semithick] (q1) -- (m);
	\draw[semithick] (q2) -- (m);

	\draw[semithick] (p-1) -- (q-2) -- (p0) -- (q-1) -- (p-1);
	\draw[semithick] (p1) -- (q2) -- (p0) -- (q1) -- (p1);

	\node[label] (H1) at (5,2) {$\GSpec(\Hcal_1)$};
	\node[white] (am) at (5,1) {};
	\node[black] (aq-1) [below left=18pt and 5pt of am] {};
	\node[black] (aq1) [below right=18pt and 5pt of am] {};
	\node[black] (aq-2) [left=28pt of aq-1] {};
	\node[black] (aq2) [right=28pt of aq1] {};
	\path (aq-2) edge [draw=none] node {$\;\cdots$} (aq-1);
	\path (aq2) edge [draw=none] node {$\cdots\,$} (aq1);

	\node[black] (ap0) [below right=18pt and 5pt of aq-1] {};
	\node[black] (ap1) [right=of ap0] {};
	\node[black] (ap-1) [left=of ap0] {};

	\draw[semithick] (ap-1) .. controls +(45:.7) and +(270:.5) .. (am);
	\draw[semithick] (ap0) -- (am);
	\draw[semithick] (ap1) .. controls +(135:.7) and +(270:.5) .. (am);

	\draw[semithick] (ap-1) -- (aq-2) -- (ap0) -- (aq-1) -- (ap-1);
	\draw[semithick] (ap1) -- (aq2) -- (ap0) -- (aq1) -- (ap1);

	\node[label] (H) at (10,2) {$\GSpec(\Hcal)$};
	\node[black] (bm) at (10,1) {};
	\node[black] (bq-1) [below left=18pt and 5pt of bm] {};
	\node[black] (bq1) [below right=18pt and 5pt of bm] {};
	\node[black] (bq-2) [left=28pt of bq-1] {};
	\node[black] (bq2) [right=28pt of bq1] {};
	\path (bq-2) edge [draw=none] node {$\,\cdots$} (bq-1);
	\path (bq2) edge [draw=none] node {$\cdots$} (bq1);

	\node[black] (bp0) [below right=18pt and 5pt of bq-1] {};
	\node[black] (bp1) [right=of bp0] {};
	\node[black] (bp-1) [left=of bp0] {};

	\draw[semithick] (bp-1) -- (bq-2) -- (bp0) -- (bq-1) -- (bp-1);
	\draw[semithick] (bp1) -- (bq2) -- (bp0) -- (bq1) -- (bp1);

	\draw [-{Computer Modern Rightarrow[length=3pt]},decorate,HRS] (R) -- (H1);
	\draw [-{Computer Modern Rightarrow[length=3pt]},decorate,HRS] (H1) -- (H);
\end{tikzpicture}

	\vskip 1em
	\caption{
		The Gabriel topologies of Examples~\ref{example:twodimtwostep} and
		~\ref{example:H1-twodim}, illustrated by the Hasse quiver of their closure
		order. We depict the primes layered by height. Moving to the right, we tilt
		at the specialization closed set consisting of the white points, twice at
		the maximal prime in this case. After the first step, the primes of height
		$1$ become open points, while the closure $\gncl_\preceq(\mf)$ of the
		maximal prime still contains every minimal prime. After the second tilt,
		$\{\mf\}$ becomes clopen.
		}
	\label{figure:2-dim-1-step}
\end{figure}

\begin{ex}\label{example:dim3}
	Let $R$ be a commutative noetherian ring of Krull dimension
	three. Consider the specialisation-closed subset $V_0:=\{\,\pf\in\Spec(R)\mid
	\height\pf\geq 2\,\}$, and the sp-filtration $\Spec(R)\supsetneq
	V_0\supsetneq \emptyset$. The corresponding $t$-structure is obtained by
	right-mutating the standard $t$-structure at the hereditary torsion pair of
	finite type with support $V_0$. This pair does not fit in either of the two
	cases of \S\ref{subsec:disconnect}: indeed, on one hand
	$V_0^\cpl\subseteq\Spec(R)$ is not
	a discrete subspace. On the other hand, 
	the corresponding torsion pair is not perfect because $V_0^\cpl$ is not
	coherent, as the minimal primes of $V_0$ have height two in $\Spec(R)$.

	We determine the topology on $\GSpec(\Hcal)$ for the heart $\Hcal$ corresponding to this
	filtration. Let $\pf \subseteq \qf$ be such that $\qf \in
	V_0$ and $\pf \not\in V_0$. If $\pf$ is of height one then $\pf$ is maximal in
	$V_0^\cpl$, and thus $\pf \not\preceq \qf$ by
	Lemma~\ref{lemma:maximal-difference}. Alternatively this follows also from
	Theorem~\ref{thm:onestep-recipe}: with the notation used there,
	$\lambda_\ast^{-1}(V_0)$ is the set of all
	nonzero primes over the domain $(R/\pf)_\qf$ and thus it has coherent
	complement. Let now $\pf$ be a minimal prime: we claim that $\pf
	\preceq \qf$. Indeed, this follows from Theorem~\ref{thm:onestep-recipe}, as
	the minimal elements of the restriction of $V_0$ to $\Spec((R/\pf)_\qf)$ have all
	height two. Again, together with Lemma~\ref{lemma:bound-topology} this
	completes the description of $\GSpec(\Hcal)$, see Figure~\ref{figure:3-dim}.
\end{ex}

\begin{figure}[ht]
\begin{tikzpicture}
	[black/.style={circle,draw=black,semithick,fill=black,inner sep=2pt},
	 white/.style={circle,draw=black,semithick,fill=white,inner sep=2pt},
	 dots/.style={draw=none,fill=none},
	 label/.style={rectangle,draw=black,semithick,},
	 HRS/.style={decoration={snake,pre length=2pt,post length=2pt}},
	 every edge quotes/.style={},
	 node distance=14pt and 10pt]
	\node[label] (R) at (0,2) {$\Spec(R)$};
	\node[white] (m) at (0,1) {};

	\node[white] (q0) [below=18pt of m] {};
	\node[white] (q-1) [left=14pt of q0] {};
	\node[white] (q1) [right=14pt of q0] {};
	\node[white] (q-2) [left=14pt of q-1] {};
	\node[white] (q2) [right=14pt of q1] {};
	\node[white] (q-3) [left=28pt of q-2] {};
	\node[white] (q3) [right=28pt of q2] {};
	\path (q-3) edge [draw=none] node {$\;\;\cdots$} (q-2);
	\path (q3) edge [draw=none] node {$\cdots\;\,$} (q2);

	\node[black] (p-1) [below left=18pt and 5pt of q0] {};
	\node[black] (p1) [below right=18pt and 5pt of q0] {};
	\node[black] (p-2) [left=14pt of p-1] {};
	\node[black] (p2) [right=14pt of p1] {};
	\node[black] (p-3) [left=28pt of p-2] {};
	\node[black] (p3) [right=28pt of p2] {};
	\path (p-3) edge [draw=none] node {$\,\cdots$} (p-2);
	\path (p3) edge [draw=none] node {$\cdots$} (p2);

	\node[black] (r) [below right=18pt and 5pt of p-1] {};

	\draw[semithick] (q0) -- (m);
	\draw[semithick] (q-1) -- (m);
	\draw[semithick] (q-2) -- (m);
	\draw[semithick] (q-3) -- (m);
	\draw[semithick] (q1) -- (m);
	\draw[semithick] (q2) -- (m);
	\draw[semithick] (q3) -- (m);

	\draw[semithick] (q-2) -- (p-3) -- (q-3) -- (p-2);
	\draw[semithick] (p-3) -- (q-1);
	\draw[semithick] (p-1) -- (q-2) -- (p-2) -- (q-1) -- (p-1) -- (q0);
	\draw[semithick] (p1) -- (q2) -- (p2) -- (q1) -- (p1) -- (q0);
	\draw[semithick] (p-2) -- (q0) -- (p2);
	\draw[semithick] (q2) -- (p3) -- (q3) -- (p2);
	\draw[semithick] (p3) -- (q1);

	\draw[semithick] (r) -- (p-3);
	\draw[semithick] (r) -- (p-2);
	\draw[semithick] (r) -- (p-1);
	\draw[semithick] (r) -- (p1);
	\draw[semithick] (r) -- (p2);
	\draw[semithick] (r) -- (p3);

	\node[label] (H) at (8,2) {$\GSpec(\Hcal)$};
	\node[black] (am) at (8,1) {};

	\node[black] (aq0) [below=18pt of am] {};
	\node[black] (aq-1) [left=14pt of aq0] {};
	\node[black] (aq1) [right=14pt of aq0] {};
	\node[black] (aq-2) [left=14pt of aq-1] {};
	\node[black] (aq2) [right=14pt of aq1] {};
	\node[black] (aq-3) [left=28pt of aq-2] {};
	\node[black] (aq3) [right=28pt of aq2] {};
	\path (aq-3) edge [draw=none] node {$\;\;\,\cdots$} (aq-2);
	\path (aq3) edge [draw=none] node {$\cdots\;\;$} (aq2);

	\node[black] (ap-1) [below left=18pt and 12pt of aq0] {};
	\node[black] (ap1) [below right=18pt and 12pt of aq0] {};
	\node[black] (ap-2) [left=14pt of ap-1] {};
	\node[black] (ap2) [right=14pt of ap1] {};
	\node[black] (ap-3) [left=28pt of ap-2] {};
	\node[black] (ap3) [right=28pt of ap2] {};
	\path (ap-3) edge [draw=none] node {$\,\cdots$} (ap-2);
	\path (ap3) edge [draw=none] node {$\cdots$} (ap2);

	\node[black] (ar) [below right=18pt and 12pt of ap-1] {};

	\draw[semithick] (aq0) -- (am);
	\draw[semithick] (aq-1) -- (am);
	\draw[semithick] (aq-2) -- (am);
	\draw[semithick] (aq-3) -- (am);
	\draw[semithick] (aq1) -- (am);
	\draw[semithick] (aq2) -- (am);
	\draw[semithick] (aq3) -- (am);

	\draw[semithick] (ar) -- (ap-3);
	\draw[semithick] (ar) -- (ap-2);
	\draw[semithick] (ar) -- (ap-1);
	\draw[semithick] (ar) -- (ap1);
	\draw[semithick] (ar) -- (ap2);
	\draw[semithick] (ar) -- (ap3);

	\draw[semithick] (ar) .. controls +(95:1) .. (aq-3);
	\draw[semithick] (ar) .. controls +(95:1) .. (aq-2);
	\draw[semithick] (ar) .. controls +(92:1) .. (aq-1);
	\draw[semithick] (ar) -- (aq0);
	\draw[semithick] (ar) .. controls +(88:1) .. (aq1);
	\draw[semithick] (ar) .. controls +(85:1) .. (aq2);
	\draw[semithick] (ar) .. controls +(85:1) .. (aq3);

	\draw [-{Computer Modern Rightarrow[length=3pt]},decorate,HRS] (R) -- (H);
\end{tikzpicture}
\centering

	\vskip 1em
	\caption{Illustration of the Gabriel topology of
		Example~\ref{example:dim3} (for clarity, in the case of a local domain). In the
		right tilt at the white points (those of height at least $2$) the prime
		ideals of height $1$ become open points,
		while the minimal ideals are still in the closure of every point.}
	\label{figure:3-dim}
\end{figure}
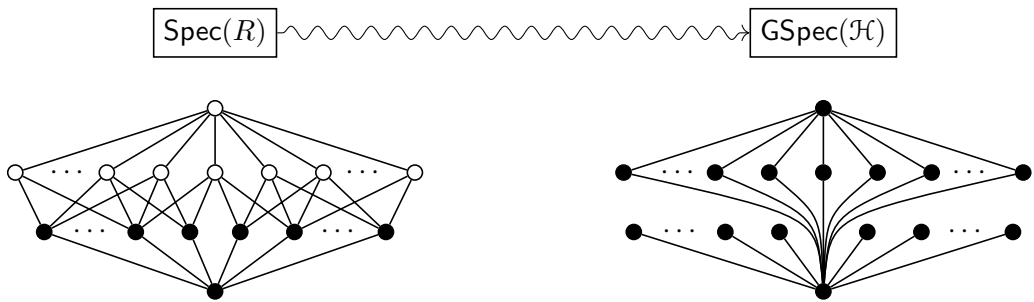

\begin{ex}\label{example:2dim-coherent-noncoherent}
	Let $(R,\mf)$ be a local commutative noetherian domain of Krull dimension two.
	Fix a prime ideal $\pf$ of height one and let us consider the specialisation
	closed subset $V_0:=V(\pf)$ as well as the sp-filtration
	$\Spec(R)\supsetneq V_0\supsetneq \emptyset$. In this case, the Gabriel
	topology differs depending on whether the complement $V_0^\cpl$ is
	coherent or not. In the first case, $V_0^\cpl$ becomes open in $\GSpec(\Hcal)$
	and so the topology us fully described by the disconnection along $V_0$ and
	Lemma~\ref{lemma:bound-topology}.

	Assume now that $V_0^\cpl$ is not coherent (which may happen, see
	Example~\ref{example:Nagata-ring}) and let us compute the Gabriel
	topology of $\GSpec(\Hcal)$ using Theorem~\ref{thm:onestep-recipe}. As the
	zero prime $\of$ is the only non-maximal prime in $V_0^\cpl$, it
	suffices again to compute when $\of \preceq \qf$ for $\qf \in V_0 =
	\{\pf, \mf\}$. As $V_0^\cpl$ is not coherent, we have $\of \preceq
	\mf$. On the other hand, $\of \not\preceq \pf$, as $R_\pf$ is a ring of
	Krull dimension one and thus the the complement of $\lambda_\ast^{-1}(V_0)$
	is coherent in $\Spec(R_\pf)$. 
\end{ex}

It is possible to find rings with homeomorphic Zariski spectra exhibiting the
two contrasting behaviours present in the example above. This shows that, as
mentioned before, the topology of Gabriel spectrum of a two-term cosilting
complex is not completely determined by $\Spec(R)$.

\begin{ex}\label{example:Nagata-ring}
	Let $k$ be a field of countable cardinality and characteristic different from two. Consider the following two local commutative noetherian domains of Krull
	dimension two: the polynomial ring $k[x,y]_{(x,y)}$ in two
	variables localised at the maximal ideal $(x,y)$, and the normal ring $R$ constructed by Nagata \cite{naga-58}, where we also use $k$ as the base field in the construction (this is where $k$ of characteristic different from two is required). These two rings have homeomorphic Zariski spectra. Indeed, the homeomorphism class of the Zariski spectrum of any local commutative noetherian domain of Krull
	dimension two is completely determined by the cardinality of the spectrum. It is well-known that Krull dimension $>1$ prevents the spectrum from being finite. On the other hand, both $k[x,y]_{(x,y)}$ and $R$ are noetherian rings of countable cardinality, and as such they have a countable amount of distinct prime ideals. We note that in $k[x,y]_{(x,y)}$, as in most well-behaved examples
	of local domains of Krull dimension $2$, the Zariski closure $V(\pf)$ of a
	prime ideal $\pf$ of height one will always have a coherent complement (see
	\cite[Prop.~6.4]{ange-mark-stov-taka-vito-20} for details). Nagata's ring $R$, on the other
	hand, has a prime ideal $\pf$ of height one such that $V(\pf)^\cpl$ is not
	coherent \cite[Example 6.5]{ange-mark-stov-taka-vito-20}.
\end{ex}

\begin{figure}[ht]
	\centering
\begin{tikzpicture}
	[black/.style={circle,draw=black,semithick,fill=black,inner sep=2pt},
	 white/.style={circle,draw=black,semithick,fill=white,inner sep=2pt},
	 dots/.style={draw=none,fill=none},
	 label/.style={rectangle,draw=black,semithick,},
	 HRS/.style={decoration={snake,pre length=2pt,post length=2pt}},
	 every edge quotes/.style={},
	 node distance=14pt and 10pt]
	\node[label] (L) at (0,2) {$\Spec(k[x,y]_{(x,y)})$};
	\node[white] (m) at (0,1) {};

	\node[white] (q-1) [below left=18pt and 7pt of m] {};
	\node[black] (q1) [below right=18pt and 7pt of m] {};
	\node[black] (q-2) [left=14pt of q-1] {};
	\node[black] (q2) [right=14pt of q1] {};
	\node[black] (q-3) [left=28pt of q-2] {};
	\node[black] (q3) [right=28pt of q2] {};
	\path (q-3) edge [draw=none] node {$\;\;\,\cdots$} (q-2);
	\path (q3) edge [draw=none] node {$\cdots\;\;$} (q2);

	\node[black] (p) [below right=18pt and 7pt of q-1] {};

	\draw[semithick] (q-1) -- (m);
	\draw[semithick] (q-2) -- (m);
	\draw[semithick] (q-3) -- (m);
	\draw[semithick] (q1) -- (m);
	\draw[semithick] (q2) -- (m);
	\draw[semithick] (q3) -- (m);

	\draw[semithick] (q-1) -- (p);
	\draw[semithick] (q-2) -- (p);
	\draw[semithick] (q-3) -- (p);
	\draw[semithick] (q1) -- (p);
	\draw[semithick] (q2) -- (p);
	\draw[semithick] (q3) -- (p);

	\node[label] (H) at (8,2) {$\GSpec(\Hcal)$};
	\node[black] (am) at (8,1) {};

	\node[black] (aq-1) [below left=18pt and 7pt of am] {};
	\node[black] (aq1) [below right=18pt and 7pt of am] {};
	\node[black] (aq-2) [left=14pt of aq-1] {};
	\node[black] (aq2) [right=14pt of aq1] {};
	\node[black] (aq-3) [left=28pt of aq-2] {};
	\node[black] (aq3) [right=28pt of aq2] {};
	\path (aq-3) edge [draw=none] node {$\;\;\cdots$} (aq-2);
	\path (aq3) edge [draw=none] node {$\cdots\,\;$} (aq2);

	\node[black] (ap) [below right=18pt and 7pt of aq-1] {};

	\draw[semithick] (aq-1) -- (am);

	\draw[semithick] (aq-2) -- (ap);
	\draw[semithick] (aq-3) -- (ap);
	\draw[semithick] (aq1) -- (ap);
	\draw[semithick] (aq2) -- (ap);
	\draw[semithick] (aq3) -- (ap);

	\draw [-{Computer Modern Rightarrow[length=3pt]},decorate,HRS] (L) -- (H);
\end{tikzpicture}
\vskip 2em
\begin{tikzpicture}
	[black/.style={circle,draw=black,semithick,fill=black,inner sep=2pt},
	 white/.style={circle,draw=black,semithick,fill=white,inner sep=2pt},
	 dots/.style={draw=none,fill=none},
	 label/.style={rectangle,draw=black,semithick,},
	 HRS/.style={decoration={snake,pre length=2pt,post length=2pt}},
	 every edge quotes/.style={},
	 node distance=14pt and 10pt]
	\node[label] (L) at (0,2) {$\Spec(R)$};
	\node[white] (m) at (0,1) {};

	\node[white] (q-1) [below left=18pt and 7pt of m] {};
	\node[black] (q1) [below right=18pt and 7pt of m] {};
	\node[black] (q-2) [left=14pt of q-1] {};
	\node[black] (q2) [right=14pt of q1] {};
	\node[black] (q-3) [left=28pt of q-2] {};
	\node[black] (q3) [right=28pt of q2] {};
	\path (q-3) edge [draw=none] node {$\;\;\,\cdots$} (q-2);
	\path (q3) edge [draw=none] node {$\cdots\;\;$} (q2);

	\node[black] (p) [below right=18pt and 7pt of q-1] {};

	\draw[semithick] (q-1) -- (m);
	\draw[semithick] (q-2) -- (m);
	\draw[semithick] (q-3) -- (m);
	\draw[semithick] (q1) -- (m);
	\draw[semithick] (q2) -- (m);
	\draw[semithick] (q3) -- (m);

	\draw[semithick] (q-1) -- (p);
	\draw[semithick] (q-2) -- (p);
	\draw[semithick] (q-3) -- (p);
	\draw[semithick] (q1) -- (p);
	\draw[semithick] (q2) -- (p);
	\draw[semithick] (q3) -- (p);

	\node[label] (H) at (8,2) {$\GSpec(\Hcal)$};
	\node[black] (am) at (8,1) {};

	\node[black] (aq-1) [below left=18pt and 7pt of am] {};
	\node[black] (aq1) [below right=18pt and 7pt of am] {};
	\node[black] (aq-2) [left=14pt of aq-1] {};
	\node[black] (aq2) [right=14pt of aq1] {};
	\node[black] (aq-3) [left=28pt of aq-2] {};
	\node[black] (aq3) [right=28pt of aq2] {};
	\path (aq-3) edge [draw=none] node {$\;\;\cdots$} (aq-2);
	\path (aq3) edge [draw=none] node {$\cdots\,\;$} (aq2);

	\node[black] (ap) [below right=18pt and 7pt of aq-1] {};

	\draw[semithick] (aq-1) -- (am);

	\draw[semithick] (aq-2) -- (ap);
	\draw[semithick] (aq-3) -- (ap);
	\draw[semithick] (aq1) -- (ap);
	\draw[semithick] (aq2) -- (ap);
	\draw[semithick] (aq3) -- (ap);

	\draw[semithick] (ap) -- (am);

	\draw [-{Computer Modern Rightarrow[length=3pt]},decorate,HRS] (L) -- (H);
\end{tikzpicture}
	\caption{Here we illustrate how the resulting Gabriel topology differs when
		right tilting at the Zariski closure of a height one ideal depending on
		whether the complement is coherent or not in
		Example~\ref{example:2dim-coherent-noncoherent} and
		Example~\ref{example:Nagata-ring}.}
\end{figure}
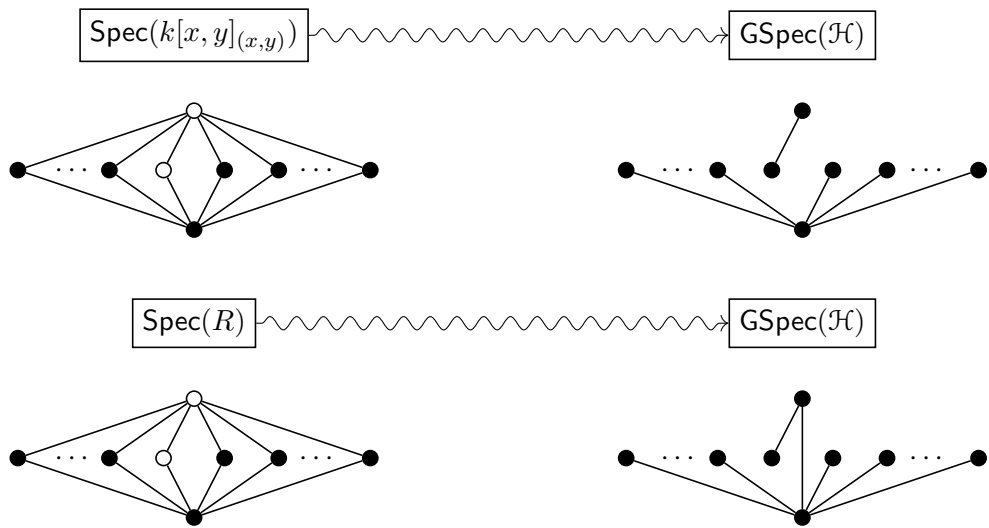

\subsection{Intermediate truncated-slice filtrations}\label{subsec:zerodim-examples}

We conclude with the computation of the Gabriel spectra for a specific type of
bounded cosilting complexes in $\D(R)$. An intermediate filtration,
which up to shift has the form
$\Spec(R)=:V_{-1}\supsetneq V_0\supseteq \cdots\supseteq V_{n-1}\supsetneq
V_n:=\emptyset$,
is said to be \emph{truncated-slice} if for every $0\leq i<n$ the differences $V_{i-1}\setminus
V_{i}$ are discrete subspaces of $\Spec(R)$ with respect to the Hochster
topology, \emph{i.e.}\ they do not contain proper inclusions of primes. The adjective `truncated'
refers to the fact that we do not require any condition on the last nontrivial
difference $V_{n-1}\setminus V_n=V_{n-1}$. These filtrations therefore slightly
generalise the intermediate \emph{slice filtrations} of
\cite[\S3]{hrbe-naka-stov-24}, which are those for which $V_{n-1}$ does not
contain proper inclusions either.

Consider a truncated-slice filtration, as above, and the associated heart
$\Hcal$. The results of \S\ref{subsec:preliminary-computations} are enough to give an explicit
description of the Gabriel tropology of $\Hcal$.

\begin{prop}\label{prop:truncated-slice}
	Let $\tau$ be the $t$-structure associated to an intermediate truncated-slice
	filtration $\Spec(R)=V_{-1}\supsetneq V_0\supseteq\cdots\supseteq V_{n-1}\supsetneq
	V_n:=\emptyset$, with heart $\Hcal$. Then:
	\begin{enumerate}
		\item $V_{n-1}^\cpl$ is a discrete clopen in $\GSpec(\Hcal)$, 
			and the open subsets of $V_{n-1}$ are the specialisation-closed subsets.
		\item $\tau$ can be obtained from the standard $t$-structure by
			a chain of perfect right mutations at discrete closed subspaces.
	\end{enumerate}
	If $\tau$ is associated to a slice filtration, $\GSpec(\Hcal)$
	is discrete, and so $\Hcal$ is semi-artinian.
\end{prop}

\begin{proof}
	As in the beginning of the section, we denote by $\tau_i$, for
	$0\leq i\leq n$, the $t$-structures associated to the truncations of the
	sp-filtration of $\tau$, which are all truncated-slice filtrations. As
	mentioned, $\tau_{i+1}$ is obtained from $\tau_i$ by a right mutation, whose
	underlying closed set is $V_i^\cpl$.

	(1) We know that $V_{n-1}^\cpl$ is closed by
	Proposition~\ref{prop:smashing-cohomological}. Let $\pf$ be a point in
	$V_{n-1}^\cpl$, which must then belong to a difference $V_{i-1}\setminus
	V_{i}$ for some $0\leq
	i<n$. Since the filtration is truncated-slice, $\pf$ is maximal in
	$V_i\setminus V_{i-1}$; hence, it is an open point in $\GSpec(\Hcal)$ by
	Lemma~\ref{lemma:maximal-difference}. This shows that $V_{n-1}^\cpl$, which is
	closed by Lemma~\ref{lemma:preceq-finer}, is in fact a
	discrete clopen subset of $\GSpec(\Hcal)$. The open subsets of the clopen
	$V_{n-1}=V_{n-1}\setminus V_n$ are the specialisation-closed subsets by
	Lemma~\ref{lemma:bound-topology}.

	(2) We argue by induction on $n\geq 0$. If $n=0$, $\tau$ is the standard
	$t$-structure, and there is nothing to prove. For $n>0$, consider
	$\tau_{n-1}$, which is itself truncated-slice. As mentioned, $\tau_n$ is obtained from
	$\tau_{n-1}$ by right mutation at the closed subset $V_{n-1}^\cpl$. Expressing
	this set as 
	$V_{n-1}^\cpl=V_{n-2}^\cpl\sqcup (V_{n-2}\setminus V_{n-1})$, we have that
	$V_{n-2}^\cpl$ is a discrete clopen in $\GSpec(\Hcal_{n-1})$, by item (1). On
	the other hand, the subspace topology on $V_{n-2}\setminus V_{n-1}$ is
	discrete, by Lemma~\ref{lemma:bound-topology} and the fact that this
	difference does not contain proper inclusions of primes by assumption.
	Therefore, $V_{n-1}^\cpl$ is discrete for $\GSpec(\Hcal_{n-1})$. This shows
	that the right mutation from $\tau_{n-1}$ to $\tau_n$ is discrete, and
	therefore also perfect by Proposition~\ref{prop:zerodim-perfect}.
\end{proof}

We conclude with two examples of intermediate slice filtrations.

\begin{ex}\label{ex:height}
	The \emph{height filtration} is the filtration of $\Spec(R)$ given by the
	rule:
	\[V_n = \{\pf \in \Spec(R) \mid \height(\pf) > n\} \quad\text{for all }n \in \Zbb.\]
	The height filtration is an intermediate slice filtration for any commutative
	noetherian ring $R$ of finite Krull dimension \cite[\S 3]{hrbe-naka-stov-24}.
	In view of Lemma~\ref{prop:truncated-slice}, we then obtain that there is a
	finite chain of discrete perfect right mutations connecting the
	standard heart in $\D(R)$, for $R$ of finite Krull dimension, to a semi-artinian heart.
\end{ex}

\begin{ex}\label{ex:CM-excellent}
	A \emph{codimension function} is a function $d: \Spec(R) \to \Zbb$ satisfying
	$d(\qf) = d(\pf) + 1$ for any minimal proper inclusion $\pf \subsetneq \qf$ in
	$\Spec(R)$. A codimension function does not always exist --- if $R$ is local,
	its existence is equivalent to $\Spec(R)$ being a catenary space. Also, it may
	happen that the codimension function for $R$ exists, but the height function
	$\pf \mapsto \height(\pf)$ is not a codimension function
	\cite[Ex.~2.17(2)]{hrbe-naka-stov-24}. If a codimension function exists, it is
	uniquely determined on each connected component of $\Spec(R)$ up to an
	additive constant. Assume that $R$ is of finite Krull dimension. Given a
	codimension function $d$, the filtration given by the rule
	\[V_n = \{\pf \in \Spec(R) \mid d(\pf) > n\}\quad\text{for all }n\in\Zbb\]
	is an intermediate slice filtration
	\cite[\S3]{hrbe-naka-stov-24}. The corresponding heart $\Hcal$ does not depend
	on the choice of the codimension function \cite[Rmk.~4.10]{hrbe-naka-stov-24}
	and is called the \emph{Cohen--Macaulay heart}.
	Takahashi showed in
	\cite{taka-23} that if $R$ is \emph{CM-excellent} then the $t$-structure
	induced this way from any codimension function is \emph{restrictable}, and thus of
	derived type by \cite[Thm.~6.16]{pavo-vito-21}. In fact, the converse was
	proved in \cite[Cor.~5.6]{hrbe-mart-24} --- if a codimension function gives
	rise to a $t$-structure of derived type then $R$ is CM-excellent. We remark
	that for a local ring being CM-excellent is equivalent to it being a
	homomorphic image of a local Cohen--Macaulay ring \cite{kawa-02,kawa-08}, and
	in this case the codimension function always exists, so that $\Mod(R)$ is
	derived equivalent to the semi-artinian Cohen--Macaulay heart.
\end{ex}

\bibliographystyle{plain}
\bibliography{references.bib}

\begin{thebibliography}{10}

\bibitem{alon-jere-saor-10}
Leovigildo Alonso~Tarr{\'{\i}}o, Ana Jerem{\'{\i}}as~L{\'{o}}pez, and Manuel
  Saor{\'{\i}}n.
\newblock Compactly generated {$t$}-structures on the derived category of a
  {N}oetherian ring.
\newblock {\em J. Algebra}, 324(3):313--346, 2010.

\bibitem{ange-laki-stov-vito-22}
Lidia Angeleri~H{\"{u}}gel, Rosanna Laking, Jan
  {\v{S}}{\v{t}}ov{\'{\i}}{\v{c}}ek, and Jorge Vit{\'{o}}ria.
\newblock Mutation and torsion pairs.
\newblock 2022.

\bibitem{ange-mark-stov-taka-vito-20}
Lidia Angeleri~H{\"{u}}gel, Frederik Marks, Jan
  {\v{S}}{\v{t}}ov{\'{\i}}{\v{c}}ek, Ryo Takahashi, and Jorge Vit{\'{o}}ria.
\newblock Flat ring epimorphisms and universal localizations of commutative
  rings.
\newblock {\em Q. J. Math.}, 71(4):1489--1520, 2020.

\bibitem{ange-mark-vito-17}
Lidia Angeleri~H{\"{u}}gel, Frederik Marks, and Jorge Vit{\'{o}}ria.
\newblock Torsion pairs in silting theory.
\newblock {\em Pacific J. Math.}, 291(2):257--278, 2017.

\bibitem{beli-reit-07}
Apostolos Beligiannis and Idun Reiten.
\newblock Homological and homotopical aspects of torsion theories.
\newblock {\em Mem. Amer. Math. Soc.}, 188(883):viii+207, 2007.

\bibitem{bbd-81}
A.~A. Be\u{\i}linson, J.~Bernstein, and P.~Deligne.
\newblock Faisceaux pervers.
\newblock In {\em Analysis and topology on singular spaces, {I} ({L}uminy,
  1981)}, volume 100 of {\em Ast\'{e}risque}, pages 5--171. Soc. Math. France,
  Paris, 1982.

\bibitem{bond-orlo-01}
Alexey~I. Bondal and Dmitri Orlov.
\newblock Reconstruction of a variety from the derived category and groups of
  autoequivalences.
\newblock {\em Compositio Math.}, 125(3):327--344, 2001.

\bibitem{brea-hrbe-modo-22}
Simion Breaz, Michal Hrbek, and George~Ciprian Modoi.
\newblock Silting, cosilting, and extensions of commutative ring.
\newblock {\em arXiv preprint arXiv:2204.01374}, 2022.

\bibitem{burk-94-thesis}
Kevin Burke.
\newblock {\em Some model-theoretic properties of functor categories for
  modules}.
\newblock PhD thesis, University of Manchester, 1994.

\bibitem{chri-holm-09}
Lars~Winther Christensen and Henrik Holm.
\newblock Ascent properties of auslander categories.
\newblock {\em Canadian Journal of Mathematics}, 61(1):76--108, 2009.

\bibitem{craw-94}
William Crawley-Boevey.
\newblock Locally finitely presented additive categories.
\newblock {\em Comm. Algebra}, 22(5):1641--1674, 1994.

\bibitem{gabr-62}
Pierre Gabriel.
\newblock Des cat\'{e}gories ab\'{e}liennes.
\newblock {\em Bull. Soc. Math. France}, 90:323--448, 1962.

\bibitem{happ-reit-smal-96}
Dieter Happel, Idun Reiten, and Sverre~O. Smal\o.
\newblock Tilting in abelian categories and quasitilted algebras.
\newblock {\em Mem. Amer. Math. Soc.}, 120(575):viii+ 88, 1996.

\bibitem{herz-97}
Ivo Herzog.
\newblock The {Z}iegler spectrum of a locally coherent {G}rothendieck category.
\newblock {\em Proc. London Math. Soc. (3)}, 74(3):503--558, 1997.

\bibitem{hrbe-hu-zhu-24}
Michal Hrbek, Jiangsheng Hu, and Rongmin Zhu.
\newblock Gluing compactly generated t-structures over stalks of affine
  schemes.
\newblock {\em Israel Journal of Mathematics}, pages 1--41, 2024.

\bibitem{hrbe-mart-24}
Michal Hrbek and Lorenzo Martini.
\newblock Product-complete tilting complexes and cohen--macaulay hearts.
\newblock {\em Revista Matem{\'a}tica Iberoamericana}, 40(6):2339--2369, 2024.

\bibitem{hrbe-naka-21}
Michal Hrbek and Tsutomu Nakamura.
\newblock Telescope conjecture for homotopically smashing t-structures over
  commutative noetherian rings.
\newblock {\em J. Pure Appl. Algebra}, 225(4):106571, 13, 2021.

\bibitem{hrbe-naka-stov-24}
Michal Hrbek, Tsutomu Nakamura, and Jan {\v{S}}{\v{t}}ov{\'{\i}}{\v{c}}ek.
\newblock Tilting complexes and codimension functions over commutative
  {N}oetherian rings.
\newblock {\em Nagoya Math. J.}, 255:618--693, 2024.

\bibitem{hrbe-pavo-21}
Michal Hrbek and Sergio Pavon.
\newblock Singular equivalences to locally coherent hearts of commutative
  noetherian rings.
\newblock {\em J. Algebra}, 632:117--153, 2023.

\bibitem{jens-lenz-89}
Christian~U. Jensen and Helmut Lenzing.
\newblock {\em Model-theoretic algebra with particular emphasis on fields,
  rings, modules}, volume~2 of {\em Algebra, Logic and Applications}.
\newblock Gordon and Breach Science Publishers, New York, 1989.

\bibitem{kawa-02}
Takesi Kawasaki.
\newblock On arithmetic macaulayfication of noetherian rings.
\newblock {\em Transactions of the American Mathematical Society},
  354(1):123--149, 2002.

\bibitem{kawa-08}
Takesi Kawasaki.
\newblock Finiteness of cousin cohomologies.
\newblock {\em Transactions of the American Mathematical Society},
  360(5):2709--2739, 2008.

\bibitem{krau-97}
Henning Krause.
\newblock The spectrum of a locally coherent category.
\newblock {\em J. Pure Appl. Algebra}, 114(3):259--271, 1997.

\bibitem{krau-98b}
Henning Krause.
\newblock Exactly definable categories.
\newblock {\em J. Algebra}, 201(2):456--492, 1998.

\bibitem{krau-00}
Henning Krause.
\newblock Smashing subcategories and the telescope conjecture---an algebraic
  approach.
\newblock {\em Invent. Math.}, 139(1):99--133, 2000.

\bibitem{krau-08}
Henning Krause.
\newblock Thick subcategories of modules over commutative {N}oetherian rings
  (with an appendix by {S}rikanth {I}yengar).
\newblock {\em Math. Ann.}, 340(4):733--747, 2008.

\bibitem{laki-20}
Rosanna Laking.
\newblock Purity in compactly generated derivators and t-structures with
  {G}rothendieck hearts.
\newblock {\em Math. Z.}, 295(3-4):1615--1641, 2020.

\bibitem{mark-vito-18}
Frederik Marks and Jorge Vit{\'{o}}ria.
\newblock Silting and cosilting classes in derived categories.
\newblock {\em J. Algebra}, 501:526--544, 2018.

\bibitem{matl-58}
Eben Matlis.
\newblock Injective modules over {N}oetherian rings.
\newblock {\em Pacific J. Math.}, 8:511--528, 1958.

\bibitem{naga-58}
Masayoshi Nagata.
\newblock An example of a normal local ring which is analytically reducible.
\newblock {\em Memoirs of the College of Science, University of Kyoto. Series
  A: Mathematics}, 31(1):83--85, 1958.

\bibitem{neem-92a}
Amnon Neeman.
\newblock The chromatic tower for {$D(R)$}. {W}ith an appendix by {M}arcel
  {B}\"{o}kstedt.
\newblock {\em Topology}, 31(3):519--532, 1992.

\bibitem{neem-01}
Amnon Neeman.
\newblock {\em Triangulated categories}, volume 148 of {\em Annals of
  Mathematics Studies}.
\newblock Princeton University Press, Princeton, NJ, 2001.

\bibitem{neem-21}
Amnon Neeman.
\newblock The {$t$}-structures generated by objects.
\newblock {\em Trans. Amer. Math. Soc.}, 374(11):8161--8175, 2021.

\bibitem{nico-saor-zvon-19}
Pedro Nicol{\'{a}}s, Manuel Saor{\'{\i}}n, and Alexandra Zvonareva.
\newblock Silting theory in triangulated categories with coproducts.
\newblock {\em J. Pure Appl. Algebra}, 223(6):2273--2319, 2019.

\bibitem{parr-saor-15}
Carlos~E. Parra and Manuel Saor{\'{\i}}n.
\newblock Direct limits in the heart of a t-structure: the case of a torsion
  pair.
\newblock {\em J. Pure Appl. Algebra}, 219(9):4117--4143, 2015.

\bibitem{pavo-25}
Sergio Pavon.
\newblock Torsion-simple objects in abelian categories.
\newblock {\em J. Pure Appl. Algebra}, 229(1):Paper No. 107818, 2025.

\bibitem{pavo-vito-21}
Sergio Pavon and Jorge Vit{\'{o}}ria.
\newblock Hearts for commutative {N}oetherian rings: torsion pairs and derived
  equivalences.
\newblock {\em Doc. Math.}, 26:829--871, 2021.

\bibitem{pope-73}
Nicolae Popescu.
\newblock {\em Abelian categories with applications to rings and modules}.
\newblock London Mathematical Society Monographs, No. 3. Academic Press,
  London-New York, 1973.

\bibitem{posi-stov-21}
Leonid Positselski and Jan {\v{S}}{\v{t}}ov{\'{\i}}{\v{c}}ek.
\newblock The tilting-cotilting correspondence.
\newblock {\em Int. Math. Res. Not. IMRN}, (1):191--276, 2021.

\bibitem{pres-09}
Mike Prest.
\newblock {\em Purity, spectra and localisation}, volume 121 of {\em
  Encyclopedia of Mathematics and its Applications}.
\newblock Cambridge University Press, Cambridge, 2009.

\bibitem{psar-vito-18}
Chrysostomos Psaroudakis and Jorge Vit{\'{o}}ria.
\newblock Realisation functors in tilting theory.
\newblock {\em Math. Z.}, 288(3-4):965--1028, 2018.

\bibitem{ried-scho-91}
Christine Riedtmann and Aidan Schofield.
\newblock On a simplicial complex associated with tilting modules.
\newblock {\em Comment. Math. Helv.}, 66(1):70--78, 1991.

\bibitem{saor-stov-viri-17}
Manuel Saor{\'{\i}}n, Jan {\v{S}}{\v{t}}ov{\'{\i}}{\v{c}}ek, and Simone Virili.
\newblock $t$-structures on stable derivators and {G}rothendieck hearts.
\newblock {\em Advances in Mathematics}, 429:109139, 2023.

\bibitem{stacks}
The {Stacks project authors}.
\newblock The stacks project.
\newblock {\tt https://stacks.math.columbia.edu}, 2023.

\bibitem{sten-75}
Bo~Stenstr{\"{o}}m.
\newblock {\em Rings of quotients}.
\newblock Die Grundlehren der mathematischen Wissenschaften, Band 217.
  Springer-Verlag, New York-Heidelberg, 1975.

\bibitem{stov-24}
Jan {\v{S}}{\v{t}}ov{\'{\i}}{\v{c}}ek.
\newblock Flat epimorphisms and silting epimorphisms coincide for commutative
  rings.
\newblock {\em In preparation}, 2025.

\bibitem{taka-23}
Ryo Takahashi.
\newblock Faltings’ annihilator theorem and t-structures of derived
  categories.
\newblock {\em Mathematische Zeitschrift}, 304(1):10, 2023.

\end{thebibliography}

\end{document}